\documentclass[a4paper,12pt]{amsart}
\usepackage{amssymb,amsmath,amsfonts,amsthm}
\usepackage[margin=1in]{geometry}

\newcommand{\ds}{\displaystyle}

\newcommand{\NN}{\mathbb N}

\newcommand{\CC}{\mathbb C}
\newcommand{\RR}{\mathbb R}
\newcommand{\ZZ}{\mathbb Z}
\newcommand{\EE}{\mathcal E}
\newcommand{\DD}{\mathcal D}
\newcommand{\SSS}{\mathcal S}

\theoremstyle{plain}
\newtheorem{theorem}{Theorem}[section]
\newtheorem{proposition}[theorem]{Proposition}
\newtheorem{lemma}[theorem]{Lemma}
\newtheorem{corollary}[theorem]{Corollary}

\theoremstyle{remark}
\newtheorem{remark}[theorem]{Remark}

\theoremstyle{definition}
\newtheorem{definition}[theorem]{Definition}
\newtheorem{example}[theorem]{Example}
\allowdisplaybreaks

\numberwithin{equation}{section}

\newcommand{\beq}{\begin{eqnarray}}
\newcommand{\eeq}{\end{eqnarray}}

\newcommand{\beqs}{\begin{eqnarray*}}
\newcommand{\eeqs}{\end{eqnarray*}}

\begin{document}

\title[Quasianalytic classes -- Parametrix and convolution]{On quasianalytic classes of Gelfand-Shilov type. Parametrix and convolution}

\author[S. Pilipovi\'{c}]{Stevan Pilipovi\'{c}}
\address{S. Pilipovi\'{c}, Department of Mathematics and Informatics, University of Novi Sad, Trg Dositeja Obradovi\'ca 4, 21000 Novi Sad, Serbia}
\email{stevan.pilipovic@dmi.uns.ac.rs}
\thanks{S. Pilipovi\'{c} is supported by the Serbian Ministry of Education, Science and Technological Development, through the project 174024}

\author[B. Prangoski]{Bojan Prangoski}
\address{B. Prangoski, Faculty of Mechanical Engineering, University Ss. Cyril and Methodius, Karpos II bb, 1000 Skopje, Macedonia}
\email{bprangoski@yahoo.com}

\author[J. Vindas]{Jasson Vindas}
\address{J. Vindas, Department of Mathematics, Ghent University, Krijgslaan 281 Gebouw S22, 9000 Gent, Belgium}
\email{jasson.vindas@ugent.be}
\thanks{J. Vindas gratefully acknowledges support by Ghent University, through the BOF-grant 01N01014}


\begin{abstract}
We develop a convolution theory for quasianalytic ultradistributions of Gelfand-Shilov type. We also construct a special class of ultrapolynomials, and use it as a base for the parametrix method in the study of new topological and structural properties of several quasianalytic spaces of functions and ultradistributions. In particular, our results apply to Fourier hyperfunctions and Fourier ultra-hyperfunctions.
\end{abstract}

\maketitle

\section*{Introduction}

Convolution is among the most important operations in mathematical analysis. In the case of  distributions, this is a classical and much studied topic within  Schwartz' theory \cite{hor,SchwartzS,SchwartzV,SchwartzT,Shiraishi} (see also the relevant references
\cite{DD,dirvo,K, SI}). Although many central problems concerning the convolution of Schwartz distributions were solved long time ago, one may still find many interesting results in the recent literature \cite{BO,DPV,ort,Wagner}. In the case of non-quasianalytic ultradistributions, the existence and characterizations of the convolution of  Beurling ultradistributions was established  in \cite{PK,PilipovicC}. The long-standing problem of extending such theory to Roumieu ultradistributions, or even to classical Gelfand-Shilov spaces \cite{NR}, of non-quasianalytic type was solved only until recently in our paper \cite{PB} and substantially improved in \cite{DPPV}. A key ingredient for the improvements in the latter paper is the use of the so-called parametrix method, based there on parametrices for ultradifferential operators constructed by Komatsu in \cite{Komatsu4}.

In this article we develop a convolution theory for quasianalytic ultradistributions. The problems that we consider in this paper significantly differ from the non-quasianalytic case and require the development of new methods for their analysis. In fact, we shall establish here a number of new topological and structural properties for spaces of quasianalytic functions and ultradistributions that, to the best of our knowledge, have been lacking in the literature for quite long time. An important technical tool in this work is the construction of suitable ultrapolynomials with respect to a class of quasianalytic sequences and their use as a base for the parametrix method.

 Our interest lies in the spaces $\SSS^{(M_p)}_{(A_p)}(\RR^d)$ (Beurling case) and $\SSS^{\{M_p\}}_{\{A_p\}}(\RR^d)$ (Roumieu case) of  Gelfand-Shilov mixed type and their duals, where the weight sequences satisfy suitable conditions (see Section \ref{preli}) but may be quasianalytic, namely, they do not necessarily satisfy the so-called $(M.3')$ condition \cite{Komatsu1}. We employ the common notation $\mathcal{S}^{\ast}_{\dagger}(\mathbb{R}^{d})$ to treat the Beurling and Roumieu cases simultaneously. We remark that all of our results cover the important cases of Fourier hyperfunctions and ultra-hyperfunctions as particular instances; in fact, they correspond to the symmetric choice $A_p=M_p=p!$ (see \cite{morimoto69,morimoto,zharinov} for hyperfunctions and ultrahyperfunctions; see \cite{dm-1,dm-2} for ultradistributions on homogeneous spaces and compact manifolds). However, it should also be emphasized that we go beyond (ultra)-hyperfunctions, as the weight sequence $M_{p}$ may satisfy $M_p\subset p!$.

The paper is organized as follows. In Section \ref{preli} we fix the assumptions on the weight sequences and discuss some preliminary notions. We set the ground for the parametrix method in Section \ref{structure and topology}, where we also deduce various topological and structural properties of $\mathcal{S}^{\ast}_{\dagger}(\mathbb{R}^{d})$  and their duals $\mathcal{S}'^{\ast}_{\dagger}(\mathbb{R}^{d})$. Indeed, we provide a structural characterization for $\mathcal{S}'^{\ast}_{\dagger}(\mathbb{R}^{d})$ and establish the nuclearity of these spaces. As a preparation for our study of the general convolution of quasianalytic ultradistributions, we collect and explain in Section \ref{TIBU} a number of results concerning the class of ultradistribution spaces associated to translation-invariant Banach spaces recently introduced and studied by us and Dimovski in \cite{DPV,DPPV,dpv-2}. The main concern in the rest of the paper is the convolution. Naturally, a good understanding of the topological properties of the spaces $\dot{\mathcal{B}}^{\ast}_{\eta}$, $\mathcal{D}^{\ast}_{L^{\infty}_{\eta}}$, and $\mathcal{D}'^{\ast}_{L^{1}_{\eta}}$ (ultradistribution weighted extensions of the corresponding Schwartz spaces \cite{SchwartzT,SchwartzS}) should play an essential role in such study. In Section \ref{tensor} we study the $\varepsilon$ tensor product of $\dot{\mathcal{B}}^{\ast}_{\eta}$ with a sequentially complete, quasi-complete, or complete l.c.s., while in Section \ref{section convolution} we show that $(\DD^*_{L^{\infty}_{\eta},c})'_b$ and $\DD'^*_{L^1_{\eta}}$ are isomorphic as l.c.s.; such investigations are new and of independent interest. All these topological results are then used for the analysis and characterization of the convolution in Theorem \ref{convolution}.

\section{Preliminaries}\label{preli}

Let $\{M_p\}_{p\in\mathbb{N}}$ and $\{A_p\}_{p\in\mathbb{N}}$ be  sequences of positive numbers, $M_0=M_1=A_{0}=A_{1}=1.$ Throughout the article, we impose the following assumptions over these weight sequences. The sequence $M_p$ satisfies:

\indent $(M.1)$ $M_{p}^{2} \leq M_{p-1} M_{p+1}, \; \; p \in\ZZ_+$;\\
\indent $(M.2)$ $\ds M_{p} \leq c_0H^{p} \min_{0\leq q\leq p} \{M_{p-q} M_{q}\}$, $p,q\in \NN$, for some $c_0,H\geq1$;\\
\indent $(M.5)$ there exists $q>0$ such that $M_p^q$ is strongly
non-quasianalytic, i.e., there exists $c_0\geq 1$ such that $
\sum_{j=p+1}^{\infty}{M_{j-1}^q}/{M_j^q}\leq c_0
p{M_p^q}/{M_{p+1}^q},\,\, \forall p\in\ZZ_+. $

We assume that $A_{p}$ also satisfies $(M.1)$ and $(M.2)$. Without any loss of generality, we can assume that the constants $c_0$ and $H$ from  $(M.2)$ are the same for $M_p$ and $A_p$. Moreover, we  assume that $A_p$ satisfies:\\
\indent $(M.6)$ $p!\subset A_p$; i.e., there exist $c_0,L_0\geq1$
such that $p!\leq c_0 L_0^p A_p$, $p\in\NN$.\\
Note that $(M.5)$
implies that there exists $\kappa>0$ such that $p!^{\kappa}\subset
M_p$, i.e., there exist $c_0,L_0>0$ such that $p!^{\kappa}\leq c_0
L_0^p M_p$, $p\in\NN$ (cf. \cite[Lemma 4.1]{Komatsu1}). Naturally,
by enlarging it if necessary, we may assume that $q\in\ZZ_+, q\geq
2,$ in $(M.5)$. Following \cite{Komatsu1}, for $p\in\ZZ_+$, we
denote $m_p=M_p/M_{p-1}$ and for $\rho\geq 0$ let $m(\rho)$ be the
number of indices $p$ with $m_p\leq \rho$. As a consequence of
\cite[Proposition 4.4]{Komatsu1},  a change of variables shows that
$(M.5)$ holds if and only if
\beqs
\int_{\rho}^{\infty}\frac{m(\lambda)}{\lambda^{q+1}}d\lambda \leq
c \frac{m(\rho)}{\rho^q},\,\, \forall \rho\geq m_1.
\eeqs
A sufficient condition for $M_p$ to satisfy $(M.5)$ is obtained if
the sequence $m_p/p^{\lambda}$ is eventually quasi-increasing for
some $\lambda>0$, namely,  there is $c$ such that $m_{p+1}/m_p\geq
c (1+1/p)^{\lambda}$ for $p\geq 1$. Although it is not part of our
assumptions, we are primary interested in the quasianalytic case,
i.e., $\sum_{p=1}^{\infty}{M_{p-1}}/{M_p}=\infty$.

We denote by $M(\cdot)$ and $A(\cdot)$ the associated functions \cite{Komatsu1} of $M_p$ and $A_p$, respectively, and by $\tilde{M}(\cdot)$ the associated function of $\tilde{M}_p:=M_p^{q}$. Clearly, $\tilde{M}(\rho^q)=qM(\rho)$, $\rho>0$. Denote by $\mathfrak{R}$ the set of all positive sequences which increase to $\infty.$ For $(l_p)\in\mathfrak{R}$, denote as $N_{l_p}(\cdot)$ and $B_{l_p}(\cdot)$ the associated functions for the sequences $M_p\prod_{j=1}^pl_j$ and $A_p\prod_{j=1}^pl_j$, respectively.  In the sequel we will often use the following technical result.

\begin{lemma}[\cite{BojanL}]\label{nwseq}
Let $(k_p)\in\mathfrak{R}$. There exists $(k'_p)\in\mathfrak{R}$ such that $k'_p\leq k_p$ and
\beqs
\prod_{j=1}^{p+q}k'_j\leq 2^{p+q}\prod_{j=1}^{p}k'_j\cdot\prod_{j=1}^{q}k'_j,\,\, \mbox{for all}\,\, p,q\in\ZZ_+.
\eeqs
\end{lemma}

\indent We denote by $\SSS^{M_p,h}_{A_p,h}, h>0,$ the $(B)$-space (Banach space) of all $\varphi\in C^{\infty}(\RR^d)$ for which the norm
\beqs
\sigma_h(\varphi)=\sup_{\alpha} \frac{h^{|\alpha|}\left\|e^{A(h|\cdot|)}D^{\alpha}\varphi\right\|_{L^{\infty}(\RR^d)}}{M_{\alpha}}
\eeqs
is finite. One easily verifies that for $h_1H<h_2$ the canonical inclusion $\SSS^{M_p,h_2}_{A_p,h_2}\rightarrow\SSS^{M_p,h_1}_{A_p,h_1}$ is compact. As l.c.s. we define $\ds\SSS^{(M_p)}_{(A_p)}(\RR^d)=\lim_{\substack{\longleftarrow\\ h\rightarrow \infty}} \SSS^{M_p,h}_{A_p,h}$ and $\ds\SSS^{\{M_p\}}_{\{A_p\}}(\RR^d)=\lim_{\substack{\longrightarrow\\ h\rightarrow 0}} \SSS^{M_p,h}_{A_p,h}$. Thus $\SSS^{(M_p)}_{(A_p)}(\RR^d)$ is an $(FS)$-space and $\SSS^{\{M_p\}}_{\{A_p\}}(\RR^d)$ is a $(DFS)$-space. In particular they are both Montel spaces. In the sequel we shall employ $\SSS^*_{\dagger}(\RR^d)$ as a common notation for $\SSS^{(M_p)}_{(A_p)}(\RR^d)$ (Beurling case) and $\SSS^{\{M_p\}}_{\{A_p\}}(\RR^d)$ (Roumieu case).\\
\indent For each $(r_p)\in\mathfrak{R}$, by $\SSS^{M_p,(r_p)}_{A_p,(r_p)}$ we denote the space of all $\varphi\in C^{\infty}(\RR^d)$ such that
\beqs
\sigma_{(r_p)}(\varphi)=\sup_{\alpha} \frac{\left\|e^{B_{r_p}(|\cdot|)}D^{\alpha}\varphi\right\|_{L^{\infty}(\RR^d)}}{M_{\alpha}\prod_{j=1}^{|\alpha|}r_j}<\infty.
\eeqs
Provided with the norm $\sigma_{(r_p)}$, the space $\SSS^{M_p,(r_p)}_{A_p,(r_p)}$ becomes a $(B)$-space. As in \cite{PilipovicK}, one can prove that $\SSS^{\{M_p\}}_{\{A_p\}}(\RR^d)$ is topologically isomorphic to $\ds \lim_{\substack{\longleftarrow \\ (r_p)\in\mathfrak{R}}}\SSS^{M_p,(r_p)}_{A_p,(r_p)}$.

We will often make use of weighted $L^p$ spaces. We fix the notation. Let $\eta:\mathbb{R}^d\rightarrow (0,\infty)$ be a measurable function.
For $1\leq p<\infty$ we denote as $L^p_{\eta}$ the spaces of measurable functions $g$ such that $\|g\|_{p,\eta}:=\|\eta g\|_{p}<\infty$. On the other hand, we make an exception and define $L^\infty_{\eta}$ via the norm $\|g\|_{\infty,\eta}:=\| g/\eta\|_{\infty}$. As usual, if $\eta=1$, we drop it from the notation.

Following Komatsu \cite{Komatsu3} we will use the ensuing
terminology. For a subset $A$ of a locally convex space $F$ (from
now on always abbreviated as l.c.s.), we say that $f\in F$ is a
sequential limit point (resp. a bounding limit point) of $A$ if
there exists a sequence in $A$ (resp. a bounded net in $A$) which
converges to $f$. We say that $A$ is sequentially closed (resp.
boundedly closed) if it coincides with the set of all sequential
limit points (resp. the set of all bounding limit points) of $A$.
Since the intersection of sequentially closed sets (resp.
boundedly closed sets) is sequentially closed (resp. boundedly
closed) there exists smallest sequentially closed set (resp.
smallest boundedly closed set) which contains $A$. We will call
this set the
sequential closure (resp. the bounding closure) of $A$ in $F$.

As usual, a subset $A$ of $F$ will be called sequentially complete
if every Cauchy sequence in $A$ is convergent in $A$. Moreover,
$A$ is said to be quasi-complete (Komatsu \cite{Komatsu3} uses the
term boundedly complete) if every bounded Cauchy net in $A$ is
convergent in $A$. The sequential completion of $A\subseteq F$
(resp. quasi-completion) is the sequential closure (resp. the
bounding closure) of $A$ in the completion of $F$.

For $X$ and $Y$ two l.c.s., we denote by
$\mathcal{L}(X,Y)$ the space of all continuous linear mappings
from $X$ to $Y$. When we want to indicate a specific topology on
this space we will use the following indices: $b$ for the strong
topology, $\sigma$ for the weak topology, $p$ for the topology of
precompact convergence, $c$ for the topology of compact convex
circled convergence. If $X$ is quasi-complete then clearly
$\mathcal{L}_p(X,Y)$ is isomorphic to $\mathcal{L}_c(X,Y)$. We use
the same indices as above for denoting the corresponding
topologies on the dual $X'$ of the l.c.s. $X$. If $t$ is any of
these indices we denote by $\mathcal{L}_{\epsilon}(X'_t,Y)$ the
space $\mathcal{L}(X'_t,Y)$ equipped with the topology of
uniform convergence on all equicontinuous subsets of $X'$.\\
\indent Following Schwartz \cite{SchwartzV} and Komatsu
\cite{Komatsu3} for two l.c.s. $X$ and $Y$ we denote by
$X\varepsilon Y$ (the $\varepsilon$ product of $X$ and $Y$) the space
of all bilinear functionals on $X'_c\times Y'_c$ which are
hypocontinuous with the respect to the equicontinuous subsets of
$X'$ and $Y'$ and equipped with the topology of uniform
convergence on the products of equicontinuous subsets of $X'$ and
$Y'$. The tensor product $X\otimes Y$ is canonically embedded into
$X\varepsilon Y$ via $(x\otimes y)(x',y')=\langle x',x\rangle
\langle y',y\rangle$. Clearly, the topology induced by
$X\varepsilon Y$ on $X\otimes Y$ is the $\epsilon$ topology. As
pointed in \cite[p. 657]{Komatsu3}, we have the following canonical
isomorphisms of l.c.s.
\beqs
X\varepsilon Y\cong\mathcal{L}_{\epsilon}(X'_c,Y)\cong\mathcal{L}_{\epsilon}(Y'_c,X).
\eeqs
Grothendieck \cite{Grothendieck} first
introduced the approximation property for a l.c.s. $X$:\\
\indent  $X$ is said to have the
approximation property if the identity mapping
$\mathrm{Id}\in\mathcal{L}(X,X)$ is in the closure of $X'\otimes
X$ in $\mathcal{L}_p(X,X)$.\\
Later Schwartz \cite{SchwartzV} introduced the following weaker
version:\\
\indent $X$ is said to have the weak approximation property if
$\mathrm{Id}$ is in the closure of $X'\otimes X$ in
$\mathcal{L}_c(X,X)$.\\
If $X$ is quasi-complete the above two definitions are
equivalent. We also need the following notion \cite{Komatsu3,SchwartzV}:\\
\indent We say that $X$ satisfies the weak sequential
approximation property (resp. the weak bounded approximation
property) if $\mathrm{Id}$ is in the sequential closure (resp. the
bounding closure) of $X'\otimes X$ in $\mathcal{L}_c(X,X)$. Every nuclear space has the weak approximation
property.\\
\indent Lastly, we mention that there is a similar concept to this: the so called metric approximation property (see \cite[Section 43.8]{kothe2} for the definition) that was recently investigated in the setting of mixed-norm $L^p$ spaces and modulation spaces; see \cite{dmw-1,dmw-2,delm}.

\section{Structural and topological properties of $\SSS'^*_{\dagger}(\RR^d)$}
\label{structure and topology}

\subsection{Parametrix} The construction of ultradifferential operators for the parametrix method is a basic result for accomplishing the main results of the paper.
\begin{lemma}\label{ultpoly}
Let $r'\geq 1$ and $k>0$ (resp. $(k_p)\in\mathfrak{R}$). There exists an ultrapolynomial $P(z)$ of class $(M_p)$ (resp. of class $\{M_p\}$) such that $P$ does not vanish on $\RR^{d}$ and satisfies the following estimate:\\
\indent There exists $C>0$ such that for all $x\in\RR^d$ and $\alpha\in\NN^d$
\beq
\left|D^{\alpha}\left({1}/{P(x)}\right)\right|&\leq& C\frac{\alpha!}{r'^{|\alpha|}} e^{-M(k|x|)},\,\, \mbox{resp.}\label{bgrowth}\\
\left|D^{\alpha}\left({1}/{P(x)}\right)\right|&\leq& C\frac{\alpha!}{r'^{|\alpha|}} e^{-N_{k_p}(|x|)}.\label{rgrowth}
\eeq
\end{lemma}

\begin{proof} Let $r'\geq 1$ and $k>0$ (resp. $(k_p)\in\mathfrak{R}$). In the $(M_p)$ case set $l=2Hk\sqrt{2d}$. In the $\{M_p\}$ case define $l'_p=k_p/(4H\sqrt{2d})$, $p\in\ZZ_+$. By Lemma \ref{nwseq}, we can select $(l_p)\in\mathfrak{R}$ such that $l_p\leq l'_p$, $p\in\ZZ_+$ and $\prod_{j=1}^{p+q}l_j\leq 2^{p+q}\prod_{j=1}^{p}l_j\cdot\prod_{j=1}^{q}l_j$, for all $p,q\in\ZZ_+$. Let $j_0\in\ZZ_+$, $j_0\geq 2$, be such that $r'l/m_j\leq 2^{-q-5}d^{-q}$ (resp. $r'/(l_jm_j)\leq 2^{-q-5}d^{-q}$), for all $j\geq j_0$, $j\in\ZZ_+$; such $j_0$ exists since $m_j\rightarrow\infty$ and $(l_p)\in\mathfrak{R}$. Consider the function
\beqs
\tilde{P}_l(z)=\prod_{j=j_0}^{\infty}\left(1+\frac{l^{2q}z^2}{\tilde{m}_j^2}\right)\,\, \left(\mbox{resp.}\,\, \tilde{P}_{l_p}(z)=\prod_{j=j_0}^{\infty}\left(1+\frac{z^2}{l_j^{2q}\tilde{m}_j^2}\right)\right), z\in \CC^d,
\eeqs
where $z^2=z_1^2+...+z^2_d$. Clearly $\tilde{P}_l(z)$ (resp. $\tilde{P}_{l_p}(z)$), is an entire function. One easily verifies that the function of one complex variable $\zeta\mapsto \tilde{P}_l(\zeta,0,...,0)$ (resp. $\zeta\mapsto \tilde{P}_{l_p}(\zeta,0,...,0)$) satisfies the condition of \cite[Proposition 4.6]{Komatsu1}. Denote briefly this function by $\tilde{\tilde{P}}_l$ (resp. by $\tilde{\tilde{P}}_{l_p}$). Hence $\tilde{\tilde{P}}_l$ (resp. $\tilde{\tilde{P}}_{l_p}$), satisfies the equivalent conditions of \cite[Proposition 4.5]{Komatsu1}, i.e., there exist $h,C>0$ (resp. for every $h>0$ there exists $C>0$), such that $|\tilde{\tilde{P}}_l(\zeta)|\leq C e^{\tilde{M}(h|\zeta|)}$ for all $\zeta\in\CC$ (resp. $|\tilde{\tilde{P}}_{l_p}(\zeta)|\leq C e^{\tilde{M}(h|\zeta|)}$ for all $\zeta\in\CC$). In the Beurling case, observe that
\beqs
|\tilde{P}_l(z)|\leq \prod_{j=j_0}^{\infty}\left(1+\frac{l^{2q}|z|^2}{\tilde{m}_j^2}\right)=\tilde{\tilde{P}}_l(|z|)\leq C e^{\tilde{M}(h|z|)},\, z\in\CC^d,
\eeqs
for some $h,C>0$. Similarly, in the Roumieu case, one obtains the same inequality with $\tilde{P}_{l_p}$ for each $h$ and the corresponding $C=C(h)$. Thus $\tilde{P}_l$ (resp. $\tilde{P}_{l_p}$) is an ultrapolynomial of $(\tilde{M}_p)$ class (resp. of $\{\tilde{M}_p\}$ class). Define $P_l(z)=\tilde{P}_l(z^q_1,...,z^q_d)$ (resp. $P_{l_p}(z)=\tilde{P}_{l_p}(z^q_1,...,z^q_d)$). Then, if we use that $\tilde{M}(\rho^q)=qM(\rho)$ and \cite[Proposition 3.6]{Komatsu1} for $M_p$, we have that there are $h,C>0$ (resp. for every $h>0$ there exists $C>0$), such that $|P_l(z)|\leq Ce^{M(h|z|)}$ (resp. $|P_{l_p}(z)|\leq Ce^{M(h|z|)}$), for all $z\in\CC^d$. Denote $W=\RR^d_x+i\{y\in\RR^d|\, |y_j|\leq 2r',j=1,...,d\}$. For $z=x+iy\in W$, with $|x|\leq 2^{q+4}d^q r'$ and $j\geq j_0$ we have
\beqs
\frac{l^{2q}\left|z_1^{2q}+...+z_d^{2q}\right|}{\tilde{m}_j^2}\leq \left(\frac{l^{2}(|x|^2+|y|^2)}{m_j^{2}}\right)^q\leq \frac{1}{2^q},
\eeqs
i.e., $P_l$ does not have zeros in $W$ when $|x|\leq 2^{q+4}d^q r'$. Similarly, $P_{l_p}$ doesn't have zeroes when $z\in W$ and $|x|\leq 2^{q+4}d^q r'$. When $z=x+iy\in W$ and $|x|\geq 2^{q+4}d^q r'$, we have $|y|\leq 2r'\sqrt{d}\leq 2^{-q-3} d^{-q+1}|x|$. Thus
\beqs
\mathrm{Re}(z^{2q}_1+...+z^{2q}_d)&=&\mathrm{Re}\sum_{j=1}^d \left(|z_j|^2+2ix_jy_j-2y_j^2\right)^q\\
&=&\sum_{j=1}^d|z_j|^{2q}+\mathrm{Re}\sum_{j=1}^d \sum_{t=1}^q{q\choose t} |z_j|^{2(q-t)}\left(2ix_jy_j-2y_j^2\right)^t.
\eeqs
Now, observe that\\
\\
$\ds \left|\mathrm{Re}\sum_{j=1}^d \sum_{t=1}^q{q\choose t} |z_j|^{2(q-t)}\left(2ix_jy_j-2y_j^2\right)^t\right|$
\beqs
&\leq&\sum_{j=1}^d \sum_{t=1}^q{q\choose t} |z_j|^{2(q-t)}\left(2|x_j||y_j|+2y_j^2\right)^t\leq \sum_{j=1}^d \sum_{t=1}^q{q\choose t} |z_j|^{2(q-t)}\cdot\frac{|x|^{2t}}{2^{t(q+1)}d^{t(q-1)}}\\
&\leq& \frac{|z|^{2q}}{2^{q+1}d^{q-1}} \sum_{t=1}^q{q\choose t}\leq \frac{|z|^{2q}}{2d^{q-1}}.
\eeqs
On the other hand, we have
$
\sum_{j=1}^d|z_j|^{2q}\geq \frac{|z|^{2q}}{d^{q-1}}.
$
Hence, for $z\in W$ and $|x|\geq 2^{q+4}d^q r'$, we obtain
\beqs
\mathrm{Re}(z^{2q}_1+...+z^{2q}_d)\geq |z|^{2q}/(2d^{q-1})\geq (2d)^{-q}|z|^{2q}>0.
\eeqs
Thus $P_l$, resp. $P_{l_p}$, does vanish on $W$. Moreover, by the above estimates for $z=x+iy\in W$ and $|x|\geq 2^{q+4}d^q r'$ in the Roumieu case, we have
\beqs
|P_{l_p}(z)|&\geq& \sup_{\substack{p\in\ZZ_+\\ p\geq j_0}}\prod_{j=j_0}^p\left|1+\frac{z^{2q}_1+...+z^{2q}_d}{l_j^{2q} m_j^{2q}}\right| \geq \sup_{\substack{p\in\ZZ_+\\ p\geq j_0}}\prod_{j=j_0}^p\frac{|z|^{2q}}{(2dl_j^2 m_j^2)^q}\\
&=& \left(\sup_{p\in\NN} \frac{|z|^{p+1} M_{j_0-1}\prod_{j=1}^{j_0-1} l_j}{(2d)^{(p+1)/2}M_{p+j_0}\prod_{j=1}^{p+j_0} l_j}\right)^{2q}\\
&\geq& c_0^{-2q}(2H)^{-2q(j_0-1)}\left(\sup_{p\in\NN} \frac{|z|^{p+1}}{(2H\sqrt{2d})^{p+1} M_{p+1}\prod_{j=1}^{p+1} l_j}\right)^{2q}\\
&\geq& c_0^{-2q}(2H)^{-2q(j_0-1)}\left(\sup_{p\in\NN} \frac{|z|^{p+1}}{M_{p+1}\prod_{j=1}^{p+1} (k_j/2)}\right)^{2q}\\
&=& C'e^{2qN_{k_p/2}(|z|)}\geq  C'e^{N_{k_p/2}(|z|)},
\eeqs
and similarly $|P_l(x)|\geq C'e^{M(2k|z|)}$ in the Beurling case. As $P_l(z)$ and $P_{l_p}(z)$ don't have zeroes on $W$, the same inequalities hold on $W$, possibly with another $C'$. Hence, for $x\in\RR^d$, by Cauchy integral formula applied on the distinguished boundary of the polydisc $T(x)=\{w\in\CC^d|\, |w_j-x_j|\leq r',\, j=1,...,d\}\subseteq \mathrm{int}\,W$, we have
\beqs
\left|D^{\alpha}\left(\frac{1}{P_{l_p}(x)}\right)\right|\leq \frac{\alpha!}{r'^{|\alpha|}} \sup_{w\in T(x)} \frac{1}{|P_{l_p}(w)|}.
\eeqs
If $|x|\geq 2r'\sqrt{d}$, then for $w\in T(x)$, one easily verifies that $|w|\geq |x|/2$, hence
\beq\label{stripb}
|P_{l_p}(w)|\geq C'e^{N_{k_p/2}(|w|)}\geq C'e^{N_{k_p}(|x|)}.
\eeq
This yields (\ref{rgrowth}) for $|x|\geq 2r'\sqrt{d}$ and by similar calculations also (\ref{bgrowth}) for $|x|\geq 2r'\sqrt{d}$. If $|x|\leq 2r'\sqrt{d}$, one also has (\ref{stripb}) with possibly another $C'$ (observe that $1\leq e^{M(k|x|)}\leq C''$ and $1\leq e^{N_{k_p}(|x|)}\leq C''$ when $|x|\leq 2r'\sqrt{d}$). Hence (\ref{bgrowth}) and (\ref{rgrowth}) are valid for all $x\in\RR^d$.
\end{proof}

\begin{proposition}\label{parametrix}
For every $t>0$ (resp. $(t_p)\in\mathfrak{R}$) there exist $G\in\SSS^{M_p,t}_{A_p,t}$ (resp. $G\in\SSS^{M_p,(t_p)}_{A_p,(t_p)}$) and an ultradifferential operator $P(D)$ of class $(M_p)$ (resp. $\{M_p\}$), such that $P(D)G=\delta$.
\end{proposition}

\begin{proof} Without loss of generality, we may assume $t\geq 1$ in the Beurling case. Let $k=(2L_0dt)^2$ and $r'=2L_0dt$ in the Beurling case. In the Roumieu case, let $r'=2$ and take $(k_p)\in\mathfrak{R}$ such that $k_p\leq t_p/d$, $p\in\ZZ_+$, and $\prod_{j=1}^{p+q} k_j\leq 2^{p+q} \prod_{j=1}^p k_j \cdot \prod_{j=1}^q k_j$ for all $p,q\in\ZZ_+$ (such a sequence exists by Lemma \ref{nwseq}). In view of Lemma \ref{ultpoly}, we can find an ultrapolynomial $P(z)$ of class $(M_p)$, resp. $\{M_p\}$, such that
\beqs
\left|D^{\alpha}\left({1}/{P(x)}\right)\right|&\leq& C\frac{\alpha!}{r'^{|\alpha|}} e^{-M(2Hk|x|)},\quad x\in\RR^d,\, \alpha\in\NN^d\\
\left|D^{\alpha}\left({1}/{P(x)}\right)\right|&\leq& C\frac{\alpha!}{r'^{|\alpha|}} e^{-N_{k_p/(8H)}(|x|)},\quad x\in\RR^d,\, \alpha\in\NN^d.
\eeqs
Define
\beqs
G(x)=\frac{1}{(2\pi)^d}\int_{\RR^d}\frac{e^{i x\xi}}{P(\xi)}d\xi,\quad x\in \RR^d.
\eeqs
Obviously $G\in C^{\infty}(\RR^d)$. We have
\beqs
\left|x^{\beta}D^{\alpha}G(x)\right|=\frac{1}{(2\pi)^d}\left|\int_{\RR^d}e^{i x\xi} (-D_{\xi})^{\beta}\left(\frac{\xi^{\alpha}}{P(\xi)}\right)d\xi\right|\leq \frac{1}{(2\pi)^d}\int_{\RR^d} \left|D^{\beta}_{\xi}\left(\frac{\xi^{\alpha}}{P(\xi)}\right)\right|d\xi.
\eeqs
In the Roumieu case we have
\beqs
\left|D^{\beta}_{\xi}\left(\frac{\xi^{\alpha}}{P(\xi)}\right)\right|&\leq& C\sum_{\substack{\gamma\leq\beta\\ \gamma\leq\alpha}} {\beta\choose\gamma}\frac{\alpha!(\beta-\gamma)!}{(\alpha-\gamma)! r'^{|\beta|-|\gamma|}}\cdot|\xi|^{|\alpha|-|\gamma|} e^{-N_{k_p/(8H)}(|\xi|)}\\
&\leq& C\frac{2^{|\alpha|}\beta!}{r'^{|\beta|}}\sum_{\substack{\gamma\leq\beta\\ \gamma\leq\alpha}} {\beta\choose\gamma}r'^{|\gamma|}|\xi|^{|\alpha|-|\gamma|}e^{-N_{k_p/(8H)}(|\xi|)}
\eeqs
and similarly, in the Beurling case, we obtain
\beqs
\left|D^{\beta}_{\xi}\left(\frac{\xi^{\alpha}}{P(\xi)}\right)\right|\leq C\frac{2^{|\alpha|}\beta!}{r'^{|\beta|}}\sum_{\substack{\gamma\leq\beta\\ \gamma\leq\alpha}} {\beta\choose\gamma}r'^{|\gamma|}|\xi|^{|\alpha|-|\gamma|}e^{-M(2Hk|\xi|)}.
\eeqs
Observe that, for $|\xi|\geq 1$, in the Roumieu case we have (recall $r'=2$)
$$
r'^{|\gamma|}|\xi|^{|\alpha|-|\gamma|}e^{-N_{k_p/(8H)}(|\xi|)}\leq 2^{|\alpha|}|\xi|^{-d-1}|\xi|^{|\alpha|+d+1}e^{-N_{k_p/(8H)}(|\xi|)}
$$
$$\leq  2^{|\alpha|}|\xi|^{-d-1}M_{|\alpha|+d+1}\prod_{j=1}^{|\alpha|+d+1}(k_j/(8H))
\leq c_0M_{d+1}\prod_{j=1}^{d+1}(k_j/2)|\xi|^{-d-1}M_{\alpha}\prod_{j=1}^{|\alpha|}(k_j/2)
$$$$= C'|\xi|^{-d-1}M_{\alpha}\prod_{j=1}^{|\alpha|}(k_j/2)
$$
and for $|\xi|\leq 1$, we obtain
\beqs
r'^{|\gamma|}|\xi|^{|\alpha|-|\gamma|}e^{-N_{k_p/(8H)}(|\xi|)}\leq 2^{|\alpha|}\leq c'M_{\alpha}\prod_{j=1}^{|\alpha|}(k_j/2).
\eeqs
Thus, by using $\beta!\leq c_1A_{\beta} K_{\beta}$, where $K_{\beta}=\prod_{j=1}^{|\beta|}k_j$ for $\beta\in\NN^d$, we obtain
\beqs
\int_{\RR^d}\left|D^{\beta}_{\xi}\left(\frac{\xi^{\alpha}}{P(\xi)}\right)\right|dx\leq C'_1M_{\alpha}A_{\beta}K_{\alpha}K_{\beta},
\eeqs
i.e., $\left|x^{\beta}D^{\alpha}G(x)\right|\leq C''M_{\alpha}A_{\beta}K_{\alpha}K_{\beta}$, for all $x\in\RR^d$, $\alpha,\beta\in\NN^d$. In the Beurling case, for $|\xi|\geq 1$, (recall $k=(2L_0dt)^2$ and $r'=2L_0dt$)
$$
r'^{|\gamma|}|\xi|^{|\alpha|-|\gamma|}e^{-M(2Hk|\xi|)}\leq r'^{|\alpha|}|\xi|^{-d-1}|\xi|^{|\alpha|+d+1}e^{-M(2Hk|\xi|)}
$$$$\leq r'^{|\alpha|}|\xi|^{-d-1} (2Hk)^{-|\alpha|-d-1}M_{|\alpha|+d+1}
\leq c_0r'^{|\alpha|}(2k)^{-|\alpha|-d-1}M_{d+1}|\xi|^{-d-1}M_{\alpha}
$$$$\leq C'(2L_0dt)^{-|\alpha|}|\xi|^{-d-1}M_{\alpha}.
$$
For $|\xi|\leq 1$ we have
\beqs
r'^{|\gamma|}|\xi|^{|\alpha|-|\gamma|}e^{-M(2Hk|\xi|)}\leq r'^{|\alpha|}\leq c'(2L_0dt)^{-|\alpha|}M_{\alpha}.
\eeqs
Now, by using $\beta!\leq c_0L_0^{|\beta|}A_{\beta}$ we obtain $\left|x^{\beta}D^{\alpha}G(x)\right|\leq C'''(dt)^{-|\alpha|-|\beta|}M_{\alpha}A_{\beta}$, for all $x\in\RR^d$, $\alpha,\beta\in\NN^d$ (recall $r'=2L_0dt$). Observe
\beqs
|x|^{2|\beta|}\left|D^{\alpha}G(x)\right|^2=\sum_{|\gamma|=|\beta|} \frac{|\beta|!}{\gamma!}x^{2\gamma}\left|D^{\alpha}G(x)\right|^2.
\eeqs
Thus $|x|^{|\beta|}\left|D^{\alpha}G(x)\right|\leq C_2 t^{-|\alpha|-|\beta|}M_{\alpha}A_{\beta}$ in the Beurling case and $|x|^{|\beta|}\left|D^{\alpha}G(x)\right|\leq C_2M_{\alpha}A_{\beta}T_{\alpha}T_{\beta}$ in the Roumieu case (we set $T_{\gamma}=\prod_{j=1}^{|\gamma|}t_j$ for $\gamma\in\NN^d$), for all $x\in\RR^d$, $\alpha,\beta\in\NN^d$ (recall $k_p\leq t_p/d$, $p\in\ZZ_+$). Hence
\beqs
\frac{t^{|\alpha|}\left|D^{\alpha}G(x)\right|e^{A(t|x|)}}{M_{\alpha}}\leq C_2\,\, \left(\mbox{resp.}\,\, \frac{\left|D^{\alpha}G(x)\right|e^{B_{t_p}(|x|)}}{M_{\alpha}T_{\alpha}}\leq C_2\right),\,\, \forall x\in\RR^d,\, \forall\alpha\in\NN^d.
\eeqs
It remains to prove that $P(D)G=\delta$. For $\varphi\in\SSS^*_{\dagger}(\RR^d)$ we have
\beqs
\langle P(D)G, \varphi\rangle&=&\int_{\RR^d}G(x)P(-D)\varphi(x)dx=\frac{1}{(2\pi)^d}\int_{\RR^{2d}}\frac{e^{i x\xi}}{P(\xi)}\cdot P(-D)\varphi(x)dxd\xi\\
&=&\int_{\RR^d}\frac{\mathcal{F}^{-1}(P(-D)\varphi)(\xi)}{P(\xi)}d\xi= \int_{\RR^d}\mathcal{F}^{-1}\varphi(\xi)d\xi=\varphi(0),
\eeqs
i.e., $P(D)G=\delta$, which completes the proof.
\end{proof}

\begin{example} Using Proposition \ref{parametrix},  one can construct an entire function $f$ of exponential decay such that $P(D)f$ is continuous but nowhere differentiable. More precisely, let $M_p\subset p!$ (e.g., $M_p=p!^{\sigma}$ with $0<\sigma<1$) and $A_p=p!$.  By considering the Weierstrass continuous nowhere differentiable function on $\RR$, Proposition \ref{parametrix} implies that, for each fixed $\tau,r>0$, one can find an entire function $f$ satisfying $|f^{(k)}(x)|\leq C M_k e^{-\tau|x|}/r^k$, $\forall x\in\RR$, $\forall k\in\NN$ and an ultradifferential operator $P(D)$ of class $(M_p)$ such that $P(D)f=$Weierstrass function.
\end{example}

\begin{lemma}\label{appincl}
Let $r>0$ (resp. $(r_p)\in\mathfrak{R}$).
\begin{itemize}
\item[$i)$] For each $\chi,\varphi\in \mathcal{S}^*_{\dagger}(\mathbb{R}^d)$ and $\psi\in\SSS^{M_p,r}_{A_p,r}$ (resp. $\psi\in\SSS^{M_p,(r_p)}_{A_p,(r_p)}$), $\chi*(\varphi\psi)\in \SSS^*_{\dagger}(\RR^d)$.
\item[$ii)$] Let $\varphi,\chi\in \mathcal{S}^*_{\dagger}(\mathbb{R}^d)$ with $\varphi(0)=1$ and $\int_{\RR^d}\chi(x)dx=1$. For each $n\in\ZZ_+$ define $\chi_n(x)=n^d\chi(nx)$ and $\varphi_n(x)=\varphi(x/n)$. Then there exists $k\geq 2r$ (resp. $(k_p)\in\mathfrak{R}$ with $(k_p)\leq (r_p/2)$), such that the operators $\tilde{Q}_n:\psi\mapsto\chi_n*(\varphi_n\psi)$, are continuous as mappings from $\SSS^{M_p,k}_{A_p,k}$ to $\SSS^{M_p,r}_{A_p,r}$ (resp. from $\SSS^{M_p,(k_p)}_{A_p,(k_p)}$ to $\SSS^{M_p,(r_p)}_{A_p,(r_p)}$), for all $n\in\ZZ_+$. Moreover $\tilde{Q}_n\rightarrow \mathrm{Id}$ in $\mathcal{L}_b\left(\SSS^{M_p,k}_{A_p,k},\SSS^{M_p,r}_{A_p,r}\right)$ (resp. in $\mathcal{L}_b\left(\SSS^{M_p,(k_p)}_{A_p,(k_p)},\SSS^{M_p,(r_p)}_{A_p,(r_p)}\right)$).
\end{itemize}
\end{lemma}

\begin{proof} We give the proof for the Roumieu case. The Beurling case is similar. To see that $\chi*(\varphi\psi)\in \SSS^{\{M_p\}}_{\{A_p\}}(\RR^d)$ let $(l_p)\in\mathfrak{R}$ be arbitrary but fixed. Set $k_p=l_p/2$, $p\in\ZZ_+$. Clearly $(k_p)\in\mathfrak{R}$. Observe that
\beqs
e^{B_{l_p}(|x|)}\leq 2e^{B_{l_p}(2|x-y|)}e^{B_{l_p}(2|y|)}= 2e^{B_{k_p}(|x-y|)}e^{B_{k_p}(|y|)},
\eeqs
so
$$
\left|D^{\alpha}\left(\chi*(\varphi\psi)(x)\right)\right|\leq \int_{\RR^d} \left|D^{\alpha}\chi(y)\right||\varphi(x-y)||\psi(x-y)|dy\\
$$
$$\leq C_1M_{\alpha} \prod_{j=1}^{|\alpha|}k_j\int_{\RR^d} e^{-B_{k_p}(|y|)} e^{-B_{k_p}(|x-y|)}e^{-B_{r_p}(|x-y|)}dy\\
\leq C_2 M_{\alpha} \prod_{j=1}^{|\alpha|}l_j\cdot e^{-B_{l_p}(|x|)};
$$
which completes the proof of $i)$.

We will prove $ii)$ in the Roumieu case. The Beurling case is similar. Let $(r_p)\in\mathfrak{R}$ be fixed. By Lemma \ref{nwseq} one can find $(k'_p)\in\mathfrak{R}$ such that $k'_p\leq r_p/(4H)$, $\forall p\in\ZZ_+$ and $\prod_{j=1}^{p+q} k'_j\leq 2^{p+q} \prod_{j=1}^p k'_j\cdot \prod_{j=1}^q k'_j$ for all $p,q\in\ZZ_+$. Let $k_p=k'_p/(2H)$, $p\in\ZZ_+$. Obviously, $(k_p)\in\mathfrak{R}$. We will prove that $(k_p)$ satisfies the required conditions. For $\alpha\in\NN^d$ set $R_{\alpha}=\prod_{j=1}^{|\alpha|}r_j$, $K'_{\alpha}=\prod_{j=1}^{|\alpha|}k'_j$ and $K_{\alpha}=\prod_{j=1}^{|\alpha|}k_j$. For $\psi\in\SSS^{M_p,(k_p)}_{A_p,(k_p)}$ denote $\psi_n=\varphi_n\psi$ and proceed as follows\\
\\
$\ds \frac{\left|D^{\alpha}\left(\chi_n*(\varphi_n\psi)\right)(x)-D^{\alpha}\psi(x)\right|e^{B_{r_p}(|x|)}}{M_{\alpha}R_{\alpha}}$
\beqs
&\leq& 2\int_{\RR^d}|\chi_n(y)|e^{B_{r_p}(2|y|)}\cdot \frac{\left|D^{\alpha}\psi_n(x-y)-D^{\alpha}\psi_n(x)\right|e^{B_{r_p}(2|x-y|)}}{M_{\alpha}R_{\alpha}}dy\\
&{}&+ \frac{\left|D^{\alpha}\psi_n(x)-D^{\alpha}\psi(x)\right|e^{B_{r_p}(|x|)}}{M_{\alpha}R_{\alpha}}\\
&\leq& 2\int_{\RR^d}|\chi(y)|e^{B_{r_p}(2|y|)}\cdot \frac{\left|D^{\alpha}\psi_n(x-y/n)-D^{\alpha}\psi_n(x)\right|e^{B_{r_p}(2|x-y/n|)}}{M_{\alpha}R_{\alpha}}dy\\
&{}&+ \frac{\left|D^{\alpha}\psi_n(x)-D^{\alpha}\psi(x)\right|e^{B_{k'_p}(|x|)}}{M_{\alpha}K'_{\alpha}}.
\eeqs
Denote the first term as $S'_{n,\alpha}(x)$ and the second one as $S''_{n,\alpha}(x)$. To estimate $S''_{n,\alpha}(x)$, observe that, by construction, the sequences $M_p\prod_{j=1}^p k'_j$ and $A_p\prod_{j=1}^p k'_j$ satisfy $(M.2)$ with constant $2H$ instead of $H$. Thus \cite[Proosition 3.6]{Komatsu1} implies
\beqs
e^{2B_{k'_p}(|x|)}\leq c'e^{B_{k'_p}(2H|x|)}=c'e^{B_{k_p}(|x|)}.
\eeqs
Hence,\\
\\
$S''_{n,\alpha}(x)$
\beqs
&\leq&\frac{\left|1-\varphi(x/n)\right|\left|D^{\alpha}\psi(x)\right|e^{B_{k'_p}(|x|)}}{M_{\alpha} K'_{\alpha}} +\frac{1}{n}\sum_{\substack{\beta\leq \alpha\\ \beta\neq 0}} {\alpha\choose\beta} \frac{\left|D^{\beta}\varphi(x/n)\right|\left|D^{\alpha-\beta}\psi(x)\right|e^{B_{k'_p}(|x|)}}{M_{\alpha} K'_{\alpha}}\\
&\leq& c'\sigma_{(k_p)}(\psi)\left|1-\varphi(x/n)\right|e^{-B_{k'_p}(|x|)}+ \frac{\sigma_{(k_p)}(\varphi)\sigma_{(k_p)}(\psi)}{n(2H)^{|\alpha|}} \sum_{\substack{\beta\leq \alpha\\ \beta\neq 0}} {\alpha\choose\beta}\\
&\leq&
c'\sigma_{(k_p)}(\psi)\left|1-\varphi(x/n)\right|e^{-B_{k'_p}(|x|)}+
\frac{\sigma_{(k_p)}(\varphi)\sigma_{(k_p)}(\psi)}{n}
\eeqs
Set
$a_n=\left\|\left(1-\varphi(\cdot/n)\right)e^{-B_{k'_p}(|\cdot|)}\right\|_{L^{\infty}(\RR^d)}$.
We prove that $a_n\rightarrow 0$. Let $\varepsilon>0$. Pick $C>0$
such that $e^{-B_{k'_p}(|x|)}\leq
\varepsilon/(1+\|\varphi\|_{L^{\infty}(\RR^d)})$ when $|x|\geq C$.
Since $\varphi(0)=1$ and $\varphi$ is continuous there exists
$n_0\in\ZZ_+$ such that for all $n\geq n_0$, $|1-\varphi(x/n)|\leq
\varepsilon$ for all $|x|\leq C$. Thus
$\left|1-\varphi(x/n)\right|e^{-B_{k'_p}(|x|)}\leq \varepsilon$
for all $x\in\RR^d$ and $n\geq n_0$, which proves $a_n\rightarrow
0$. We obtain
\beq\label{est15} S''_{n,\alpha}(x)\leq
\sigma_{(k_p)}(\psi)\left(c'a_n+\sigma_{(k_p)}(\varphi)/n\right),
 \eeq
for all, $x\in\RR^d$, $\alpha\in\NN^d$,
$\psi\in\SSS^{M_p,(k_p)}_{A_p,(k_p)}$. To estimate
$S'_{n,\alpha}$, Taylor expand $D^{\alpha}\psi_n$ at $x$ to obtain
\beqs
\left|D^{\alpha}\psi_n(x-y/n)-D^{\alpha}\psi_n(x)\right|&\leq& \frac{|y|}{n}\sum_{|\beta|=1}\int_0^1\left|D^{\alpha+\beta}\psi_n(x-ty/n)\right|dt\\
&\leq& \frac{d\sigma_{(k'_p)}(\psi_n)M_{|\alpha|+1}K'_{|\alpha|+1}|y|}{n}\int_0^1 e^{-B_{k'_p}(|x-ty/n|)}dt.
\eeqs
Observe that
\beqs
e^{B_{r_p}(2|x-y/n|)}&\leq& 2e^{B_{r_p}(4|x-ty/n|)}e^{B_{r_p}(4|(1-t)y/n|)}\leq 2e^{B_{k'_p}(|x-ty/n|)}e^{B_{k'_p}(|y|)}\,\,\,\, \mbox{and}\\
M_{|\alpha|+1}K'_{|\alpha|+1}&\leq& c_0(2H)^{|\alpha|+1}M_{\alpha}k'_1K'_{\alpha}\leq c_0r_1M_{\alpha}R_{\alpha}.
\eeqs
Moreover by definition $S''_{n,\alpha}(x)$, $\ds\sup_{\alpha} \|S''_{n,\alpha}\|_{L^{\infty}(\RR^d)}=\sigma_{(k'_p)}(\psi_n-\psi)$. Hence, the above estimate for $S''_{n,\alpha}(x)$ implies
\beqs
\sigma_{(k'_p)}(\psi_n)\leq \sigma_{(k_p)}(\psi)\left(c'a_n+ \frac{\sigma_{(k_p)}(\varphi)}{n}+1\right)\leq c''\sigma_{(k_p)}(\psi).
\eeqs
Now, for $S'_{n,\alpha}(x)$ we have $S'_{n,\alpha}(x)\leq C'\sigma_{(k_p)}(\psi)/n$. This estimate together with (\ref{est15}) proves that $\tilde{Q}_n\in\mathcal{L}\left(\SSS^{M_p,(k_p)}_{A_p,(k_p)},\SSS^{M_p,(r_p)}_{A_p,(r_p)}\right)$, for all $n\in\ZZ_+$ and $\tilde{Q}_n\rightarrow \mathrm{Id}$ in $\mathcal{L}_b\left(\SSS^{M_p,(k_p)}_{A_p,(k_p)},\SSS^{M_p,(r_p)}_{A_p,(r_p)}\right)$, which completes the proof.
\end{proof}

\subsection{Structure theorems for $\SSS'^*_{\dagger}(\RR^d)$} \label{structure}
 In the next proposition for $t>0$ (resp. $(t_p)\in\mathfrak{R}$), we denote by $\overline{\SSS}^{M_p,t}_{A_p,t}$ (resp. by $\overline{\SSS}^{M_p,(t_p)}_{A_p,(t_p)}$) the closure of $\SSS^{(M_p)}_{(A_p)}(\RR^d)$ in $\SSS^{M_p,t}_{A_p,t}$ (resp. the closure of $\SSS^{\{M_p\}}_{\{A_p\}}(\RR^d)$ in $\SSS^{M_p,(t_p)}_{A_p,(t_p)}$).

\begin{proposition}\label{parametrix1}
Let $B$ be a bounded subset of $\SSS'^*_{\dagger}(\RR^d)$. There exists $k>0$ (resp. $(k_p)\in\mathfrak{R}$) such that each $f\in B$ can be extended to a continuous functional $\tilde{f}$ on $\overline{\SSS}^{M_p,k}_{A_p,k}$ (resp. on $\overline{\SSS}^{M_p,(k_p)}_{A_p,(k_p)}$). Moreover, there exists $l\geq k$ (resp. $(l_p)\in\mathfrak{R}$ with $(l_p)\leq (k_p)$) such that $\SSS^{M_p,l}_{A_p,l}\subseteq\overline{\SSS}^{M_p,k}_{A_p,k}$ (resp. $\SSS^{M_p,(l_p)}_{A_p,(l_p)}\subseteq \overline{\SSS}^{M_p,(k_p)}_{A_p,(k_p)}$) and $*:\SSS^{M_p,l}_{A_p,l}\times\SSS^{M_p,l}_{A_p,l}\rightarrow \overline{\SSS}^{M_p,k}_{A_p,k}$ (resp. $*:\SSS^{M_p,(l_p)}_{A_p,(l_p)}\times\SSS^{M_p,(l_p)}_{A_p,(l_p)}\rightarrow \overline{\SSS}^{M_p,(k_p)}_{A_p,(k_p)}$) is a continuous bilinear mapping. Furthermore, there exist an ultradifferential operator $P(D)$ of class $*$ and $u\in \overline{\SSS}^{M_p,l}_{A_p,l}$ (resp. $u\in\overline{\SSS}^{M_p,(l_p)}_{A_p,(l_p)}$) such that $P(D)u=\delta$ and $f=(P(D)u)*f=P(D)(u*\tilde{f})$ for each $f\in B$, where $u*\tilde{f}$ is the image of $\tilde{f}$ under the transpose of the continuous mapping $\varphi\mapsto\check{u}*\varphi$, $\SSS^{(M_p)}_{(A_p)}(\RR^d)\rightarrow \overline{\SSS}^{M_p,k}_{A_p,k}$ (resp. $\SSS^{\{M_p\}}_{\{A_p\}}(\RR^d)\rightarrow \overline{\SSS}^{M_p,(k_p)}_{A_p,(k_p)}$). For $f\in B$, $u*\tilde{f}\in L^{\infty}_{e^{A(l|\cdot|)}}\cap C(\RR^d)$ (resp. $u*\tilde{f}\in L^{\infty}_{e^{B_{l_p}(|\cdot|)}}\cap C(\RR^d)$) and in fact $u*\tilde{f}(x)=\langle \tilde{f}, u(x-\cdot)\rangle$. The set $\{u*\tilde{f}|\, f\in B\}$ is bounded in $L^{\infty}_{e^{A(l|\cdot|)}}$ (resp. in $L^{\infty}_{e^{B_{l_p}(|\cdot|)}}$).
\end{proposition}

\begin{proof} We give the proof in the Roumieu case. The Beurling case is analogous. For brevity in notation for $(t_p)\in\mathfrak{R}$, set $X_{(t_p)}=\overline{\SSS}^{M_p,(t_p)}_{A_p,(t_p)}$. Since $B$ is bounded in $\SSS'^{\{M_p\}}_{\{A_p\}}(\RR^d)$, it is equicontinuous. Hence, there exist $(r_p)\in\mathfrak{R}$ and $C>0$ such that $|\langle f,\varphi\rangle|\leq C \sigma_{(r_p)}(\varphi)$ for all $f\in B$, i.e., each $f$ can be extended to a continuous linear functional on $X_{(r_p)}$ and if we denote this extension of $f\in B$ by $\tilde{f}$, then the set $\tilde{B}=\{\tilde{f}\in X'_{(r_p)}|\, f\in B\}$ is bounded in $X'_{(r_p)}$. For $n\in\ZZ_+$, let $\chi_n$ and $\varphi_n$ be as in $ii)$ of Lemma \ref{appincl}. One easily verifies that the mapping $Q_n:\psi\mapsto \chi_n*(\varphi_n\psi)$ is continuous from $\SSS^{\{M_p\}}_{\{A_p\}}(\RR^d)$ to $\SSS^{\{M_p\}}_{\{A_p\}}(\RR^d)$ for all $n\in\ZZ_+$ and hence also $Q_n\in\mathcal{L}\left(\SSS^{\{M_p\}}_{\{A_p\}}(\RR^d),\SSS^{M_p,(r_p)}_{A_p,(r_p)}\right)$ for all $n\in\ZZ_+$. For this $(r_p)$, Lemma \ref{appincl} implies that there exists $(k_p)\in\mathfrak{R}$ with $(k_p)\leq (r_p/(2H))$ for which the operators $\tilde{Q}_n:\psi\mapsto\chi_n*(\varphi_n\psi)$, are continuous as mappings from $\SSS^{M_p,(k_p)}_{A_p,(k_p)}$ to $\SSS^{M_p,(r_p)}_{A_p,(r_p)}$, for all $n\in\ZZ_+$ and $\tilde{Q}_n\rightarrow \mathrm{Id}$ in $\mathcal{L}_b\left(\SSS^{M_p,(k_p)}_{A_p,(k_p)},\SSS^{M_p,(r_p)}_{A_p,(r_p)}\right)$. By $i)$ of Lemma \ref{appincl}, $\tilde{Q}_n(X_{k_p})\subseteq \SSS^{\{M_p\}}_{\{A_p\}}(\RR^d)\subseteq X_{(r_p)}$ for all $n\in\ZZ_+$, hence $\tilde{Q}_n\in\mathcal{L}\left(X_{(k_p)},X_{(r_p)}\right)$ for all $n\in\ZZ_+$ and $\tilde{Q}_n\rightarrow \mathrm{Id}$ in $\mathcal{L}_b\left(X_{(k_p)},X_{(r_p)}\right)$. Obviously the restriction of $\tilde{Q}_n$ to $\SSS^{\{M_p\}}_{\{A_p\}}(\RR^d)$ is $Q_n$. Denoting by ${}^t\tilde{Q}_n$ the transpose of $\tilde{Q}_n$ we have ${}^t\tilde{Q}_n\rightarrow \mathrm{Id}$ in $\mathcal{L}_b\left(X'_{(r_p)},X'_{(k_p)}\right)$. Also for $\tilde{f}\in \tilde{B}$, the restriction of ${}^t\tilde{Q}_n\tilde{f}$ to $\SSS^{\{M_p\}}_{\{A_p\}}(\RR^d)$ is exactly the ultradistribution ${}^tQ_nf=\varphi_n(\check{\chi}_n*f)$. But we have that $g_{f,n}=\varphi_n(\check{\chi}_n*f)\in \SSS^{\{M_p\}}_{\{A_p\}}(\RR^d)$. Again, Lemma \ref{appincl} implies that there exists $(l'_p)\in\mathfrak{R}$ with $(l'_p)\leq (k_p/(2H))$ such that for each $v\in \SSS^{M_p,(l'_p)}_{A_p,(l'_p)}$, $\chi_n*(\varphi_nv)\rightarrow v$ in $\SSS^{M_p,(k_p)}_{A_p,(k_p)}$ and $\chi_n*(\varphi_nv)\in \SSS^{\{M_p\}}_{\{A_p\}}(\RR^d)$, i.e., $\SSS^{M_p,(l'_p)}_{A_p,(l'_p)}\subseteq X_{(k_p)}$. By Lemma \ref{nwseq} we can assume that $\prod_{j=1}^{p+q}l'_j\leq 2^{p+q} \prod_{j=1}^p l'_j\cdot \prod_{j=1}^q l'_j$ for all $p,q\in\ZZ_+$. Set $l_p=l'_p/(4H)$, $p\in\ZZ_+$. One easily verifies that for each $h\in \RR^d$, $T_h\in\mathcal{L}_b\left(\SSS^{M_p,(l_p)}_{A_p,(l_p)},\SSS^{M_p,(l'_p)}_{A_p,(l'_p)}\right)=Y_{(l_p),(l'_p)}$ and there exists $c>0$ such that $\|T_h\|_{Y_{(l_p),(l'_p)}}\leq ce^{B_{l_p}(|h|)}$. Moreover, there exists $c'>0$ such that for $h,h_0\in\RR^d$ and $v\in \SSS^{M_p,(l_p)}_{A_p,(l_p)}$,
\beq\label{115rl}
\sigma_{(l'_p)}(T_hv-T_{h_0}v)\leq c' |h-h_0|\sigma_{(l_p)}(v)e^{B_{l_p}(|h_0|+|h|)}.
\eeq
To prove this one estimates $D^{\alpha}v(x+h)-D^{\alpha}v(x+h_0)$ by Taylor expanding $D^{\alpha}v$ at $x+h_0$ similarly as in the proof of $ii)$ of Lemma \ref{appincl}. Also one readily verifies that for this $(l_p)$ and $v,v'\in\SSS^{M_p,(l_p)}_{A_p,(l_p)}$, $\sigma_{(l'_p)}(v*v')\leq c''\sigma_{(l_p)}(v)\sigma_{(l_p)}(v')$, i.e., $*:\SSS^{M_p,(l_p)}_{A_p,(l_p)}\times \SSS^{M_p,(l_p)}_{A_p,(l_p)}\rightarrow \SSS^{M_p,(l'_p)}_{A_p,(l'_p)}$ is continuous bilinear mapping. Similarly as above, take $(l''_p)\in\mathfrak{R}$ with $(l''_p)\leq (l_p)$ such that $\SSS^{M_p,(l''_p)}_{A_p,(l''_p)}\subseteq X_{(l_p)}$. By Proposition \ref{parametrix} there exist $u\in\SSS^{M_p,(l''_p)}_{A_p,(l''_p)}\subseteq X_{(l_p)}\subseteq X_{(k_p)}$ and an ultradifferential operator $P(D)$ of class $\{M_p\}$ such that $P(D)u=\delta$. Hence $(P(D)u)*f=f$. Moreover, $(P(D)u)*g_{f,n}=g_{f,n}$. Since $g_{f,n}\rightarrow f$ in $\SSS'^{\{M_p\}}_{\{A_p\}}(\RR^d)$ we obtain $(P(D)u)*g_{f,n}\rightarrow f=(P(D)u)*f$ in $\SSS'^{\{M_p\}}_{\{A_p\}}(\RR^d)$. On the other hand, since $g_{f,n}\in\SSS^{\{M_p\}}_{\{A_p\}}(\RR^d)$, $(P(D)u)*g_{f,n}=P(D)(u*g_{f,n})$. Let $\varphi\in\SSS^{\{M_p\}}_{\{A_p\}}(\RR^d)$. Then $\check{u}*\varphi\in \SSS^{M_p,(l'_p)}_{A_p,(l'_p)}\subseteq X_{(k_p)}$ and thus
\beqs
{}_{\SSS'^{\{M_p\}}_{\{A_p\}}}\langle u*g_{f,n},\varphi\rangle_{\SSS^{\{M_p\}}_{\{A_p\}}}&=&\int_{\RR^d}u*g_{f,n}(x)\varphi(x)dx =\int_{\RR^d}g_{f,n}(x)\check{u}*\varphi(x)dx\\
&=&{}_{X'_{(k_p)}}\langle g_{f,n}, \check{u}*\varphi\rangle_{X_{(k_p)}}.
\eeqs
The right hand side tends to ${}_{X'_{(k_p)}}\langle \tilde{f}, \check{u}*\varphi\rangle_{X_{(k_p)}}$ where we used the same notation for $\tilde{f}\in X'_{(r_p)}$ and its restriction to $X_{(k_p)}$. Now, observe that
\beqs
x\mapsto {}_{X'_{(k_p)}}\langle \tilde{f}, u(x-\cdot)\rangle_{X_{(k_p)}}={}_{X'_{(k_p)}}\langle \tilde{f}, T_x\check{u}\rangle_{X_{(k_p)}},\,\, \RR^d\rightarrow \CC,
\eeqs
is well defined function which we denote by $F_f$. By the estimates for $\|T_h\|_{Y_{(l_p),(l'_p)}}$ it follows that there exists $C'$ such that $|F_f(x)|\leq C' \|\tilde{f}\|_{X'_{(k_p)}}e^{B_{l_p}(|x|)}$. Moreover, by (\ref{115rl}) it follows that $F_f$ is continuous. Hence $\{F_f|\,f\in B\}\subseteq L^{\infty}_{e^{B_{l_p}(|x|)}}\cap C(\RR^d)$ and $\{F_f|\,f\in B\}$ is bounded in $L^{\infty}_{e^{B_{l_p}(|x|)}}$. Moreover, for each $f\in B$, $u*g_{f,n}(x)={}_{X'_{(k_p)}}\langle g_{f,n},T_x\check{u}\rangle_{X_{(k_p)}}$ and thus $u*g_{f,n}\rightarrow F_f$ in $L^{\infty}_{e^{B_{l_p}(|x|)}}$ hence also in $\SSS'^{\{M_p\}}_{\{A_p\}}(\RR^d)$. Hence $\langle \tilde{f},u(x-\cdot)\rangle$ is exactly the image of $\tilde{f}$ under the transpose of the mapping $\varphi\mapsto\check{u}*\varphi$, $\SSS^{\{M_p\}}_{\{A_p\}}(\RR^d)\rightarrow X_{(k_p)}$. Now, clearly $P(D)(u*g_{f,n})(x)\rightarrow P(D)(\langle \tilde{f},u(x-\cdot)\rangle)$ in $\SSS'^{\{M_p\}}_{\{A_p\}}(\RR^d)$, which completes the proof of the proposition.
\end{proof}

\begin{remark} From the proof of this proposition it is clear that the numbers $k$ and $l$ in the Beurling case can be chosen arbitrary large (resp. the sequence $(k_p),(l_p)\in\mathfrak{R}$ in the Roumieu case can be chosen arbitrary small) in the following sense:\\
\indent Given $t>0$ (resp. $(t_p)\in\mathfrak{R})$ one can choose the numbers $k$ and $l$ (resp. the sequences $(k_p),(l_p)\in\mathfrak{R}$) such that $t\leq k\leq l$ (resp. $(l_p)\leq (k_p)\leq (t_p)$).
\end{remark}

\begin{lemma}\label{boundedsetS}
Let $B\subseteq\SSS'^*_{\dagger}(\RR^d)$. The following are equivalent:
\begin{itemize}
\item[$i)$] $B$ is bounded in $\SSS'^*_{\dagger}(\RR^d)$;
\item[$ii)$] for each $\varphi\in\SSS^*_{\dagger}(\RR^d)$, $\{f*\varphi|\, f\in B\}$ is bounded in $\SSS'^*_{\dagger}(\RR^d)$;
\item[$iii)$] for each $\varphi\in\SSS^*_{\dagger}(\RR^d)$ there exist $t,C>0$ (resp. $(t_p)\in\mathfrak{R}$ and $C>0$) such that $|(f*\varphi)(x)|\leq C e^{A(t|x|)}$ (resp. $|(f*\varphi)(x)|\leq C e^{B_{t_p}(|x|)}$), for all $x\in \RR^d$, $f\in B$;
\item[$iv)$] there exist $C,t>0$ (resp. there exist $(t_p)\in\mathfrak{R}$ and $C>0$) such that
    \beqs
    \left|(f*\varphi)(x)\right|\leq Ce^{A(t|x|)}\sigma_t(\varphi)\,\, \left(\mbox{resp.}\,\, \left|(f*\varphi)(x)\right|\leq Ce^{B_{t_p}(|x|)}\sigma_{(t_p)}(\varphi)\right).
    \eeqs
    for all $\varphi\in\SSS^*_{\dagger}(\RR^d)$, $x\in\RR^d$, $f\in B$.
\end{itemize}
\end{lemma}

\begin{proof} The implication $i)\Rightarrow ii)$ is trivial. To prove $ii)\Rightarrow iii)$ suppose that $ii)$ holds. For $f\in B$ denote by $G_f$ the bilinear mapping defined by $(\varphi,\psi)\mapsto\langle f,\varphi*\psi\rangle$, $\SSS^*_{\dagger}(\RR^d)\times \SSS^*_{\dagger}(\RR^d)\rightarrow \CC$. Clearly $G_f$ is continuous. Moreover, for fixed $\varphi\in\SSS^*_{\dagger}(\RR^d)$, by assumption, the set $\{f*\check{\varphi}|\, f\in B\}$ is bounded in $\SSS'^*_{\dagger}(\RR^d)$, hence equicontinuous ($\SSS^*_{\dagger}(\RR^d)$ is barrelled), so the mappings $\psi\mapsto G_f(\varphi,\psi)=\langle f*\check{\varphi},\psi\rangle$, $f\in B$, are equicontinuous subset of $\mathcal{L}\left(\SSS^*_{\dagger}(\RR^d),\CC\right)$. An analogous conclusion holds for the set of mappings $\varphi\mapsto G_f(\varphi,\psi)=\langle f*\check{\psi},\varphi\rangle$, $f\in B$, when $\psi\in\SSS^*_{\dagger}(\RR^d)$ is fixed. Thus $\{G_f|\, f\in B\}$ is a separately equicontinuous set. As $\SSS^{(M_p)}_{(A_p)}(\RR^d)$ is an $(F)$-space (resp. $\SSS^{\{M_p\}}_{\{A_p\}}(\RR^d)$ is a barrelled $(DF)$-space), \cite[Theorem 2, p. 158]{kothe2} (resp. \cite[Theorem 11, p. 161]{kothe2}) implies that $\{G_f|\, f\in B\}$ is equicontinuous. We will continue the proof in the Roumieu case. The Beurling case is similar. For $(t_p)\in\mathfrak{R}$, denote by $X_{(t_p)}$ the closure of $\SSS^{\{M_p\}}_{\{A_p\}}(\RR^d)$ in $\SSS^{M_p, (t_p)}_{A_p,(t_p)}$. The equicontinuity of $\{G_f|\, f\in B\}$ implies that there exist $(r_p)\in\mathfrak{R}$ and $C>0$, such that $|G_f(\varphi,\psi)|\leq C \sigma_{(r_p)}(\varphi)\sigma_{(r_p)}(\psi)$ for all $f\in B$. Now, for $\chi\in \SSS^{\{M_p\}}_{\{A_p\}}(\RR^d)$, we have
\beq
|\langle f*\varphi*\psi,\chi\rangle|&=&|G_f(\check{\varphi},\chi*\check{\psi})|\leq C \sigma_{(r_p)}(\check{\varphi})\sigma_{(r_p)}(\chi*\check{\psi})\\
&\leq& C_1 \sigma_{(r_p)}(\varphi)\sigma_{(r_p/2)}(\psi)\|\chi\|_{L^1_{e^{B_{r_p/2}(|\cdot|)}}}\label{estft}
\eeq
Since $\SSS^{\{M_p\}}_{\{A_p\}}(\RR^d)$ is dense in $L^1_{e^{B_{r_p/2}(|\cdot|)}}$ we obtain that the bilinear mappings $T_f:(\varphi,\psi)\rightarrow f*\varphi*\psi$, $\SSS^{\{M_p\}}_{\{A_p\}}\times \SSS^{\{M_p\}}_{\{A_p\}}(\RR^d)\rightarrow L^{\infty}_{e^{B_{r_p/2}(|\cdot|)}}$, are well defined and continuous and the set $\{T_f|\, f\in B\}$ is equicontinuous. Set $r'_p=r_p/2$, $p\in \ZZ_+$ and denote by $Y_{(r'_p)}$ the space $L^{\infty}_{e^{B_{r'_p}(|\cdot|)}}$. The estimate (\ref{estft}) implies that $T_f$ can be extended by continuity to $X_{(r'_p)}\times X_{(r'_p)}$ and if we denote this extensions by $\tilde{T}_f$ we have
\beq\label{inqbs}
\|\tilde{T}_f(\tilde{\varphi},\tilde{\psi})\|_{Y_{(r'_p)}}\leq C_1\sigma_{(r'_p)}(\tilde{\varphi})\sigma_{(r'_p)}(\tilde{\psi}),\,\, \forall \tilde{\varphi},\tilde{\psi}\in X_{(r'_p)},\, \forall f\in B.
\eeq
For $\varphi\in\SSS^{\{M_p\}}_{\{A_p\}}(\RR^d)$ denote by $F_{f,\varphi}$ the $C^{\infty}$ function $f*\varphi$. We prove that for each $\varphi\in\SSS^{\{M_p\}}_{\{A_p\}}(\RR^d)$, $B_{\varphi}=\{f*\varphi|\, f\in B\}$ is a bounded subset of $Y_{(r'_p)}$. By assumption, $B_{\varphi}$ is bounded in $\SSS'^{\{M_p\}}_{\{A_p\}}(\RR^d)$. Hence Proposition \ref{parametrix1} implies that there exist $(l_p),(k_p)\in\mathfrak{R}$ with $(l_p)\leq (k_p)$ such that $F_{f,\varphi}$ can be extended to $X_{(k_p)}$, $\SSS^{M_p,(l_p)}_{A_p,(l_p)}\subseteq X_{(k_p)}$, the convolution is a continuous bilinear mapping from $\SSS^{M_p,(l_p)}_{A_p,(l_p)}\times\SSS^{M_p,(l_p)}_{A_p,(l_p)}$ to $X_{(k_p)}$ and there exists $u\in X_{(l_p)}$ and $P(D)$ of class $\{M_p\}$ such that $P(D)u=\delta$ and $F_{f,\varphi}=P(D)(u*F_{f,\varphi})$, where $u*F_{f,\varphi}$ is the transpose of the continuous mapping $\psi\mapsto \check{u}*\psi$, $\SSS^{\{M_p\}}_{\{A_p\}}(\RR^d)\rightarrow X_{(k_p)}$. Of course, we can assume that $(k_p)\leq (r'_p)$. Since $\SSS^{\{M_p\}}_{\{A_p\}}(\RR^d)$ is dense in $X_{(l_p)}$, there exist $u_n\in \SSS^{\{M_p\}}_{\{A_p\}}(\RR^d)$, $n\in\ZZ_+$, such that $u_n\rightarrow u$ in $X_{(l_p)}$. Since $*:\SSS^{M_p,(l_p)}_{A_p,(l_p)}\times \SSS^{M_p,(l_p)}_{A_p,(l_p)}\rightarrow X_{(k_p)}$ is continuous,
\beqs
\langle T_f(u_n,\varphi),\psi\rangle={}_{X'_{(k_p)}}\langle F_{f,\varphi},\check{u}_n*\psi\rangle_{X_{(k_p)}}\rightarrow {}_{X'_{(k_p)}} \langle F_{f,\varphi},\check{u}*\psi\rangle_{X_{(k_p)}}=\langle u*F_{f,\varphi},\psi\rangle
\eeqs
for all $\psi\in\SSS^{\{M_p\}}_{\{A_p\}}(\RR^d)$. On the other hand, $T_f(u_n,\varphi)\rightarrow \tilde{T}_f(u,\varphi)$ in $Y_{(r'_p)}$ hence also in $\SSS'^{\{M_p\}}_{\{A_p\}}(\RR^d)$. Thus $\tilde{T}_f(u,\varphi)=u*F_{f,\varphi}$ which in turn implies $F_{f,\varphi}=P(D)(\tilde{T}_f(u,\varphi))$. For $\{u_n\}_{n\in\ZZ_+}$ as before, observe that $P(D)(T_f(u_n,\varphi))=T_f(u_n,P(D)\varphi)$. The right hand side tends to $\tilde{T}_f(u,P(D)\varphi)$ and consequently $F_{f,\varphi}=\tilde{T}_f(u,P(D)\varphi)$. Now (\ref{inqbs}) implies $\|F_{f,\varphi}\|_{Y_{(r'_p)}}\leq C_1\sigma_{(r'_p)}(u)\sigma_{(r'_p)}(P(D)\varphi)$, i.e., $B_{\varphi}$ is bounded in $Y_{(r'_p)}$ for each $\varphi\in\SSS^{\{M_p\}}_{\{A_p\}}(\RR^d)$. But the elements of $B_{\varphi}$ are continuous functions hence $iii)$ holds. Next we prove $iii)\Rightarrow i)$. Fix $\varphi\in\SSS^*_{\dagger}(\RR^d)$. Since $f*\varphi$ is continuous, $|\langle f,\check{\varphi}\rangle|=|f*\varphi(0)|\leq C$ for all $f\in B$, i.e., $B$ is weakly bounded in $\SSS'^*_{\dagger}(\RR^d)$ and thus also strongly bounded.\\
\indent $iv)\Rightarrow iii)$ is trivial. To prove $i)\Rightarrow iv)$, observe that since $B$ is bounded in $\SSS'^*_{\dagger}(\RR^d)$ it is equicontinuous ($\SSS^*_{\dagger}(\RR^d)$ is barrelled). We give the proof in the Roumieu case, the Beurling case is similar. There exist $(r_p)\in\mathfrak{R}$ and $C'>0$ such that $|\langle f,\varphi\rangle|\leq C'\sigma_{(r_p)}(\varphi)$ for all $\varphi\in\SSS^{\{M_p\}}_{\{A_p\}}(\RR^d)$, $f\in B$. Let $(t_p)=(r_p/2)$. Then
\beqs
e^{B_{r_p}(|y|)}\leq 2e^{B_{r_p}(2|x-y|)} e^{B_{r_p}(2|x|)}=2e^{B_{t_p}(|x-y|)} e^{B_{t_p}(|x|)}.
\eeqs
We have
\beqs
|f*\varphi(x)|=|\langle f,\varphi(x-\cdot)\rangle|\leq C'\sup_{\alpha}\sup_{y}\frac{\left|e^{B_{r_p}(|y|)}D^{\alpha}\varphi(x-y)\right|}{M_{\alpha}\prod_{j=1}^{|\alpha|}r_j}\leq C e^{B_{t_p}(|x|)}\sigma_{(t_p)}(\varphi)
\eeqs
for all $x\in\RR^d$, $\varphi\in\SSS^{\{M_p\}}_{\{A_p\}}(\RR^d)$, $f\in B$.
\end{proof}

It is interesting to note that the property $iv)$ turns out to characterize the space $\SSS'^*_{\dagger}(\RR^d)$.

\begin{corollary}\label{cha_s}
Let $f\in \SSS'^{(M_p)}_{(p!)}(\RR^d)$ (resp. $f\in\SSS'^{\{M_p\}}_{\{p!\}}(\RR^d)$). Then $f\in\SSS'^*_{\dagger}(\RR^d)$ if and only if there exists $t>0$ (resp. there exists $(t_p)\in\mathfrak{R}$) such that for every $\varphi\in \SSS^{(M_p)}_{(p!)}(\RR^d)$ (resp. $\varphi\in\SSS^{\{M_p\}}_{\{p!\}}(\RR^d)$)
\beqs
\sup_{x\in\RR^d}e^{-A(t|x|)}|(f*\varphi)(x)|<\infty\,\, \left(\mbox{resp.} \sup_{x\in\RR^d}e^{-B_{t_p}(|x|)}|(f*\varphi)(x)|<\infty\right).
\eeqs
\end{corollary}

\begin{proof} The direct implication follows by Lemma \ref{boundedsetS}. We prove the converse implication only in the Roumieu case, as the Beurling case is similar. For $(r_p)\in\mathfrak{R}$, denote by $\tilde{\sigma}_{(r_p)}$ the norm on $\SSS^{M_p,(r_p)}_{p!,(r_p)}$ to distinguish it from the norm $\sigma_{(r_p)}$ on $\SSS^{M_p,(r_p)}_{A_p,(r_p)}$. Also, for brevity in notation, we denote by $\tilde{X}_{(r_p)}$ the closure of $\SSS^{\{M_p\}}_{\{p!\}}(\RR^d)$ in $\SSS^{M_p,(r_p)}_{p!,(r_p)}$. Let $G_f$ be the continuous mapping $\varphi\mapsto f*\varphi$, $\SSS^{\{M_p\}}_{\{p!\}}(\RR^d)\rightarrow \SSS'^{\{M_p\}}_{\{p!\}}(\RR^d)$. By assumption, the image of $G_f$ is contained in $L^{\infty}_{e^{B_{t_p}(|\cdot|)}}(\RR^d)$. Since the latter space is continuously injected into $\SSS'^{\{M_p\}}_{\{p!\}}(\RR^d)$, the graph of $G_f$ considered as a mapping from $\SSS^{\{M_p\}}_{\{p!\}}(\RR^d)$ to $L^{\infty}_{e^{B_{t_p}(|\cdot|)}}(\RR^d)$ is closed. As $\SSS^{\{M_p\}}_{\{p!\}}(\RR^d)$ is barreled and $L^{\infty}_{e^{B_{t_p}(|\cdot|)}}(\RR^d)$ is a $(B)$-space hence a Pt\'{a}k space (cf. \cite[Sect. IV. 8, p. 162]{Sch}) the Pt\'{a}k closed graph theorem (cf. \cite[Thm. 8.5, p. 166]{Sch}) implies that $G_f:\SSS^{\{M_p\}}_{\{p!\}}(\RR^d)\rightarrow L^{\infty}_{e^{B_{t_p}(|\cdot|)}}(\RR^d)$ is continuous. Thus, there exist $(r_p)\in\mathfrak{R}$ and $C>0$ such that $\left\|f*\varphi\right\|_{L^{\infty}_{e^{B_{t_p}(|\cdot|)}}}\leq C\tilde{\sigma}_{(r_p)}(\varphi)$, i.e., $G_f$ can be extended to a continuous mapping $\tilde{G}_f$ from $\tilde{X}_{(r_p)}$ to $L^{\infty}_{e^{B_{t_p}(|\cdot|)}}$. For $f\in\SSS'^{\{M_p\}}_{\{p!\}}(\RR^d)$ find $(k_p),(l_p)\in\mathfrak{R}$ with $(l_p)\leq (k_p)$, $u\in \tilde{X}_{(l_p)}$ and $P(D)$ of class $\{M_p\}$ as in Proposition \ref{parametrix1}. Then $P(D)(u*\tilde{f})=f$ in $\SSS'^{\{M_p\}}_{\{p!\}}(\RR^d)$, where $\tilde{f}$ is the extension of $f$ to $\tilde{X}_{(k_p)}$. Of course, we can take $(k_p)\leq (r_p)$. Then $u*\tilde{f}=\tilde{G}_f(u)\in L^{\infty}_{e^{B_{t_p}(|\cdot|)}}\subseteq \SSS'^{\{M_p\}}_{\{A_p\}}(\RR^d)$. As $P(D)$ is of class $\{M_p\}$, $f=P(D)(u*\tilde{f})\in\SSS'^{\{M_p\}}_{\{A_p\}}(\RR^d)$.
\end{proof}

\subsection{Nuclearity}\label{nuclearity} We shall now establish the nuclearity of $\SSS^*_{\dagger}(\RR^d)$. We first need to introduce some notation. Given a positive function $\Psi\in C(\RR^{d})$ and  $k\in\NN$ we denote as $C^k_{\Psi}(\RR^d)$ the $(B)$-space of all $k$-times continuously differentiable functions for which the norm $\ds\|\varphi\|_{\Psi,k}=\sup_{|\alpha|\leq k}\left\|\Psi D^{\alpha}\varphi\right\|_{L^{\infty}(\RR^d)}$ is finite.  Let $\mathbf{I}=(-1,1)^d$ and denote by $C_{\mathbf{I}}^k(\RR^d)$ the space of all $k$-times continuously differentiable periodic functions with period $\mathbf{I}$. Equipped with the norm $\sup_{|\alpha|\leq k}\|D^{\alpha}\varphi\|_{L^{\infty}(\mathbf{I})}$ it becomes a $(B)$-space. We denote by $C_{L^{\infty}}(\RR^d)$ the $(B)$-space of all bounded continuous functions on $\RR^d$ with the supremum norm. Finally, we also employ the notation $\langle x \rangle=(1+|x|^{2})^{1/2}$.

\begin{lemma}\label{nuccp} Suppose that $\lim_{|x|\to\infty} \langle x \rangle^{j} /\Psi(x)=0$, for all $j\in\mathbb{Z}_{+}$. Then, the inclusion mapping $C^{d+1}_{\Psi}(\RR^d)\rightarrow C_{L^{\infty}}(\RR^d)$ is nuclear.
\end{lemma}

\begin{proof} Consider the function $f:(-1,1)\rightarrow \RR$, defined by $f(t)=(1-t)^{-1}-(1+t)^{-1}$. It is a diffeomorphism between $(-1,1)$ and $\RR$ with inverse $f^{-1}(t)=(\langle t\rangle-1)/t$ (notice that $f^{-1}$ is analytic on a strip along the real axis since the singularity at $0$ is removable). Denote by $g:\mathbf{I}\rightarrow\RR^d$ the mapping $g(x)=(f(x_1),...,f(x_d))$. It is a diffeomorphism between $\mathbf{I}$ and $\RR^d$ with inverse $g^{-1}(y)=\left((\langle y_1\rangle-1)/y_1,...,(\langle y_d\rangle-1)/y_d\right)$. For $\varphi\in C^{d+1}_{\Psi}(\RR^d_y)$ consider the function $\varphi\circ g$. Clearly it is an element of $C^{d+1}(\mathbf{I})$. For $\kappa\in\NN^d$, we denote $\partial^{\kappa}_xg(x)=(\partial^{\kappa}_xf(x_1),...,\partial^{\kappa}_xf(x_d))$. Using the Fa\'a di Bruno formula (see \cite[Theorem 2.1]{faadib}) one can see that for $1\leq |\alpha|\leq d+1$, $\partial^{\alpha}_x(\varphi\circ g)(x)$ is a finite sum of terms of the form $const.\cdot(\partial^{\beta}\varphi)\circ g(x)\prod_{j=1}^{|\alpha|} \left(\partial^{\gamma^{(j)}}_xg(x)\right)^{\nu^{(j)}}$, where $1\leq |\beta|\leq |\alpha|$, $\gamma^{(j)}=\nu^{(j)}=0$ for $1\leq j\leq |\alpha|-s$ and $\gamma^{(j)}\neq0\neq\nu^{(j)}$ for $|\alpha|-s+1\leq j\leq |\alpha|$ and $\sum_{j=1}^{|\alpha|}\nu^{(j)}=\beta$ and $\sum_{j=1}^{|\alpha|}\left|\nu^{(j)}\right|\gamma^{(j)}=\alpha$ (as in \cite[Theorem 2.1]{faadib} we set $0^0=1$). Now, observe that $f^{(k)}(t)=k!\left((1-t)^{-k-1}+(-1)^{k+1}(1+t)^{-k-1}\right)$. For $x\in\mathbf{I}$, let $y_j=f(x_j)$, $j=1,...,d$. Since
\beqs
\lim_{y_j\rightarrow 0}\left(1-\frac{\langle y_j\rangle-1}{y_j}\right)^{-1}=1\,\, \mbox{and}\,\, \lim_{y_j\rightarrow 0} \left(1+\frac{\langle y_j\rangle-1}{y_j}\right)^{-1}=1
\eeqs
there exists $0<\varepsilon<1$ such that for $|y_j|\leq \varepsilon$, $j=1,...,d$,
\beqs
\frac{|y_j|}{\left|y_j-\langle y_j\rangle+1\right|}\leq 2\,\, \mbox{and}\,\, \frac{|y_j|}{\left|y_j+\langle y_j\rangle-1\right|}\leq 2.
\eeqs
For $y_j\geq \varepsilon$ we have $\left|y_j-\langle y_j\rangle+1\right|=y_j+1-\langle y_j\rangle\geq \varepsilon/(2\langle y_j\rangle)$ and $\left|y_j+\langle y_j\rangle-1\right|\geq \varepsilon$. Similarly, for $y_j\leq -\varepsilon$, we have $\left|y_j-\langle y_j\rangle+1\right|\geq \varepsilon$ and $\left|y_j+\langle y_j\rangle-1\right|\geq \varepsilon/(2\langle y_j\rangle)$. Thus, for $k\in\ZZ_+$,
\beqs
\frac{|y_j|^k}{\left|y_j-\langle y_j\rangle+1\right|^k}\leq c^k\langle y_j\rangle^{2k}\,\, \mbox{and}\,\, \frac{|y_j|^k}{\left|y_j+\langle y_j\rangle-1\right|^{k}}\leq c^k\langle y_j\rangle^{2k}
\eeqs
for all $y_j\in\RR$, where $c=2(1+\varepsilon^{-1})$. We obtain
\beqs
\left|f^{(k)}(x_j)\right|\leq \frac{k!|y_j|^{k+1}}{\left|y_j-\langle y_j\rangle+1\right|^{k+1}}+\frac{k!|y_j|^{k+1}}{\left|y_j+\langle y_j\rangle-1\right|^{k+1}}\leq 2c^{k+1}k!\langle y_j\rangle^{2(k+1)}.
\eeqs
Now, we estimate as follows
\beqs
\prod_{j=1}^{|\alpha|} \left|\left(\partial^{\gamma^{(j)}}_xg(x)\right)^{\nu^{(j)}}\right|&\leq& \prod_{j=1}^{|\alpha|}\prod_{k=1}^d\left|\partial^{\gamma^{(j)}_k}_{x_k}f(x_k)\right|^{\nu^{(j)}_k}\leq C_1\prod_{j=1}^{|\alpha|}\prod_{k=1}^d \langle y_k\rangle^{2\gamma^{(j)}_k \nu^{(j)}_k+2\nu^{(j)}_k}\\
&\leq& C_1\prod_{j=1}^{|\alpha|} \langle y\rangle^{2\left|\gamma^{(j)}\right| \left|\nu^{(j)}\right|+2\left|\nu^{(j)}\right|}=C_1\langle y\rangle^{2|\alpha|+2|\beta|}\\
&\leq& C_1\langle y\rangle^{4(d+1)}.
\eeqs
Hence\\
\\
$\ds \sup_{|\alpha|\leq d+1}\|D^{\alpha}(\varphi\circ
g)\|_{L^{\infty}(\mathbf{I})}$
\beq\label{1155779}
\leq C_2
\sup_{|\alpha|\leq d+1} \sup_{y\in\RR^d}\langle
y\rangle^{4(d+1)}\left|D^{\alpha}\varphi(y)\right|\leq C_3
\|\varphi\|_{\Psi,d+1}.
\eeq
Let $x\in \mathbf{I}$ and set
$y=g(x)$. If $x_j\geq 0$ then \beqs |1-x_j|^{-1}\leq
|f(x_j)|+|1+x_j|^{-1}\leq |y_j|+1\leq 2\langle y_j\rangle. \eeqs
If $x_j\leq 0$, clearly $|1-x_j|^{-1}\leq 1\leq 2\langle y_j\rangle$.
Similarly, $|1+x_j|^{-1}\leq 2\langle y_j\rangle$. Thus, denoting
$\mathbf{1}=(1,...,1)\in\RR^d$, for $\beta,\gamma\in\NN^d$ fixed,
by the above estimates we have \beqs
\sup_{|\alpha|\leq d+1} \sup_{x\in\mathbf{I}}\left|\frac{D^{\alpha}(\varphi\circ g) (x)}{(\mathbf{1}-x)^{\beta}(\mathbf{1}+x)^{\gamma}}\right|&\leq& C_4\sup_{|\alpha|\leq d+1}\sup_{y\in\RR^d}\left|\langle y\rangle^{4(d+1)+|\beta|+|\gamma|}D^{\alpha}\varphi(y)\right|\\
&\leq& C_5\|\varphi\|_{\Psi,d+1}.
\eeqs
For $0<l< 1$ denote $\mathbf{I}_l=(-l,l)^d$. The above estimate implies that for each $\beta,\gamma\in \NN^d$ and $\varepsilon>0$ there exists $0<l<1$ such that
\beqs
\sup_{|\alpha|\leq d+1} \sup_{x\in\mathbf{I}\backslash \mathbf{I}_l}\left|(\mathbf{1}-x)^{-\beta}(\mathbf{1}+x)^{-\gamma}D^{\alpha}(\varphi\circ g) (x)\right|\leq \varepsilon.
\eeqs
Hence $\varphi\circ g$ can be extended to a periodic function on $\RR^d$ with period $\mathbf{I}$ and this extension is an element of $C^{d+1}_{\mathbf{I}}(\RR^d_x)$. Thus, the mapping $T$ which maps $\varphi\in C^{d+1}_{\Psi}(\RR^d_y)$ to the periodic extension of $\varphi\circ g$ with period $\mathbf{I}$ is well defined as a mapping from $C^{d+1}_{\Psi}(\RR^d_y)$ into $C_{\mathbf{I}}^{d+1}(\RR^d_x)$. Moreover, by (\ref{1155779}), it is continuous.\\
\indent Next, by using a similar technique as in the proof of \cite[Lemma 2.3]{Komatsu1}, one shows that the inclusion mapping $C_{\mathbf{I}}^{d+1}(\RR^d)\rightarrow C_{\mathbf{I}}^0(\RR^d)$ is nuclear. Denote by $\tilde{T}$ the mapping $\varphi\mapsto \varphi\circ g^{-1}$, $C_{\mathbf{I}}^0(\RR^d)\rightarrow C_{L^{\infty}}(\RR^d)$. To prove that it is well defined and continuous, observe that $\varphi\circ g^{-1}$ is continuous and
\beqs
|\varphi\circ g^{-1}(y)|=\left|\varphi\left((\langle y_1\rangle-1)/y_1,...,(\langle y_d\rangle -1)/y_d\right)\right|\leq \|\varphi\|_{L^{\infty}(\mathbf{I})}
\eeqs
for all $y\in\RR^d$.\\
\indent Now, notice that the inclusion mapping $\iota:C^{d+1}_{\Psi}(\RR^d)\rightarrow C_{L^{\infty}}(\RR^d)$ can be decomposed as $\ds C^{d+1}_{\Psi}(\RR^d)\xrightarrow{T} C^{d+1}_{\mathbf{I}}(\RR^d)\xrightarrow{\iota} C^0_{\mathbf{I}}(\RR^d)\xrightarrow{\tilde{T}} C_{L^{\infty}}(\RR^d)$, which, by the above observation, proves its nuclearity.
\end{proof}

In the rest of this subsection, we are interested in the special case $\Psi(x)=e^{A(h|x|)}$, $h>0$. We recall  (\cite[Definition 3.2.3, p. 56]{pietsch}) that a continuous mapping
 $T:E\rightarrow F$,
$E$ and $F$ being normed spaces, is called quasi-nuclear if there exists a sequence of functionals $x'_j\in E'$, $j\in\ZZ_+$, such that $\sum_j \|x'_j\|_{E'}<\infty$ and $\|T(x)\|_F\leq \sum_j |\langle x'_j,x\rangle|$ for all $x\in E$, as well as
(\cite[Theorem 3.3.2, p. 62]{pietsch}) that the composition of two quasi-nuclear mappings is nuclear.

\begin{proposition}\label{nucin} We have:

\indent $i)$ For every $h>0$ there exists $h_1>h$ such that the inclusion mapping $\SSS^{M_p,h_1}_{A_p,h_1}\rightarrow\SSS^{M_p,h}_{A_p,h}$ is nuclear.\\
\indent $ii)$ For every $h>0$ there exists $h_1<h$ such that the inclusion mapping $\SSS^{M_p,h}_{A_p,h}\rightarrow\SSS^{M_p,h_1}_{A_p,h_1}$ is nuclear.
\end{proposition}

\begin{proof} For $h>0$ we denote by $X_h$ the space $\SSS^{M_p,h}_{A_p,h}$. Since the inclusion $X_{h_1}\rightarrow X_h$ (resp. $X_h\rightarrow X_{h_1}$), is the composition of two inclusions of the same type, it is enough to prove that it is quasi-nuclear. For ease of writing, set $Y_h=C^{d+1}_{e^{A(h|\cdot |)}}(\RR^d)$. We first need to construct $F\in C^{d+1}(\mathbb{R}^{d})$ such that
\begin{equation}\label{extraeq1}
c_1 e^{A(|x|)}\leq |F(x)|\leq C_1 e^{A(4|x|)},\,\,\, x\in\RR^{d},
\end{equation}
and
\begin{equation}\label{extraeq2}
|D^{\alpha}F(x)|\leq C_1 e^{A(4|x|)},\,\,\, x\in\RR^{d}, \: |\alpha|\leq d+1.
\end{equation}
Indeed, choose an even non-negative $\phi\in\mathcal{D}(\mathbb{R}^{d})$ and set $\displaystyle F(x)=(\phi* e^{A(2|\:\cdot\:|)})(x)$; then, from the elementary inequality $\displaystyle e^{A(\rho+\lambda)}\leq 2e^{A(2\rho)}e^{A(2\lambda)}$, we obtain (\ref{extraeq1}) and (\ref{extraeq2}) with
\beqs
c_1=\frac{1}{2}\int_{\mathbb{R}^{d}}\phi(y)e^{-A(2|y|)}dy\,\,\, \mbox{and}\,\,\, C_1=2 \max_{|\alpha|\leq d+1}\int_{\mathbb{R}^{d}}|D^{\alpha}\phi(y)|e^{A(4|y|)}dy.
\eeqs
Let $h>0$. For the proof of $i)$ define $F_h(x)=F(hx)$, resp. for the proof of $ii)$ define $F_h(x)=F(hx/(4H))$. Then $c_1 e^{A(h|x|)}\leq |F_h(x)|\leq C_1 e^{A(4h|x|)}$, resp. $c_1 e^{A(h|x|/(4H))}\leq |F_h(x)|\leq C_1 e^{A(h|x|/H)}$ and (\ref{extraeq2}) implies (for $|\alpha|\leq d+1$)
\begin{equation}\label{grw11}
 \left|D^{\alpha}F_h(x)\right|\leq C_1h^{|\alpha|} e^{A(4h|x|)},\, \mbox{resp.}\, \left|D^{\alpha}F_h(x)\right|\leq C_1\left(\frac{h}{4H}\right)^{|\alpha|}e^{A(h|x|/H)}.
\end{equation}
First we prove $i)$. Take $h_1=4Hh$. We prove that the inclusion $X_{h_1}\rightarrow X_h$ is quasi-nuclear. By Lemma \ref{nuccp}, there exists $S_j\in Y'_{2h}$, $j\in\ZZ_+$, such that $\sum_j \|S_j\|_{Y'_{2h}}<\infty$ and $\chi_j\in C_{L^{\infty}}(\RR^d)$, $j\in\ZZ_+$, with $\|\chi_j\|_{L^{\infty}(\RR^d)}\leq 1$ for all $j\in\ZZ_+$, such that for every $\psi\in Y_{2h}$, $\psi(x)=\sum_j S_j(\psi)\chi_{j}(x)$. Let $\varphi\in X_{h_1}$. For each $\alpha\in\NN^d$, $F_hD^{\alpha}\varphi\in Y_{2h}$. Indeed, for $|\beta|\leq d+1$, by (\ref{grw11}), we have\\
\\
$\left|e^{A(2h|x|)}D^{\beta}\left(F_h(x)D^{\alpha}\varphi(x)\right)\right|$
\beqs
&\leq& \sum_{\gamma\leq \beta} {\beta\choose\gamma} e^{A(2h|x|)}\left|D^{\beta-\gamma}F_h(x)\right| \left|D^{\alpha+\gamma}\varphi(x)\right|\\
&\leq& C_2\sum_{\gamma\leq \beta} {\beta\choose\gamma} e^{2A(4h|x|)} \left|D^{\alpha+\gamma}\varphi(x)\right|\leq C_3\|\varphi\|_{X_{h_1}}\sum_{\gamma\leq \beta} {\beta\choose\gamma} \frac{M_{\alpha+\gamma}}{h_1^{|\alpha|+|\gamma|}}\\
&\leq& C_4(1/(4h))^{|\alpha|}M_{\alpha}\|\varphi\|_{X_{h_1}},
\eeqs
where, in the third inequality, we have used $e^{2A(4h|x|)}\leq c_0 e^{A(4Hh|x|)}=c_0e^{A(h_1|x|)}$ (which follows from $(M.2)$ for $A_p$; see \cite[Proposition 3.6]{Komatsu1}). Observe that
\beqs
\left\|e^{A(h|\cdot|)}D^{\alpha}\varphi\right\|_{L^{\infty}(\RR^d)}\leq C_0\left\|F_hD^{\alpha}\varphi\right\|_{L^{\infty}(\RR^d)}\leq \sum_{j=1}^{\infty} C_0\left|S_j(F_hD^{\alpha}\varphi)\right|.
\eeqs
Denote by $T_{\alpha,j}$ the linear functional on $X_{h_1}$ defined by $T_{\alpha,j}(\varphi)=C_0h^{|\alpha|}S_j(F_hD^{\alpha}\varphi)/M_{\alpha}$. Observe that
\beqs
\left|T_{\alpha,j}(\varphi)\right|\leq C_0\frac{h^{|\alpha|}\|S_j\|_{Y'_{2h}}}{M_{\alpha}}\cdot \sup_{|\beta|\leq d+1} \left\|e^{A(2h|\cdot|)}D^{\beta}\left(F_hD^{\alpha}\varphi\right)\right\|_{L^{\infty}(\RR^d)}\leq C_6\frac{\|S_j\|_{Y'_{2h}}\|\varphi\|_{X_{h_1}}}{4^{|\alpha|}}.
\eeqs
Hence $\left\|T_{\alpha,j}\right\|_{X'_{h_1}}\leq C_6 4^{-|\alpha|} \|S_j\|_{Y'_{2h}}$. We conclude $\sum_{j,\alpha}\left\|T_{\alpha,j}\right\|_{X'_{h_1}}<\infty$ and together with the inequality
\beqs
\|\varphi\|_{X_h}&\leq& \sum_{\alpha}\frac{h^{|\alpha|}\left\|e^{A(h|\cdot|)}D^{\alpha}\varphi\right\|_{L^{\infty}(\RR^d)}}{M_{\alpha}}\leq \sum_{\alpha}\frac{C_0h^{|\alpha|}\left\|F_hD^{\alpha}\varphi\right\|_{L^{\infty}(\RR^d)}}{M_{\alpha}}\\
&\leq&\sum_{j,\alpha}\left|T_{\alpha,j}(\varphi)\right|,
\eeqs
proves the desired quasi-nuclearity.\\
\indent For $ii)$, we take $h_1=h/(4H)$ and show that the inclusion $X_h\rightarrow X_{h_1}$ is quasi-nuclear. By Lemma \ref{nuccp}, there exists $S_j\in Y'_{h/(2H)}$, $j\in\ZZ_+$, such that $\sum_j \|S_j\|_{Y'_{h/(2H)}}<\infty$ and $\chi_j\in C_{L^{\infty}}(\RR^d)$, $j\in\ZZ_+$, with $\|\chi_j\|_{L^{\infty}(\RR^d)}\leq 1$ for all $j\in\ZZ_+$, such that for every $\psi\in Y_{h/(2H)}$, $\psi(x)=\sum_j S_j(\psi)\chi(x)$. Let $\varphi\in X_h$. Similarly as above, one proves that for $\alpha,\beta\in\NN^d$, such that $|\beta|\leq d+1$,
\beqs
\left|e^{A(h|x|/(2H))}D^{\beta}\left(F_h(x)D^{\alpha}\varphi(x)\right)\right|\leq C'_3(H/h)^{|\alpha|}M_{\alpha}\|\varphi\|_{X_h}.
\eeqs
Thus, for each $\alpha\in\NN^d$, $F_hD^{\alpha}\varphi\in Y_{h/(2H)}$. Observe that
\beqs
\left\|e^{A(h_1|\cdot|)}D^{\alpha}\varphi\right\|_{L^{\infty}(\RR^d)}\leq C_0\left\|F_hD^{\alpha}\varphi\right\|_{L^{\infty}(\RR^d)}\leq \sum_{j=1}^{\infty} C_0\left|S_j(F_hD^{\alpha}\varphi)\right|.
\eeqs
Set $T_{\alpha,j}(\varphi)=C_0h_1^{|\alpha|}S_j(F_hD^{\alpha}\varphi)/M_{\alpha}$. Similarly as above one verifies that $\left\|T_{\alpha,j}\right\|_{X'_h}\leq C'_4 4^{-|\alpha|} \|S_j\|_{Y'_{h/(2H)}}$, i.e., $\sum_{j,\alpha}\left\|T_{\alpha,j}\right\|_{X'_h}<\infty$ and $\|\varphi\|_{X_{h_1}}\leq \sum_{j,\alpha}\left|T_{\alpha,j}(\varphi)\right|$ which proves the quasi nuclearity of the inclusion $X_h\rightarrow X_{h_1}$.
\end{proof}

In particular we obtain:

\begin{proposition}\label{propnuclear}
The space $\SSS^*_{\dagger}(\RR^d)$ is nuclear.
\end{proposition}

Since $\SSS^{(M_p)}_{(A_p)}(\RR^d)$ is an $(FN)$-space (Fr\'{e}chet nuclear) and $\SSS^{\{M_p\}}_{\{A_p\}}(\RR^d)$ is a $(DFN)$-space, $\SSS'^*_{\dagger}(\RR^d)$ is also nuclear. As a consequence of the nuclearity of $\SSS^*_{\dagger}(\RR^d)$ we have the following proposition. We will often tacitly apply it through the rest of the article.

\begin{proposition}\label{proptensorpi} The following canonical isomorphisms of l.c.s. hold:
\beqs
\SSS^*_{\dagger}(\RR^{d_1+d_2})\cong \SSS^*_{\dagger}(\RR^{d_1})\hat{\otimes}\SSS^*_{\dagger}(\RR^{d_2})\cong \mathcal{L}_b\left(\SSS'^*_{\dagger}(\RR^{d_1}),\SSS^*_{\dagger}(\RR^{d_2})\right),\\
\SSS'^*_{\dagger}(\RR^{d_1+d_2})\cong \SSS'^*_{\dagger}(\RR^{d_1})\hat{\otimes}\SSS'^*_{\dagger}(\RR^{d_2})\cong \mathcal{L}_b\left(\SSS^*_{\dagger}(\RR^{d_1}),\SSS'^*_{\dagger}(\RR^{d_2})\right).
\eeqs
\end{proposition}

\begin{proof} For brevity in the notation set $d=d_1+d_2$. If we prove the density of $\SSS^*_{\dagger}(\RR^{d_1})\otimes\SSS^*_{\dagger}(\RR^{d_2})$ in $\SSS^*_{\dagger}(\RR^d)$ the rest of the proof is analogous to the proof of \cite[Proposition 2]{BojanP}. Let $\psi\in\SSS^*_{\dagger}(\RR^d)$. For $\{\chi_n\}_{n\in\ZZ_+}\subseteq\SSS^*_{\dagger}(\RR^{d_1})$ and $\{\tilde{\chi}_n\}_{n\in\ZZ_+}\subseteq\SSS^*_{\dagger}(\RR^{d_2})$ as in $ii)$ of Lemma \ref{appincl}, with a similar argument as in the proof of Lemma \ref{appincl}, one readily verifies that $(\chi_n\otimes\tilde{\chi}_n)*\psi\rightarrow \psi$ in $\SSS^*_{\dagger}(\RR^d)$. Hence it is enough to prove that for arbitrary but fixed $\theta_j\in\SSS^*_{\dagger}(\RR^{d_j})$, $j=1,2$, $(\theta_1\otimes\theta_2)*\psi$ can be approximated by elements of $\SSS^*_{\dagger}(\RR^{d_1})\otimes\SSS^*_{\dagger}(\RR^{d_2})$. We give the proof for the Roumieu case, as the Beurling case is is similar. Set $\theta=\theta_1\otimes \theta_2$. Let $\varepsilon>0$ and $(r_p)\in\mathfrak{R}$. Denote $(k_p)=(r_p/2)$. Since $\DD(\RR^d)$ is dense in $L^1_{e^{B_{k_p}(|\cdot|)}}(\RR^d)$ and $\DD(\RR^{d_1})\otimes\DD(\RR^{d_2})$ is dense in $\DD(\RR^d)$ there exists $g\in\DD(\RR^{d_1})\otimes\DD(\RR^{d_2})$ such that $\|\psi-g\|_{L^1_{e^{B_{k_p}(|\cdot|)}(\RR^d)}}\leq \varepsilon/(2\sigma_{(k_p)}(\theta))$. One easily verifies that $\theta*g\in\SSS^{\{M_p\}}_{\{A_p\}}(\RR^{d_1})\otimes \SSS^{\{M_p\}}_{\{A_p\}}(\RR^{d_2})$. By using the inequality $e^{B_{r_p}(\rho+\lambda)}\leq 2e^{B_{r_p}(2\rho)}e^{B_{r_p}(2\lambda)}$, $\rho,\lambda>0$, we have $e^{B_{r_p}(|x|)}\leq 2e^{B_{k_p}(|x-y|)}e^{B_{k_p}(|y|)}$, $\forall x,y\in\RR^d$. Thus\\
\\
$\ds\frac{\left|D^{\alpha}\left(\theta*\psi-\theta*g\right)(x)\right|e^{B_{r_p}(|x|)}}{M_{\alpha}R_{\alpha}}$
\beqs
\leq 2\int_{\RR^d}\frac{\left|D^{\alpha}\theta(x-y)\right|e^{B_{k_p}(x-y)}}{M_{\alpha}R_{\alpha}}\cdot |\psi(y)-g(y)|e^{B_{k_p}(|y|)}dy\leq \varepsilon.
\eeqs
\end{proof}

\begin{remark} As inspection in its proof shows Proposition \ref{nucin} (and hence Propositions \ref{propnuclear} and \ref{proptensorpi}) remains valid if we merely assume that $A_p$ and $M_p$ satisfy $(M.1)$ and $(M.2)$.
\end{remark}

\section{Translation-invariant spaces of ultradistributions}\label{TIBU}
This section collects and explains various results about some classes of translation-invariant spaces of ultradistributions that we shall apply in the next two sections in our study of the convolution. See \cite{DPPV} for the non-quasianalytic case and  \cite{dpv-2} for the quasianalytic one (see also \cite{DPV} for the distribution case). We employ the notation $T_h$ for the translation operator $T_hg=g(\:\cdot\:+h), h\in\mathbb R^d$.

\subsection{Translation-invariant Banach spaces of ultradistribution}
We start by defining the following class of Banach spaces:

\begin{definition}\label{def E}
A $(B)$-space $E$ is said to be a translation-invariant $(B)$-space of ultradistributions of class $*-\dagger$ if it satisfies the following three axioms:
\begin{itemize}
    \item[(I)] $\mathcal{S}^*_{\dagger}(\mathbb{R}^d)\hookrightarrow E\hookrightarrow \mathcal{S}'^*_{\dagger}(\RR^d)$.
    \item[(II)] $T_{h}(E)\subseteq E$ for each $h\in\RR^d$ (i.e., $E$ is translation-invariant).
\item [(III)] There exist $\tau,C>0$ (for every $\tau>0$ there exists $C>0$), such that\footnote{Applying the closed graph theorem, the axioms (I) and (II) yield $T_h\in\mathcal{L}(E)$ for every $h\in\RR^d$, see \cite[Lemma 3.1]{dpv-2}}
    \begin{equation}\label{omega}
   \omega(h):= \|T_{-h}\|_{\mathcal{L}(E)}\leq C e^{A(\tau|h|)}.
   \end{equation}
\end{itemize}
The function $\omega:\mathbb{R}^{d}\to (0,\infty)$ is called the weight (or growth) function of $E$.
\end{definition}

These axioms imply \cite{DPPV,dpv-2} the following important property:  \\
\indent ($\widetilde{II}$) The mappings $h\mapsto T_h g$ are continuous for each $g\in E$,\\
i.e., the translation group of $E$ is a $C_0$-group; moreover, $E$ is separable.

In the rest of the section we assume that $E$ and $\omega$ are as
in Definition \ref{def E}. It should be noticed that the weight
function of $E$ is measurable, $\omega(0)=1$, and $\log \omega$ is
subadditive. We associate to $E$ the Beurling algebra
$L^{1}_{\omega}$, this convolution algebra is very important to
understand the properties of $E$. The next proposition collects
some useful results from \cite{dpv-2} concerning the natural
convolution structure on $E$.

\begin{proposition}[\cite{dpv-2}]\label{corBeurling Algebra} The $(B)$-space $E$ satisfies:

a) The convolution  mapping $ \mathcal{S}^*_{\dagger}(\mathbb{R}^d)\times \mathcal{S}^*_{\dagger}(\mathbb{R}^d)\rightarrow \mathcal{S}^*_{\dagger}(\mathbb{R}^d) $ extends as a continous bilinear mapping  $L^{1}_{\omega}\times E\rightarrow E$ so that $E$ becomes a Banach module over the Beurling
algebra $L^{1}_{\omega}$, i.e., $ \|u\ast g\|_{E}\leq \|u\|_{1,\omega}\|g\|_{E}$.

b) Let $g\in E$, $\chi\in \mathcal{S}^*_{\dagger}(\mathbb{R}^d)$ with $\displaystyle\int_{\mathbb{R}^{d}}\chi(x)dx$=1, and define $\chi_n(x)=n^d\chi(nx)$, $n\in\ZZ_+$ . Then, $ \displaystyle \lim_{n\to\infty} \|g-\chi_{n}*g\|_E=0$.

\end{proposition}

The dual space $E'$ carries a convolution structure as well.
Indeed, we can associate the Beurling algebra
$L^{1}_{\check{\omega}}$ to $E'$ (here
$\check{\omega}(x)=\omega(-x)$) and the convolution of $f\in E'$
and $u\in L^{1}_{\check{\omega}}$ is  defined via transposition: $
\left\langle u\ast f,g \right\rangle:= \left\langle f,\check
{u}\ast g\right\rangle,$ $g\in E. $ The space $E'$ then becomes a
Banach module over $L^{1}_{\check{\omega}}$.

It is important to notice that, in general, $E'$ is not a
translation-invariant $(B)$-space of ultradistributions of class
$*-\dagger$. Indeed, the properties (I) and ($\widetilde{II}$) may
fail for $E'$ (e.g., take $E=L^{1}$).

We have introduced in \cite{DPV,DPPV, dpv-2} the space
$E'_{\ast}=L^{1}_{\check{\omega}}\ast E'$. Since the Beurling
algebra $L^{1}_{\check{\omega}}$ admits bounded approximation
unities, it follows from the celebrated Cohen-Hewitt factorization
theorem \cite{kisynski} that $E'_{\ast}$ is actually a closed
linear subspace of $E'$. Thus, $E'_{\ast}$ inheres the Banach
module structure over $L^{1}_{\check{\omega}}$.  The $(B)$-space
of ultradistributions $E'_{\ast}$ possesses richer properties than
$E'$; indeed, it satisfies ($\tilde{II}$) and $b)$ from
Proposition \ref{corBeurling Algebra}; moreover, we have the
explicit description \cite{DPPV,dpv-2} $E'_{\ast}=\{f\in E':
\lim_{h\to 0}\|T_{h}f-f\|_{E'}=0\}$. When $E$ is reflexive, we
have proved \cite{DPPV, dpv-2} that $E'$ is also a
translation-invariant Banach space of ultradistributions of class
$*-\dagger$
and in fact $E'=E'_{\ast}$. We now give some typical examples of
$E$.

\begin{example} \label{ex1}
Let $\eta$ be a \emph{ultrapolynomially bounded weight function of class $\dagger$}, that is, a continous function $\eta:\mathbb{R}^d\rightarrow (0,\infty)$ that fulfills the requirement $\eta(x+h)\leq C\eta(x)e^{A(\tau|h|)}$ for some $C,\tau>0$, resp. for every $\tau>0$ there exists $C>0$).
 One clearly has that $E=L^{p}_{\eta}$ are translation-invariant Banach spaces of tempered ultradistributions for $p\in[1,\infty)$. The case $p=\infty$ is an exception, because the properties (I) fails for $L^{\infty}_{\eta}$. In view of reflexivity, the space $E'_{\ast}$ corresponding to $E=L^{p}_{\eta^{-1}}$ is $E'_{\ast}=E'=L^{q}_{\eta}$ whenever $1<p<\infty$, where $q$ is the conjugate index to $p$. On the other hand, $E'_{\ast}=UC_{\eta}:= \left\{u\in L^{\infty}_{\eta}: \lim_{h\to0}||T_{h}u-u||_{\infty,\eta}=0 \right\}$ for $E=L^{1}_{\eta}$. The weight function of $L^{p}_{\eta}$ is $\omega(h)= \|T_{h}\eta\|_{\infty,\eta}$ for $p\in[1,\infty)$, while that for $L^{\infty}_{\eta}$ is $\check{\omega}$ (cf. \cite[Prop. 10]{DPV}). Another instance is $E=C_{\eta}$ where $C_{\eta}=\{\varphi\in C(\RR^d)|\, \lim_{|x|\rightarrow\infty} \varphi(x)/\eta(x)=0\}$ with norm $\|\cdot\|_{\infty,\eta}$; in this case, one readily verifies that $E'_{\ast}=L^{1}_{\eta}$.
\end{example}

\subsection{The spaces $\mathcal{D}^*_{E}$ and $\mathcal{D}'^*_{E'_{\ast}}$}\label{tspace}
Following \cite{DPPV,dpv-2}, we introduce the test function space $\mathcal{D}^*_{E}$ as follows.
Let $$\mathcal{D}^{M_p,m}_{E}=\left\{\varphi\in E\Big|\, D^{\alpha}\varphi\in E, \forall \alpha\in\NN^d, \|\varphi\|_{E,m}=\sup_{\alpha\in\NN^d}\frac{m^\alpha \|D^\alpha\varphi\|_E}{M_\alpha}<\infty\right\}.$$  We define the (Hausdorff) l.c.s.
\begin{equation} \label{aldef1}
\mathcal{D}_E^{(M_p)}=\lim_{\substack{\longleftarrow\\ m\rightarrow\infty}} \mathcal{D}_E^{M_p,m},\,\,\,\mathcal{D}_E^{\{M_p\}}=\lim_{\substack{\longrightarrow\\ m\rightarrow 0}} \mathcal{D}_E^{M_p,m}.
\end{equation}
An important result from  \cite{DPPV, dpv-2} is the alternative description of the Roumieu case via
the projective family of $(B)$-spaces
\begin{equation}
\label{aldef2}
\mathcal{D}_E^{\{M_p\},(r_p)}=\left\{\varphi \in E\Bigg|\, D^{\alpha}\varphi\in E, \forall\alpha\in\NN^d, \|\varphi\|_{E,(r_p)}=\sup_{\alpha}\frac{\|D^\alpha\varphi\|_E}{M_\alpha \prod_{j=1}^{|\alpha|}r_j }<\infty\right\},
\end{equation}
with $(r_p)\in\mathfrak{R}$; we have the (topological) equality $\ds\mathcal{D}_E^{\{M_p\}}=\lim_{\substack{\longleftarrow\\ (r_p)\in \mathfrak{R}}}\mathcal{D}_E^{\{M_p\},(r_p)}$.

We recall some properties of $\mathcal{D}^*_{E}$ that are shown in
\cite{dpv-2}. We have shown that the elements of
$\mathcal{D}^*_{E}$  are actually ultradifferentiable
functions, namely, $\mathcal{D}^*_{E}\subseteq
\mathcal{E}^*(\RR^{d})$ and the inclusion is continuous. Moreover, one has the dense embeddings
$\mathcal{S}^*_{\dagger}(\mathbb{R}^d)\hookrightarrow\mathcal{D}^{*}_E\hookrightarrow E\hookrightarrow\mathcal{S}'^*_{\dagger}(\mathbb{R}^d)$. Clearly, ultradifferential operators of class $*$ act continuously on $\mathcal{D}^*_{E}$.
Furthermore, $\mathcal{D}^*_{E}$ is a topological algebra on
$L^{1}_{\omega}$ under convolution. Obviously,
$\mathcal{D}_E^{(M_p)}$ is an $(F)$-space and
$\mathcal{D}_E^{\{M_p\}}$ is a barreled, bornological
$(DF)$-space. Moreover, $\DD^{\{M_p\}}_E$ is complete and regular,
i.e., every bounded set $B$ in $\mathcal{D}_E^{\{M_p\}}$ is
bounded in some $\mathcal{D}_E^{\{M_p\},m}$. The space
$\mathcal{D}^*_{E}$ is reflexive if $E$ is reflexive.

 The strong dual of $\mathcal{D}^*_{E}$ will be denoted as $\mathcal{D}'^*_{E'_{\ast}}$. Clearly, $\mathcal{D}'^{(M_p)}_{E'_{\ast}}$ is a complete $(DF)$-space and $\mathcal{D}'^{\{M_p\}}_{E'_{\ast}}$ is an $(F)$-space. When $E$ is reflexive, we write $\mathcal{D}'^*_{E'}=\mathcal{D}'^*_{E'_{\ast}}$ and $\SSS^*_{\dagger}(\mathbb{R}^d)$ is dense in it. The notation $\mathcal{D}'^*_{E'_{\ast}}=(\mathcal{D}^*_{E})'$ is justified by the next structural theorem from \cite{dpv-2}. The spaces $UC_{\omega}$ and $C_{\omega}$ are those closed subspaces of $L^{\infty}_{\omega}$ defined in Example \ref{ex1}.

\begin{theorem} [\cite{dpv-2}]\label{karak}
Let $B\subseteq\SSS'^*_{\dagger}(\mathbb{R}^d)$. The following statements are equivalent:
\begin{itemize}
\item [$(i)$] $B$ is a bounded subset of $\mathcal{D}'^*_{E'_{\ast}}$.
\item [$(ii)$] For each $\psi\in\SSS^*_{\dagger}(\RR^d)$, $\{f*\psi|\, f\in B\}$ is a bounded subset of $E'$.
\item [$(iii)$] For each $\psi\in\SSS^*_{\dagger}(\RR^d)$, $\{f*\psi|\, f\in B\}$ is a bounded subset of $E'_*$.
\item [$(iv)$] There exists a bounded subset $B_1$ of $E'$ and an ultradifferential operator $P(D)$ of class $*$ such that each $f\in B$ can be expressed as $f=P(D)g$ with $g\in B_1$.
\item [$(v)$] There exists $B_2\subseteq E'_*\cap  UC_{\omega}$ which is bounded in $E'_*$ and in $ UC_{\omega}$ and an ultradifferential operator $P(D)$ of class $*$ such that each $f\in B$ can be expressed as $f=P(D)g$ with $g\in B_2$. Moreover, if $E$ is reflexive, we may choose $B_2\subseteq E'_*\cap C_{\omega}$.
\end{itemize}
\end{theorem}
We shall also need the ensuing characterization of precompact subsets of $B\subseteq\mathcal{D}_{E'_{\ast}}'^*$.

\begin{proposition}\label{corprecom}
Let $B\subseteq\mathcal{D}_{E'_{\ast}}'^*$. The following statements are equivalent:
\begin{itemize}
\item [$(i)$] $B$ is precompact in $\mathcal{D}_{E'_{\ast}}'^*$.
\item [$(ii)$] There exists a precompact set $B_1$ in $E'$ and an ultradifferential operator $P(D)$ of class $*$ such that $f=P(D) g$ for some $g\in B_1$.
\item [$(iii)$] There exists a set $B_2\subseteq E'_*\cap  UC_{\omega}$ which is precompact in $E'_*$ and in $ UC_{\omega}$ and an ultradifferential operator $P(D)$ of class $*$ such that each $f\in B$ is of the form $f=P(D) g$ for some $g\in B_2$; if $E$ is reflexive one may choose $B_2\subseteq E'_*\cap C_{\omega}$.
\end{itemize}
\end{proposition}

\begin{proof} We consider the Roumieu case as the Beurling case is similar. Let $(i)$ hold. For $f\in \DD'^{\{M_p\}}_{E'_*}$ we define the linear operator $\iota(f)(\varphi)=f*\check{\varphi}$, $\SSS^{\{M_p\}}_{\{A_p\}}(\RR^d)\rightarrow E'_*$ (it is well defined by Theorem \ref{karak}). One easily verifies that $\iota(f)$ is continuous and the mapping $f\mapsto \iota(f)$, $\DD'^{\{M_p\}}_{E'_*}\rightarrow \mathcal{L}_b(\SSS^{\{M_p\}}_{\{A_p\}}(\RR^d),E'_*)$ is continuous (see the discussion after \cite[Corollary 4.11]{dpv-2} for details). Since $B$ is precompact, so is $\iota(B)$. As $\SSS^{\{M_p\}}_{\{A_p\}}(\RR^d)$ is barreled, $\iota(B)$ is equicontinuous in $\mathcal{L}_b\left(\SSS^{\{M_p\}}_{\{A_p\}}(\RR^d), E'_*\right)$. Hence there exists $(r_p)\in\mathfrak{R}$ such that the elements of $\iota(B)$ can be extended to a bounded subset $\widetilde{\iota(B)}=\{\widetilde{\iota(f)}|\, f\in B\}$ of $\mathcal{L}_b\left(X_{(r_p)}, E'_*\right)$. Let $\mathfrak{G}$ be the family of all finite subsets of $\SSS^*_{\dagger}(\RR^d)$. Since $\SSS^*_{\dagger}(\RR^d)$ is dense in $X_{(r_p)}$, $\mathfrak{G}$ is total in $X_{(r_p)}$ and $\mathcal{L}_{\mathfrak{G}}\left(X_{(r_p)},E'_*\right)$ is a l.c.s. (the index $\mathfrak{G}$ stands for the topology of uniform convergence on all sets in $\mathfrak{G}$). The topology induced on $\widetilde{\iota(B)}$ by $\mathcal{L}_{\mathfrak{G}}\left(X_{(r_p)},E'_*\right)$ is the same as the topology induced on it by $\mathcal{L}_{\sigma}\left(\SSS^{\{M_p\}}_{\{A_p\}}(\RR^d), E'_*\right)$, hence $\widetilde{\iota(B)}$ is precompact in $\mathcal{L}_{\mathfrak{G}}\left(X_{(r_p)},E'_*\right)$ (the inclusion $\mathcal{L}_b\left(\SSS^{\{M_p\}}_{\{A_p\}}(\RR^d), E'_*\right)\rightarrow \mathcal{L}_{\sigma}\left(\SSS^{\{M_p\}}_{\{A_p\}}(\RR^d), E'_*\right)$ is continuous). Now, the Banach-Steinhaus theorem \cite[Theorem 4.5, p. 85]{Sch} implies that $\widetilde{\iota(B)}$ is precompact in $\mathcal{L}_p\left(X_{(r_p)}, E'_*\right)$. Pick now, $(r'_p)\in\mathfrak{R}$ with $(r'_p)\leq (r_p)$ such that the inclusion $X_{(r'_p)}\rightarrow X_{(r_p)}$ is compact. Then the inclusion $\mathcal{L}_p\left(X_{(r_p)}, E'_*\right)\rightarrow \mathcal{L}_b\left(X_{(r'_p)}, E'_*\right)$ is continuous. Thus $\widetilde{\iota(B)}$ is precompact in $\mathcal{L}_b\left(X_{(r'_p)}, E'_*\right)$. Now one can use exactly the same technique as in the proof of $(ii)\Rightarrow (iv)$ of Theorem \ref{karak} to conclude $(ii)$ and as in the proof of $(ii)\Rightarrow (v)$ of Theorem \ref{karak} to conclude $(iii)$  (see \cite{dpv-2} for details). The implications $(ii)\Rightarrow (i)$ and $(iii)\Rightarrow (i)$ are obvious.
\end{proof}


We now specialized our discussion to weighted $L^{p}$ spaces of ultradistributions. Let $\eta$ be a ultrapolynomially bounded weight function of class $\dagger$ (cf. Example \ref{ex1}). As usual, we write $q$ for the conjugate index of $p\in [1,\infty]$. For $1<q<\infty$, the choice $E=L^{q}_{\eta^{-1}}$ leads to the spaces $\mathcal{D}'^*_{L^{p}_{\eta}}$, $1<p<\infty$. The latter spaces are reflexive. When $q=1$, we make an exception in the notation and write (in analogy to Schwartz notation) $\mathcal{B}'^*_{\eta}:=\mathcal{D}'^*_{UC_{\eta}}=(\mathcal{D}^*_{L^{1}_{\eta}})'$, the space of $\eta$-bounded ultradistributions of class $\ast$. We set $\dot{\mathcal{B}}^*_{\eta}:=\mathcal{D}^*_{C_{\eta}}$ and $\mathcal{D}_{L_{\eta}^1}'^*:=(\mathcal{D}^*_{C_{\eta}})'=(\dot{\mathcal{B}}^*_\eta)'$, the space of weighted $\eta$-integrable ultradistributions of class $\ast$.

We can also introduce the space
$\mathcal{D}^*_{L^{\infty}_{\eta}}$ as in (\ref{aldef1}) with
$L^{\infty}_{\eta}$ instead of $E$, but the reader should keep in
mind that $E=L^{\infty}_{\eta}$ is not a translation-invariant
(B)-space of ultradistributions of class $\dagger$. Nevertheless,
one can still show that ultradifferential operators act
continuously on it, that the Roumieu case coincides (as set and as
l.c.s.) with the projective limit of the spaces (\ref{aldef2})
defined with $L^{\infty}_{\eta}$ instead of $E$, and  that
$\DD^{\{M_p\}}_{L^{\infty}_{\eta}}$ is regular and complete
\cite{dpv-2}.  We have shown in \cite{dpv-2} the following result:
$\dot{\mathcal{B}}^{*}_{\eta}$ is (as l.c.s.) the closure of
$\mathcal{S}^{*}_{\dagger}(\mathbb{R}^{d})$ in
$\DD^{*}_{L^{\infty}_{\eta}}$. (This assertion is non-trivial in the
Roumieu case). We will make use of the following result in the next section.

\begin{theorem}[\cite{dpv-2}]\label{theorBB}
The strong bidual of $\dot{\mathcal{B}}^*_{\eta}$ is isomorphic to $\DD^*_{L^{\infty}_{\eta}}$ as l.c.s.. Moreover $\dot{\mathcal{B}}^{(M_p)}_{\eta}$ is a distinguished $(F)$-space and consequently $\DD'^{(M_p)}_{L^1_{\eta}}$ is barreled and bornological.
\end{theorem}

\section{$\varepsilon$ product of $\dot{\mathcal{B}}^*_{\eta}$ with a sequentially complete l.c.s.}
\label{tensor}
We are now interested in the $\varepsilon$ product of $\dot{\mathcal{B}}^*_{\eta}$ with a sequentially complete l.c.s.. Here $\eta$ always stands for an ultrapolynomially bounded weight function of class $\dagger$. We begin with the following general result.

\begin{proposition}\label{apppr}
Let $E$ be a translation-invariant $(B)$-space of ultradistributions of class $*-\dagger$. The following assertions hold:
\begin{itemize}
\item[$i)$] if $E$ possesses a Schauder basis then $\DD^*_E$ satisfies the weak sequential approximation property;
\item[$ii)$] $E$ satisfies the weak approximation property if and only if $\DD^*_E$ does it.
\end{itemize}
\end{proposition}

\begin{proof} To prove $i)$, let $\{e_n\}_{n\in\ZZ_+}$ be a normalized Schauder basis for $E$, i.e., for each $e\in E$, there exists a unique sequence of complex number $\{t_n\}_{n\in\ZZ_+}$ such that $e=\lim_{n\rightarrow \infty} \sum_{j=1}^n t_n e_n$. The linear forms $g_n: e\mapsto t_n$ form an equicontinuous subset of $E'$ and $e=\sum_{j=1}^{\infty}g_n(e) e_n$ converges uniformly on compact subsets of $E$ (cf. \cite[Theorem 9.6, p. 115]{Sch}). Let $\chi_m\in\SSS^*_{\dagger}(\RR^d)$, $m\in\ZZ_+$, be as in $ii)$ of Lemma \ref{appincl}. For $m,n\in\ZZ_+$, define $G_{m,n}\in \DD'^*_{E'_*}\otimes \DD^*_E$ by
\beqs
G_{m,n}=\sum_{j=1}^n g_n\otimes \left(\chi_m*e_n\right).
\eeqs
Also, for each $m\in\ZZ_+$, define $G_m:\DD^*_E\rightarrow \DD^*_E$ by $G_m(\varphi)=\chi_m*\varphi$. If we prove that for each $m\in\ZZ_+$, $G_{m,n}\rightarrow G_m$ and $G_m\rightarrow \mathrm{Id}$ in $\mathcal{L}_{\sigma}\left(\DD^*_E,\DD^*_E\right)$, since $\DD^*_E$ is barreled, the Banach-Steinhaus theorem will imply that these convergence also hold in $\mathcal{L}_c\left(\DD^*_E,\DD^*_E\right)$ and $i)$ will be proved. In fact, every weakly convergent sequence is weakly bounded hence equicontinuous since $\DD^*_E$ is barreled and the Banach-Steinhaus theorem implies uniform convergence on precompact subsets. Thus, fix $m\in\ZZ_+$ and $\varphi\in\DD^*_E$. By $a)$ of Proposition \ref{corBeurling Algebra},
\beqs
\frac{h^{|\alpha|}\left\|D^{\alpha}G_{m,n}(\varphi)-D^{\alpha}G_m(\varphi)\right\|_E}{M_{\alpha}}&=& \left\|\frac{h^{|\alpha|}D^{\alpha}\chi_m}{M_{\alpha}}*\left(\sum_{j=1}^ng_n(\varphi)e_n-\varphi\right)\right\|_E\\
&\leq& C\left\|\sum_{j=1}^ng_n(\varphi)e_n-\varphi\right\|_E,
\eeqs
for each $h>0$ (resp. for some $h>0$), where $C$ depends on $\chi_m$). Thus
\beqs
\|G_{m,n}(\varphi)-G_m(\varphi)\|_{E,h}\leq C\left\|\sum_{j=1}^ng_n(\varphi)e_n-\varphi\right\|_E\rightarrow0,\,\, \mbox{as}\,\, n\rightarrow \infty.
\eeqs
It remains to prove $\chi_m*\varphi\rightarrow \varphi$ in $\DD^*_E$ for each $\varphi\in\DD^*_E$. We consider the Roumieu case; the Beurling case is similar. Fix $\varphi\in\DD^{\{M_p\}}_E$. There exists $h>0$ such that $\chi,\varphi\in \DD^{M_p,h}_E$ and $\tilde{C}=\sigma_h(\chi)<\infty$. Let $0<h_1<h$ be arbitrary but fixed. We will prove that $\|\varphi-\varphi\ast \chi_m\|_{E,h_1}\rightarrow 0$. Let $\varepsilon>0$. Observe that there exists $C_1\geq1$ such that $\|\chi_m\|_{1,\omega}\leq C_1$, $\forall m\in\ZZ_+$. Choose $p_0\in\ZZ_+$ such that $(h_1/h)^p\leq \varepsilon/(2C_2)$ for all $p\geq p_0$, $p\in\NN$, where $C_2=C_1(1+\|\varphi\|_{E,h})\geq 1$.
By b) of Proposition \ref{corBeurling Algebra}, we can choose $m_0\in\ZZ_+$ such that $\ds\frac{h_1^{|\alpha|}}{M_{\alpha}}\left\|D^{\alpha}\varphi-D^{\alpha}\varphi\ast \chi_m\right\|_E\leq \varepsilon$ for all $|\alpha|\leq p_0$ and all $m\geq m_0$, $m\in\NN$. Observe that if $|\alpha|\geq p_0$ we have
\beqs
\frac{h_1^{|\alpha|}}{M_{\alpha}}\left\|D^{\alpha}\varphi-D^{\alpha}\varphi* \chi_m\right\|_E&\leq& \frac{h_1^{|\alpha|}}{M_{\alpha}}\left\|D^{\alpha}\varphi\right\|_E+\frac{h_1^{|\alpha|}}{M_{\alpha}} \left\|D^{\alpha}\varphi\right\|_E\|\chi_m\|_{1,\omega}\\
&\leq& \left(\frac{h_1}{h}\right)^{|\alpha|}\|\varphi\|_{E,h}+C_1\left(\frac{h_1}{h}\right)^{|\alpha|}\|\varphi\|_{E,h}\leq \varepsilon.
\eeqs
Hence, for $m\geq m_0$, $\|\varphi-\varphi\ast \chi_m\|_{E,h_1}\leq \varepsilon$, so $\varphi* \chi_m\rightarrow \varphi$ in $\DD^{M_p,h_1}_E$ and consequently also in $\DD^{\{M_p\}}_E$. Thus, $i)$ holds.\\
\indent To prove $ii)$, assume first that $E$ satisfies the weak approximation property. Let $\chi_m\in\SSS^*_{\dagger}(\RR^d)$, $m\in\ZZ_+$, be as in $ii)$ of Lemma \ref{appincl}. Since for each $\varphi\in\DD^*_E$, $\chi_m*\varphi\rightarrow \varphi$ in $\DD^*_E$, if we set $G_m:\varphi\mapsto \chi_m*\varphi$, $G_m:\DD^*_E\rightarrow \DD^*_E$, we have $G_m\rightarrow \mathrm{Id}$ in $\mathcal{L}_{\sigma}\left(\DD^*_E,\DD^*_E\right)$. As $\DD^*_E$ is barreled, the Banach-Steinhaus theorem implies that the convergence holds in $\mathcal{L}_p\left(\DD^*_E,\DD^*_E\right)$. Let $B$ be a precompact subset of $\DD^*_E$ and $V$ a neighborhood of zero in $\DD^*_E$ and consider the neighborhood of zero $M(B,V)=\{G\in\mathcal{L}\left(\DD^*_E,\DD^*_E\right)|\, G(B)\subseteq V\}$ in $\mathcal{L}_p\left(\DD^*_E,\DD^*_E\right)$. Pick a neighborhood of zero $W$ in $\DD^*_E$ such that $W+W\subseteq V$. There exists $m_0\in\ZZ_+$ such that $G_{m_0}\in \mathrm{Id}+M(B,W)$. By applying Proposition \ref{corBeurling Algebra}, $a)$, one concludes that the mapping $\tilde{G}_{m_0}: E\rightarrow \DD^*_E$, $\tilde{G}_{m_0}(e)=\chi_{m_0}*e$, is well defined and continuous. Obviously, the restriction of $\tilde{G}_{m_0}$ to $\DD^*_E$ is exactly $G_{m_0}$. Pick a neighborhood of zero $U$ in $E$ such that $\tilde{G}_{m_0}(U)\subseteq W$. Since the inclusion $\DD^*_E\rightarrow E$ is continuous, $B$ is precompact in $E$ and since $E$ satisfies the weak approximation property, there exists $S\in E'\otimes E$ such that $S(e)-e\in U$ for all $e\in B$. Let $S=\sum_{j=1}^n e'_j\otimes e_j$. Define $G\in \DD'^*_{E'_*}\otimes \DD^*_E$ by $G=\sum_{j=1}^n e'_j\otimes \tilde{G}_{m_0}(e_j)$. Then, for $\varphi\in B$, we have\\
\\
$G(\varphi)-\varphi$
\beqs
&=&G(\varphi)-\tilde{G}_{m_0}(\varphi)+\tilde{G}_{m_0}(\varphi)-\varphi\\
&=&\tilde{G}_{m_0}\left(\sum_{j=1}^n \langle e'_j,\varphi\rangle e_j-\varphi\right)+G_{m_0}(\varphi)-\varphi\in \tilde{G}_{m_0}(U)+W\subseteq W+W\subseteq V,
\eeqs
which proves that $\DD^*_E$ satisfies the weak approximation property. Conversely, let $\DD^*_E$ satisfies the weak approximation property. Since $\DD^*_E$ is continuously injected into $E$, if we prove that $\mathrm{Id}\in\mathcal{L}_c(E,E)$ is in the closure of the subspace $\mathcal{L}(E,\DD^*_E)$ of $\mathcal{L}_c(E,E)$, \cite[Proposition 2, p. 7]{SchwartzV} will imply that $E$ also satisfies the weak approximation property. Let $\chi_n\in\SSS^*_{\dagger}(\RR^d)$, $n\in\ZZ_+$, be as in $ii)$ of Lemma \ref{appincl}. Define $G_n:E\rightarrow \DD^*_E$, $e\mapsto \chi_n*e$. Proposition \ref{corBeurling Algebra}, $a)$ implies that $G_n$ is well defined and continuous for each $n\in\ZZ_+$ and $c)$ from the same proposition verifies $G_n\rightarrow \mathrm{Id}$ in $\mathcal{L}_{\sigma}(E,E)$. Now, the Banach-Steinhaus theorem yields that the convergence also holds in $\mathcal{L}_c(E,E)$.
\end{proof}

\begin{remark} Since $L^p(\RR^d)$, for $1\leq p<\infty$ and $C_0(\RR^d)$ have Schauder bases (cf. \cite[Corollary 4.1.4, p. 112]{SEM}) so do the spaces $L^p_{\eta}(\RR^d)$, for $1\leq p<\infty$ and $C_{\eta}$ for some positive continuous ultrapolynomially bounded weight of class $\dagger$ since $f\mapsto f/\eta$, $L^p(\RR^d)\rightarrow L^p_{\eta}(\RR^d)$, $1\leq p<\infty$ and $f\mapsto f\eta$, $C_0\rightarrow C_{\eta}$, are isometric isomorphisms between the corresponding spaces. In particular, the above proposition is applicable when $E$ is any of the aforementioned spaces.
\end{remark}

\begin{remark} Enflo, in his seminal paper \cite{enflo}, gave an example of a separable $(B)$-space which does not possess the (weak) approximation property. Later Davie \cite{davie} and Szankowski \cite{szank}, using Enflo's ideas, proved that $l^p$ has a closed subspace which does not have the approximation property for each $p\in[1,2)$ and $p\in(2,\infty)$. It is an interesting problem to find out whether there exists a translation invariant $(B)$-space of ultradistributions which does not possess the (weak) approximation property.
\end{remark}

Since $\DD^*_E$ is complete, if $F$ is sequentially complete, resp. quasi-complete, resp. complete, l.c.s. then the same holds for $\DD^*_E\varepsilon F$ (cf. \cite[Proposition 1.1]{Komatsu3}, \cite[Chapter I, p. 29]{SchwartzV}). By Proposition \ref{apppr}, if $E$ has a Schauder basis and $F$ is sequentially complete, resp. quasi-complete, then $\DD^*_E\varepsilon F$ is canonically isomorphic to the sequential closure, resp. bounding closure, of $\DD^*_E\otimes_{\epsilon} F$ in $\DD^*_E\varepsilon F$ (cf. \cite[Proposition 1.4]{Komatsu3}) and if $E$ satisfies the weak approximation property and $F$ is complete then $\DD^*_E\varepsilon F$ is canonically isomorphic to the closure of $\DD^*_E\otimes_{\epsilon} F$ in $\DD^*_E\varepsilon F$ (i.e., $\DD^*_E\hat{\otimes}_{\epsilon} F$).\\
\indent From now on we will be particularly interested in the case $E=C_{\eta}$. Since $C_{\eta}$ has a Schauder basis, $\dot{\mathcal{B}}^*_{\eta}$ satisfies the weak sequential approximation property. Given $(r_{p})\in \mathfrak{R}$, we write below $R_{\alpha}=\prod_{j=1}^{|\alpha|}r_{j}$.

\begin{lemma}\label{karbb}
$\psi\in\dot{\mathcal{B}}^*_{\eta}$ if and only if $\psi\in\DD^*_{L^{\infty}_\eta}$ and for every $\varepsilon>0$ and $r>0$ (resp. $(r_p)\in\mathfrak{R}$), there exists a compact set $K\subseteq \RR^d$ such that
\beqs
\sup_{\alpha\in\NN^d}\sup_{x\in\RR^d\backslash K} \frac{r^{|\alpha|}\left|D^{\alpha}\psi(x)\right|}{\eta(x)M_{\alpha}}<\varepsilon,\,\, \left(\mbox{resp.}\,\, \sup_{\alpha\in\NN^d}\sup_{x\in\RR^d\backslash K} \frac{\left|D^{\alpha}\psi(x)\right|}{\eta(x)M_{\alpha}R_{\alpha}}<\varepsilon\right).
\eeqs
\end{lemma}

\begin{proof} Denote by $F$ the subspace of $\DD^*_{L^{\infty}_\eta}$ consisting of all such $\psi$; $F$ is closed. Let $\varphi_n\in\SSS^*_{\dagger}(\RR^d)$, $n\in\ZZ_+$, be as in $ii)$ of Lemma \ref{appincl}. For fixed $\psi\in F$, similarly as in the proof of Lemma \ref{appincl}, one verifies that $\varphi_n\psi\in\SSS^*_{\dagger}(\RR^d)$ and $\varphi_n\psi\rightarrow \psi$ in $\DD^*_{L^{\infty}_{\eta}}$.
\end{proof}

\begin{lemma}\label{precompactscl}
$B$ is a precompact subset of $\dot{\mathcal{B}}^*_{\eta}$ if and only if $B$ is bounded in $\dot{\mathcal{B}}^*_{\eta}$ and for every $\varepsilon>0$ and $r>0$ (resp. $(r_p)\in\mathfrak{R}$), there exists a compact set $K\subseteq\RR^d$ such that
\beqs
\sup_{\varphi\in B}\sup_{\substack{\alpha\in \NN^d\\ x\in\RR^d\backslash K}}\frac{r^{|\alpha|}\left|D^{\alpha} \varphi(x)\right|}{\eta(x)M_{\alpha}}\leq \varepsilon,\,\,\left(\mbox{resp. }
\sup_{\varphi\in B}\sup_{\substack{\alpha\in \NN^d\\ x\in\RR^d\backslash K}}\frac{\left|D^{\alpha} \varphi(x)\right|} {\eta(x)M_{\alpha}R_{\alpha}}\leq \varepsilon\right).
\eeqs
\end{lemma}

\begin{proof} $\Rightarrow$. Let $\varepsilon>0$ and $r>0$ (resp. $(r_p)\in\mathfrak{R}$) and
\beqs
V_r=\left\{\varphi\in\dot{\mathcal{B}}^{(M_p)}_{\eta}\big|\,\|\varphi\|_{L^{\infty}_{\eta},r}\leq\varepsilon/2\right\}\,\, \left(\mbox{resp.}\,\, V_{(r_p)}=\left\{\varphi\in\dot{\mathcal{B}}^{\{M_p\}}_{\eta}\big|\, \|\varphi\|_{L^{\infty}_{\eta},(r_p)}\leq\varepsilon/2\right\}\right).
\eeqs
There exist $\varphi_1,...,\varphi_n\in B$ such that for each $\varphi\in B$ there exists $j\in\{1,...,n\}$ with $\varphi\in \varphi_j+V_r$ (resp. $\varphi\in \varphi_j+V_{(r_p)}$). Let $K\subset\subset\RR^d$ such that $$r^{|\alpha|}\left|D^{\alpha}\varphi_j(x)\right|/(\eta(x)M_{\alpha})\leq \varepsilon/2\  (\mbox{resp.} \left|D^{\alpha}\varphi_j(x)\right|/(\eta(x)M_{\alpha}R_{\alpha})\leq \varepsilon/2),$$ for all $x\in \RR^d\backslash K$, $\alpha\in\NN^d$, $j\in\{1,...,n\}$. For $\varphi\in B$ there exists $j\in\{1,...,n\}$ such that
$\|\varphi-\varphi_j\|_{L^{\infty}_{\eta},r}\leq\varepsilon/2$ (resp. $\|\varphi-\varphi_j\|_{L^{\infty}_{\eta},(r_p)}\leq\varepsilon/2$). The proof follows from
\beqs
\frac{r^{|\alpha|}\left|D^{\alpha}\varphi(x)\right|}{\eta(x)M_{\alpha}}\leq \frac{r^{|\alpha|}\left|D^{\alpha}
\left(\varphi(x)-\varphi_j(x)\right)\right|} {\eta(x)M_{\alpha}}+ \frac{r^{|\alpha|}\left|D^{\alpha}\varphi_j(x)\right|}{\eta(x)M_{\alpha}}
\leq \frac{\varepsilon}{2}+\frac{\varepsilon}{2}=\varepsilon,
\eeqs
$ x\in\RR^d\backslash K, \alpha \in\NN^d,$ in the Beurling case; similarly $\left|D^{\alpha}\varphi(x)\right|/(\eta(x)M_{\alpha}R_{\alpha})\leq \varepsilon$ for all $x\in\RR^d\backslash K$, $\alpha \in\NN^d$, in the Roumieu case.\\
\indent $\Leftarrow$. Let $V_r$ (resp. $V_{(r_p)}$) be the neighborhood of zero defined as above but with $\varepsilon$ instead of $\varepsilon/2$. Set $k=2r$ (resp. $(k_p)=(r_p/2)$). Since $B$ is bounded, there exists $C>0$ such that $\|\varphi\|_{L^{\infty}_{\eta},k}\leq C$ (resp. $\|\varphi\|_{L^{\infty}_{\eta},(k_p)}\leq C$), for all $\varphi\in B$. Hence there exists $p_0\in\ZZ_+$ such that for all $\varphi\in B$, $x\in\RR^d$, $|\alpha|\geq p_0$, $r^{|\alpha|}|D^{\alpha}\varphi(x)|/(\eta(x)M_{\alpha})\leq \varepsilon/2$ (resp. $|D^{\alpha}\varphi(x)|/(\eta(x)M_{\alpha}R_{\alpha})\leq \varepsilon/2$). For $\varepsilon/2$ and $r$ (resp. $(r_p)\in\mathfrak{R}$), pick a compact set $K\subseteq \RR^d$ as in the condition of the lemma. Obviously, $B$ is bounded in $C^{\infty}(\RR^d)$ and since the latter space is Montel it must be precompact in $C^{\infty}(\RR^d)$. Thus, there exists a finite subset $B_0=\{\varphi_1,...,\varphi_n\}$ of $B$ such that for each $\varphi\in B$ there exists $j\in\{1,...,n\}$ such that ($\eta$ is continuous and positive)
\beq
&{}&\sup_{|\alpha|\leq p_0}\sup_{x\in K} \frac{r^{|\alpha|}\left|D^{\alpha}\varphi(x)-D^{\alpha}\varphi_j(x)\right|}{\eta(x)M_{\alpha}}\leq \frac{\varepsilon}{2} \label{rav11}\\
&{}&\left(\mbox{resp.}\,\,\sup_{|\alpha|\leq p_0}\sup_{x\in K} \frac{\left|D^{\alpha}\varphi(x)-D^{\alpha}\varphi_j(x)\right|}{\eta(x)M_{\alpha}R_{\alpha}}\leq \frac{\varepsilon}{2}\right).\label{rav13}
\eeq
For $\varphi\in B$ take $\varphi_j\in B_0$ such that (\ref{rav11}) (resp. (\ref{rav13})) holds. When $x\in K$ and $|\alpha|\geq p_0$
\beqs
\frac{r^{|\alpha|}\left|D^{\alpha}\varphi(x)-D^{\alpha}\varphi_j(x)\right|}{\eta(x)M_{\alpha}}\leq \frac{\varepsilon}{2}+\frac{\varepsilon}{2}=\varepsilon\,\, \left(\mbox{resp.}\,\,\frac{\left|D^{\alpha}\varphi(x)-D^{\alpha}\varphi_j(x)\right|}{\eta(x)M_{\alpha}R_{\alpha}}\leq \varepsilon\right).
\eeqs
If $x\in\RR^d\backslash K$ and $\alpha\in\NN^d$ one similarly obtains $r^{|\alpha|}\left|D^{\alpha}\varphi(x)-D^{\alpha}\varphi_j(x)\right|/(\eta(x)M_{\alpha})\leq \varepsilon$ (resp. $\left|D^{\alpha}\varphi(x)-D^{\alpha}\varphi_j(x)\right|/(\eta(x)M_{\alpha}R_{\alpha})\leq \varepsilon$). Hence $\varphi\in \varphi_j+V_r$ (resp. $\varphi\in \varphi_j+V_{(r_p)}$). The proof is complete.
\end{proof}

Since $\dot{\mathcal{B}}^*_{\eta}\hookrightarrow C^{\infty}(\RR^d)$, we have that $\EE'(\RR^d)$ is continuously injected into $\DD'^*_{L^1_{\eta}}$. We denote by $\DD'^*_{L^1_{\eta},c}$ the dual of $\dot{\mathcal{B}}^*_{\eta}$ equipped with the topology of compact convex circled convergence (which coincides with the topology of precompact convergence since $\dot{\mathcal{B}}^*_{\eta}$ is complete). The next lemma shows that $\dot{\mathcal{B}}^*_{\eta}$ satisfies the condition $a)$ of \cite[Theorem 1.12]{Komatsu3}.

\begin{lemma}\label{seqclomsr}
The sequential closure of the set of measures with compact support in $\DD'^*_{L^1_{\eta},c}$ coincides with $\DD'^*_{L^1_{\eta},c}$.
\end{lemma}

\begin{proof} Let $f\in \DD'^*_{L^1_{\eta},c}$. Let $\chi_n\in\SSS^*_{\dagger}(\RR^d)$, $n\in\ZZ_+$, be as in $ii)$ of Lemma \ref{appincl}. Pick a continuous function $g$ with values in $[0,1]$ such that $\mathrm{supp}\, g\subseteq \{x\in\RR^d|\,|x|\leq 1\}$ and $g(x)=1$ on $\{x\in\RR^d|\, |x|\leq 1/2\}$. For $n\in\ZZ_+$, set $g_n(x)=g(x/n)$ and define $G_{m,n}(x)=g_n(x)(\chi_m*f)(x)\in C_c(\RR^d)\subseteq \left(C(\RR^d)\right)'$. Moreover, for $m\in\ZZ_+$ define $G_m(x)=(\chi_m*f)(x)\in L^1_{\eta}$ (cf. Theorem \ref{karak}). Since $\chi_m*f\rightarrow f$ in $\DD'^*_{L^1_{\eta},\sigma}$ (see the proof of $i)$ of Proposition \ref{apppr}) and $\dot{\mathcal{B}}^*_{\eta}$ is barreled, the Banach-Steinhaus theorem yields $\chi_m*f\rightarrow f$ in $\DD'^*_{L^1_{\eta},c}$. It remains to prove that for each $m\in\ZZ_+$, $G_{m,n}\rightarrow G_m$ as $n\rightarrow \infty$ in $\DD'^*_{L^1_{\eta},c}$. Fix $m\in\ZZ_+$ and a precompact set $B$ in $\dot{\mathcal{B}}^*_{\eta}$. Let $\varepsilon>0$. By Lemma \ref{precompactscl} there exists a compact set $K\subseteq \RR^d$ such that $|\varphi(x)|/\eta(x)\leq \varepsilon/\|\chi_m*f\|_{L^1_{\eta}}$ for all $x\in\RR^d\backslash K$, $\varphi\in B$. Pick $n_0\in\ZZ_+$ such that $K\subseteq\{x\in\RR^d|\, |x|\leq n_0/2\}$. Then for all $n\geq n_0$, $g_n(x)=1$ on $K$. Let $\varphi\in B$. We have
\beqs
|\langle G_m(x)-G_{m,n}(x),\varphi\rangle|&\leq& \int_{\RR^d}(1-g_n(x))|(\chi_m*f)(x)||\varphi(x)|dx\\
&\leq& \int_{\RR^d\backslash K} |(\chi_m*f)(x)||\varphi(x)|dx\leq \varepsilon,
\eeqs
which proves the desired convergence.
\end{proof}

\begin{definition}\label{defvekval}
Let $F$ be a sequentially complete, resp. quasi-complete, resp. complete, l.c.s. We define $\dot{\mathcal{B}}^*_{\eta}(\RR^d;F)$ to be the space of all $F$-valued smooth functions $\boldsymbol{\varphi}$ defined on $\RR^d$ such that:
\begin{itemize}
\item[$i)$] for each $r>0$ (resp. $(r_p)\in\mathfrak{R}$) and each continuous seminorm $q$ on $F$, we have
\beqs
q_r(\boldsymbol{\varphi})=\sup_{x,\alpha} q\left(\frac{r^{|\alpha|}D^{\alpha}\boldsymbol{\varphi}(x)}{\eta(x)M_{\alpha}}\right)<\infty\, \left(\mbox{resp.}\, q_{(r_p)}(\boldsymbol{\varphi})=\sup_{x,\alpha} q\left(\frac{D^{\alpha}\boldsymbol{\varphi}(x)}{\eta(x)M_{\alpha}R_{\alpha}}\right)<\infty\right);
\eeqs
\item[$ii)$] for every $\varepsilon>0$, $q$ a continuous seminorm on $F$ and $r>0$ (resp. $(r_p)\in\mathfrak{R}$), there exists a compact set $K\subseteq \RR^d$ such that
\beqs
\sup_{\alpha\in\NN^d}\sup_{x\in\RR^d\backslash K} q\left(\frac{r^{|\alpha|}D^{\alpha}\boldsymbol{\varphi}(x)}{\eta(x)M_{\alpha}}\right)\leq \varepsilon\, \left(\mbox{resp.}\, \sup_{\alpha\in\NN^d}\sup_{x\in\RR^d\backslash K} q\left(\frac{D^{\alpha}\boldsymbol{\varphi}(x)}{\eta(x)M_{\alpha}R_{\alpha}}\right)\leq \varepsilon\right).
\eeqs
\end{itemize}
Equipped with the seminorms $q_r$ (resp. $q_{(r_p)}$), where $q$ varies through the continuous seminorms of $F$ and $r>0$ (resp. $(r_p)\in\mathfrak{R}$), $\dot{\mathcal{B}}^*_{\eta}(\RR^d;F)$ becomes a (Hausdorff) l.c.s..
\end{definition}

We need the following technical lemma.

\begin{lemma}\label{precompactvek}
Let $\boldsymbol{\varphi}\in\dot{\mathcal{B}}^*_{\eta}(\RR^d;F)$. For each $r>0$ (resp. $(r_p)\in\mathfrak{R}$), the set
\beqs
B_r=\left\{\frac{r^{|\alpha|}D^{\alpha}\boldsymbol{\varphi}(x)}{\eta(x)M_{\alpha}}\bigg|\, \alpha\in\NN^d, x\in\RR^d\right\}\,\, \left(\mbox{resp.}\,\, B_{(r_p)}=\left\{\frac{D^{\alpha}\boldsymbol{\varphi}(x)}{\eta(x)M_{\alpha}R_{\alpha}}\bigg|\, \alpha\in\NN^d, x\in\RR^d\right\}\right)
\eeqs
is precompact in $F$.
\end{lemma}

\begin{proof} Let $r>0$ (resp. $(r_p)\in\mathfrak{R}$). Let $q_1,...,q_n$ be continuous seminorms on $F$ and $\varepsilon>0$ and fix a neighborhood of zero $V=\{f\in F|\, q_1(f)\leq \varepsilon,...,q_n(f)\leq \varepsilon\}$ in $F$. Let $k=2r$ (resp. $(k_p)=(r_p/2)$). Since $\boldsymbol{\varphi}\in\dot{\mathcal{B}}^*(\RR^d;F)$, there exists $C>0$ such that $q_j\left(k^{|\alpha|}D^{\alpha}\boldsymbol{\varphi}(x)/(\eta(x)M_{\alpha})\right)\leq C$ (resp. $q_j\left(D^{\alpha}\boldsymbol{\varphi}(x)/(\eta(x)M_{\alpha}K_{\alpha})\right)\leq C$), for all $\alpha\in\NN^d$, $x\in\RR^d$ and $j=1,...,n$. Hence, there exists $p_0\in\ZZ_+$ such that $q_j\left(r^{|\alpha|}D^{\alpha}\boldsymbol{\varphi}(x)/(\eta(x)M_{\alpha})\right)\leq \varepsilon/2$ (resp. $q_j\left(D^{\alpha}\boldsymbol{\varphi}(x)/(\eta(x)M_{\alpha}R_{\alpha})\right)\leq \varepsilon/2$), for all $|\alpha|\geq p_0$, $x\in\RR^d$ and $j=1,...,n$. By condition $ii)$ of Definition \ref{defvekval}, there exists a compact set $K\subseteq \RR^d$ such that $q_j\left(r^{|\alpha|}D^{\alpha}\boldsymbol{\varphi}(x)/(\eta(x)M_{\alpha})\right)\leq \varepsilon/2$ for all $\alpha\in\NN^d$, $x\in\RR^d\backslash K$, $j=1,...,n$, in the Beurling case (resp. $q_j\left(D^{\alpha}\boldsymbol{\varphi}(x)/(\eta(x)M_{\alpha}R_{\alpha})\right)\leq \varepsilon/2$ for all $\alpha\in\NN^d$, $x\in\RR^d\backslash K$, $j=1,...,n$, in the Roumieu case). Since $D^{\alpha}\boldsymbol{\varphi}/\eta$ are continuous mappings from $\RR^d$ to $F$ the set
\beqs
\tilde{B}_r=\left\{\frac{r^{|\alpha|}D^{\alpha}\boldsymbol{\varphi}(x)}{\eta(x)M_{\alpha}}\bigg|\, |\alpha|\leq p_0, x\in K\right\}\, \left(\mbox{resp.}\, \tilde{B}_{(r_p)}=\left\{\frac{D^{\alpha}\boldsymbol{\varphi}(x)}{\eta(x)M_{\alpha}R_{\alpha}}\bigg|\, |\alpha|\leq p_0, x\in K\right\}\right)
\eeqs
is compact in $F$. Thus, there exists a finite set $\tilde{B}_{0,r}\subseteq \tilde{B}_r$ (resp. a finite set $\tilde{B}_{0,(r_p)}\subseteq \tilde{B}_{(r_p)}$) such that $\tilde{B}_r\subseteq \tilde{B}_{0,r}+V$ (resp. $\tilde{B}_{(r_p)}\subseteq \tilde{B}_{0,(r_p)}+V$). Fix $x_0\in\RR^d\backslash K$ and $|\beta|>p_0$ and denote $f_0=r^{|\beta|}D^{\beta}\boldsymbol{\varphi}(x_0)/(\eta(x)M_{\beta})\in B_r$ (resp. $f_0=D^{\beta}\boldsymbol{\varphi}(x_0)/(\eta(x)M_{\beta}R_{\beta})\in B_{(r_p)}$). Let $\tilde{B}_{1,r}=\tilde{B}_{0,r}\cup \{f_0\}$ (resp. $\tilde{B}_{1,(r_p)}=\tilde{B}_{0,(r_p)}\cup \{f_0\}$). We prove that $B_r\subseteq \tilde{B}_{1,r}+V$ (resp. $B_{(r_p)}\subseteq\tilde{B}_{1,(r_p)}+V$), which will complete the proof of the lemma. For $x\in \RR^d\backslash K$ and $\alpha\in\NN^d$, in the Beurling case, by construction
\beqs
q_j\left(r^{|\alpha|}D^{\alpha}\boldsymbol{\varphi}(x)/(\eta(x)M_{\alpha})-f_0\right)\leq q_j\left(r^{|\alpha|}D^{\alpha}\boldsymbol{\varphi}(x)/(\eta(x)M_{\alpha})\right)+q_j(f_0)\leq \varepsilon
\eeqs
for all $j=1,...,n$, hence $r^{|\alpha|}D^{\alpha}\boldsymbol{\varphi}(x)/(\eta(x)M_{\alpha})\in\tilde{B}_{1,r}+V$. Similarly, in the Roumieu case, $D^{\alpha}\boldsymbol{\varphi}(x)/(\eta(x)M_{\alpha}R_{\alpha})\in\tilde{B}_{1,r}+V$. If $x\in K$ and $|\alpha|\geq p_0$ one similarly obtains $r^{|\alpha|}D^{\alpha}\boldsymbol{\varphi}(x)/(\eta(x)M_{\alpha})-f_0\in V$ (resp. $D^{\alpha}\boldsymbol{\varphi}(x)/(\eta(x)M_{\alpha}R_{\alpha})-f_0\in V$). When $x\in K$ and $|\alpha|\leq p_0$, by construction there exists $f\in \tilde{B}_{0,r}$ (resp. $f\in \tilde{B}_{0,(r_p)}$) such that $r^{|\alpha|}D^{\alpha}\boldsymbol{\varphi}(x)/(\eta(x)M_{\alpha})-f\in V$ (resp. $D^{\alpha}\boldsymbol{\varphi}(x)/(\eta(x)M_{\alpha}R_{\alpha})-f\in V$).
\end{proof}

\begin{proposition}
Let $F$ be a sequentially complete l.c.s.. The space $\dot{\mathcal{B}}^*_{\eta}(\RR^d;F)$ is isomorphic to $\dot{\mathcal{B}}^*_{\eta}\varepsilon F$ as l.c.s..
\end{proposition}

\begin{proof} We prove $\dot{\mathcal{B}}^*_{\eta}(\RR^d;F)\cong\mathcal{L}_{\epsilon}\left(F'_c,\dot{\mathcal{B}}^*_{\eta}\right)$. By Lemma \ref{seqclomsr} and \cite[Theorem 1.12]{Komatsu3} the set $\mathcal{L}_{\epsilon}\left(F'_c,\dot{\mathcal{B}}^*_{\eta}\right)$ is identified with the set of all $\mathbf{g}\in C(\RR^d;F)$ such that:
\begin{itemize}
\item[$a)$] for any $f'\in F'$, the function $\langle f', \mathbf{g}(\cdot)\rangle$ is in $\dot{\mathcal{B}}^*_{\eta}(\RR^d)$;
\item[$b)$] for every equicontinuous set $B$ in $F'$, the set $\{\langle f', \mathbf{g}(\cdot)\rangle| f'\in B\}$ is relatively compact in $\dot{\mathcal{B}}^*_{\eta}(\RR^d)$.
\end{itemize}
Let $\boldsymbol{\varphi}\in\dot{\mathcal{B}}^*_{\eta}(\RR^d;F)\subseteq C(\RR^d;F)$ and $f'\in F'$. Then the function $\varphi(x)=\langle f',\boldsymbol{\varphi}(x)\rangle$ is in $C^{\infty}(\RR^d)$ and $D^{\alpha}\varphi(x)=\langle f',D^{\alpha}\boldsymbol{\varphi}(x)\rangle$. Since $f'\in F'$, there exists a continuous seminorm $q$ on $F$ and $C>0$ such that $|\langle f',f\rangle|\leq Cq(f)$ for all $f\in F$. For $r>0$ (resp. $(r_p)\in\mathfrak{R}$) we have $\|\varphi\|_{L^{\infty}_{\eta},r}\leq Cq_r(\boldsymbol{\varphi})$ (resp. $\|\varphi\|_{L^{\infty}_{\eta},(r_p)}\leq Cq_{(r_p)}(\boldsymbol{\varphi})$). Moreover, by condition $ii)$ of Definition \ref{defvekval}, one readily checks that for every $\varepsilon>0$ and every $r>0$ (resp. $(r_p)\in\mathfrak{R}$), there exists a compact set $K\subseteq \RR^d$ such that
\beqs
\sup_{\alpha\in\NN^d}\sup_{x\in\RR^d\backslash K}\frac{r^{|\alpha|}\left|D^{\alpha}\varphi(x)\right|}{\eta(x)M_{\alpha}}\leq \varepsilon\, \left(\mbox{resp.}\, \sup_{\alpha\in\NN^d}\sup_{x\in\RR^d\backslash K} \frac{\left|D^{\alpha}\varphi(x)\right|}{\eta(x)M_{\alpha}R_{\alpha}}\leq \varepsilon\right).
\eeqs
Thus Lemma \ref{karbb} implies that $a)$ holds for $\boldsymbol{\varphi}$. To prove $b)$, let $H$ be an equicontinuous set in $F'$. Of course, we can assume that $H=U^{\circ}$ ($U^{\circ}$ stands for the polar of $U$) for a convex circled closed neighborhood of zero $U=\{f\in F|\, q_1(f)\leq 1,...,q_n(f)\leq1\}$ in $F$. For $f'\in H$, let $\varphi_{f'}(\cdot)=\langle f',\boldsymbol{\varphi}(\cdot)\rangle$. We prove that the set $\Phi=\{\varphi_{f'}|\, f'\in H\}$ is precompact in $\dot{\mathcal{B}}^*_{\eta}$ and hence relatively compact as $\dot{\mathcal{B}}^*_{\eta}$ is complete. For $r>0$ (resp. $(r_p)\in\mathfrak{R}$) the set $B_r=\{r^{|\alpha|}D^{\alpha}\boldsymbol{\varphi}(x)/\left(\eta(x)M_{\alpha}\right)|\, x\in\RR^d,\, \alpha\in\NN^d\}$ (resp. the set $B_{(r_p)}=\{D^{\alpha}\boldsymbol{\varphi}(x)/\left(\eta(x)M_{\alpha}R_{\alpha}\right)|\, x\in\RR^d,\, \alpha\in\NN^d\}$) is bounded in $F$ (in fact it is precompact by Lemma \ref{precompactvek}), hence $\ds\sup_{f'\in H} \|\varphi_{f'}\|_{L^{\infty}_{\eta},r}<\infty$ (resp. $\ds\sup_{f'\in H} \|\varphi_{f'}\|_{L^{\infty}_{\eta},(r_p)}<\infty$). Thus $\Phi$ is bounded in $\dot{\mathcal{B}}^*_{\eta}$. For $r>0$ (resp. $(r_p)\in\mathfrak{R}$) and $\varepsilon>0$ the condition $ii)$ of Definition \ref{defvekval} implies that there exists a compact set $K\subseteq \RR^d$ such that
\beqs
\sup_{\alpha\in\NN^d}\sup_{x\in\RR^d\backslash K} q_j\left(\frac{r^{|\alpha|}D^{\alpha}\boldsymbol{\varphi}(x)}{\eta(x)M_{\alpha}}\right)\leq \varepsilon\, \left(\mbox{resp.}\, \sup_{\alpha\in\NN^d}\sup_{x\in\RR^d\backslash K} q_j\left(\frac{D^{\alpha}\boldsymbol{\varphi}(x)}{\eta(x)M_{\alpha}R_{\alpha}}\right)\leq \varepsilon\right)
\eeqs
for all $j\in\{1,...,n\}$. Since $H=U^{\circ}$, we have
\beqs
\sup_{f'\in H}\sup_{\alpha\in\NN^d}\sup_{x\in\RR^d\backslash K} \frac{r^{|\alpha|}\left|D^{\alpha}\varphi_{f'}(x)\right|}{\eta(x)M_{\alpha}}\leq \varepsilon\, \left(\mbox{resp.}\, \sup_{f'\in H} \sup_{\alpha\in\NN^d}\sup_{x\in\RR^d\backslash K} \frac{\left|D^{\alpha}\varphi_{f'}(x)\right|}{\eta(x)M_{\alpha}R_{\alpha}}\leq \varepsilon\right).
\eeqs
Thus, Lemma \ref{precompactscl} implies $\Phi$ is precompact in $\dot{\mathcal{B}}^*_{\eta}$. Conversely, for $G\in\mathcal{L}_{\epsilon}(F'_c,\dot{\mathcal{B}}^*_{\eta})$ let $\mathbf{g}\in C(\RR^d;F)$ be the function which satisfies $a)$ and $b)$ that generates $G$, i.e., $G(f')(\cdot)=\langle f',\mathbf{g}(\cdot)\rangle$, $f'\in F'_c$. We have to prove that $\mathbf{g}\in\dot{\mathcal{B}}^*_{\eta}(\RR^d;F)$. Denote by $g_{f'}$ the function $x\mapsto \langle f', \mathbf{g}(\cdot)\rangle$, $\RR^d\rightarrow \CC$. Since $g_{f'}\in \dot{\mathcal{B}}^*_{\eta}(\RR^d)\subseteq C^{\infty}(\RR^d)$ for each $f'\in F'$, it follows that $\mathbf{g}\in C^{\infty}(\RR^d;F)$ (cf. \cite[Appendix, Lemma II]{SchwartzT} and the remark after it). Let $q$ be a continuous seminorm on $F$ and $r>0$ (resp. $(r_p)\in\mathfrak{R}$). Let $U=\{f\in F|\, q(f)\leq 1\}$. Then $B=U^{\circ}$ is equicontinuous subset of $F'$. Thus, by $b)$, $\{g_{f'}|\, f'\in B\}$ is precompact in $\dot{\mathcal{B}}^*_{\eta}$ and by Lemma \ref{precompactscl}, using $B^{\circ}=U$ ($U$ is convex, circled and closed), one easily verifies that $\mathbf{g}$ satisfies the conditions of Definition \ref{defvekval}, i.e., $\mathbf{g}\in\dot{\mathcal{B}}^*_{\eta}(\RR^d;F)$. We obtain that the linear mapping $\boldsymbol{\varphi}\mapsto G_{\boldsymbol{\varphi}}$, $\dot{\mathcal{B}}^*_{\eta}(\RR^d;F)\rightarrow \mathcal{L}_{\epsilon}(F'_c,\dot{\mathcal{B}}^*_{\eta})$, where $G_{\boldsymbol{\varphi}}(f')=\langle f',\boldsymbol{\varphi}(\cdot)\rangle$, is a bijection. It remains to prove that it is a topological isomorphism.\\
\indent To prove that it is continuous, let $B$ be equicontinuous subset of $F'$ and $V_{r,\varepsilon}=\{\psi\in\dot{\mathcal{B}}^{(M_p)}_{\eta}|\, \|\psi\|_{L^{\infty}_{\eta},r}\leq \varepsilon\}$ a neighborhood of zero in $\dot{\mathcal{B}}^{(M_p)}_{\eta}$ and $V_{(r_p),\varepsilon}=\{\psi\in\dot{\mathcal{B}}^{\{M_p\}}_{\eta}|\, \|\psi\|_{L^{\infty}_{\eta},(r_p)}\leq \varepsilon\}$ a neighborhood of zero in $\dot{\mathcal{B}}^{\{M_p\}}_{\eta}$. Consider the neighborhoods of zero
\beqs
M(B,V_{r,\varepsilon})&=&\{G\in\mathcal{L}(F'_c,\dot{\mathcal{B}}^{(M_p)}_{\eta})|\, G(B)\subseteq V_{r,\varepsilon}\}\,\, \mbox{in}\,\, \mathcal{L}_{\epsilon}(F'_c,\dot{\mathcal{B}}^{(M_p)}_{\eta})\,\, \mbox{and}\\
M(B,V_{(r_p),\varepsilon})&=&\{G\in\mathcal{L}(F'_c,\dot{\mathcal{B}}^{\{M_p\}}_{\eta})|\, G(B)\subseteq V_{(r_p),\varepsilon}\}\,\, \mbox{in}\,\, \mathcal{L}_{\epsilon}(F'_c,\dot{\mathcal{B}}^{\{M_p\}}_{\eta}).
\eeqs
Of course, without losing generality, we can assume that $B=U^{\circ}$ for a convex circled closed neighborhood of zero $U=\{f\in F|\, q_1(f)\leq 1,...,q_n(f)\leq1\}$ in $F$. Consider the neighborhoods of zero
\beqs
W_{r,\varepsilon}&=&\left\{\boldsymbol{\psi}\in\dot{\mathcal{B}}^{(M_p)}_{\eta}(\RR^d;F)\bigg|\,\sup_{\alpha,x}q_j \left(\frac{r^{|\alpha|}D^{\alpha}\boldsymbol{\psi}(x)}{\eta(x)M_{\alpha}}\right)\leq \varepsilon,\, j=1,...,n\right\}\,\, \mbox{and}\\
W_{(r_p),\varepsilon}&=&\left\{\boldsymbol{\psi}\in\dot{\mathcal{B}}^{\{M_p\}}_{\eta}(\RR^d;F)\bigg|\,\sup_{\alpha,x}q_j \left(\frac{D^{\alpha}\boldsymbol{\psi}(x)}{\eta(x)M_{\alpha}R_{\alpha}}\right)\leq \varepsilon,\, j=1,...,n\right\},
\eeqs
in $\dot{\mathcal{B}}^{(M_p)}_{\eta}(\RR^d;F)$ and $\dot{\mathcal{B}}^{\{M_p\}}_{\eta}(\RR^d;F)$ respectively. Then, for $\boldsymbol{\varphi}\in W_{r,\varepsilon}$ we have $r^{|\alpha|}D^{\alpha}\boldsymbol{\varphi}(x)/(\eta(x)M_{\alpha})\in \varepsilon U$ for all $\alpha\in\NN^d$, $x\in\RR^d$ (resp. for $\boldsymbol{\varphi}\in W_{(r_p),\varepsilon}$ we have $D^{\alpha}\boldsymbol{\varphi}(x)/(\eta(x)M_{\alpha}R_{\alpha})\in \varepsilon U$ for all $\alpha\in\NN^d$, $x\in\RR^d$). Thus $G_{\boldsymbol{\varphi}}(f')\in V_{r,\varepsilon}$ (resp. $G_{\boldsymbol{\varphi}}(f')\in V_{(r_p),\varepsilon}$), for all $f'\in B$ which proves the continuity of $\boldsymbol{\varphi}\mapsto G_{\boldsymbol{\varphi}}$. Conversely, for the neighborhoods of zero $W_{r,\varepsilon}$ in $\dot{\mathcal{B}}^{(M_p)}_{\eta}(\RR^d;F)$ and $W_{(r_p),\varepsilon}$ in $\dot{\mathcal{B}}^{\{M_p\}}_{\eta}(\RR^d;F)$, where $W_{r,\varepsilon}$ and $W_{(r_p),\varepsilon}$ are defined as above, consider the neighborhoods of zero $M(B,V_{r,\varepsilon})$ in $\mathcal{L}_{\epsilon}(F'_c,\dot{\mathcal{B}}^{(M_p)}_{\eta})$ and $M(B,V_{(r_p),\varepsilon})$ in $\mathcal{L}_{\epsilon}(F'_c,\dot{\mathcal{B}}^{\{M_p\}}_{\eta})$, defined as above. For $G_{\boldsymbol{\varphi}}\in M(B,V_{r,\varepsilon})$ (resp. $G_{\boldsymbol{\varphi}}\in M(B,V_{(r_p),\varepsilon})$), we have
\beqs
\left|\left\langle f',\frac{r^{|\alpha|}D^{\alpha}\boldsymbol{\varphi}(x)}{\eta(x)M_{\alpha}}\right\rangle\right|\leq \varepsilon\,\, \left(\mbox{resp.}\,\, \left|\left\langle f',\frac{D^{\alpha}\boldsymbol{\varphi}(x)}{\eta(x)M_{\alpha}R_{\alpha}}\right\rangle\right|\leq \varepsilon\right)
\eeqs
for all $f'\in B=U^{\circ}$, $\alpha\in\NN^d$, $x\in \RR^d$. This implies that $\varepsilon^{-1} r^{|\alpha|}D^{\alpha}\boldsymbol{\varphi}(x)/(\eta(x)M_{\alpha})\in B^{\circ}=U^{\circ\,\circ}=U$ (the last equality holds since $U$ is convex, circled and closed) for all $\alpha\in\NN^d$, $x\in\RR^d$, in the Beurling case and similarly $\varepsilon^{-1} D^{\alpha}\boldsymbol{\varphi}(x)/(\eta(x)M_{\alpha}R_{\alpha})\in U$ for all $\alpha\in\NN^d$, $x\in\RR^d$, in the Roumieu case. Thus $\boldsymbol{\varphi}\in W_{r,\varepsilon}$ (resp. $\boldsymbol{\varphi}\in W_{(r_p),\varepsilon}$), which proves that $\boldsymbol{\varphi}\mapsto G_{\boldsymbol{\varphi}}$, $\dot{\mathcal{B}}^*_{\eta}(\RR^d;F)\rightarrow \mathcal{L}_{\epsilon}(F'_c, \dot{\mathcal{B}}^*_{\eta})$ is topological isomorphism.
\end{proof}

Since $\dot{\mathcal{B}}^*_{\eta}$ is complete and satisfies the weak sequential approximation property, the above proposition together with \cite[Proposition 1.4]{Komatsu3} implies that if $F$ is a sequentially complete, resp. quasi-complete, resp. complete, l.c.s. the space $\dot{\mathcal{B}}^*_{\eta}(\RR^d;F)(\cong \dot{\mathcal{B}}^*_{\eta}\varepsilon F)$ is canonically isomorphic to the sequential completion, resp quasi-completion, resp. completion, of $\dot{\mathcal{B}}^*_{\eta}\otimes_{\epsilon} F$. Thus, by taking $F=\dot{\mathcal{B}}^*_{\eta_1}(\RR^m)$ where $\eta_1$ is continuous positive ultrapolynomially bounded weight of class $\dagger$, we have the canonical isomorphism of l.c.s.
\beqs
\dot{\mathcal{B}}^*_{\eta}(\RR^d;\dot{\mathcal{B}}^*_{\eta_1}(\RR^m))\cong \dot{\mathcal{B}}^*_{\eta}(\RR^d)\varepsilon \dot{\mathcal{B}}^*_{\eta_1}(\RR^m)\cong \dot{\mathcal{B}}^*_{\eta}(\RR^d)\hat{\otimes}_{\epsilon}\dot{\mathcal{B}}^*_{\eta_1}(\RR^m).
\eeqs
The above results allows us to prove the following proposition which will be essential for the proof of the main result from the following section.

\begin{proposition}\label{epsdd}
Let $\eta_j$ be positive continuous ultrapolynomially bounded weights of class $\dagger$ on $\RR^{d_j}$, for $j=1,2$. Then $\eta=\eta_1\otimes \eta_2$ is a positive continuous ultrapolynomially bounded weight of class $\dagger$ and $\dot{\mathcal{B}}^*_{\eta}(\RR^{d_1+d_2})\cong \dot{\mathcal{B}}^*_{\eta_1}(\RR^{d_1})\hat{\otimes}_{\epsilon} \dot{\mathcal{B}}^*_{\eta_2}(\RR^{d_2})$.
\end{proposition}

\begin{proof} That $\eta$ is a continuous positive ultrapolynomially bounded weight of class $\dagger$ on $\RR^{d_1+d_2}$ is obvious. By the above discussion, it is enough to prove $\dot{\mathcal{B}}^*_{\eta}(\RR^{d_1+d_2})\cong \dot{\mathcal{B}}^*_{\eta_1}(\RR^{d_1};\dot{\mathcal{B}}^*_{\eta_2}(\RR^{d_2}))$; but the proof of this fact is analogous to that of \cite[Proposition 4]{PB} and we omit it.
\end{proof}

\section{Existence of convolution of two ultradistributions}
\label{section convolution}

Denote by $\mathfrak{C}_0$ the set of all $g\in C_0(\RR^d)$ such that $g(x)> 0$ for all $x\in\RR^d$. For $g\in\mathfrak{C}_0$ and $r>0$ (resp. $(r_p)\in\mathfrak{R}$ and denote $R_{\alpha}=\prod_{j=1}^{|\alpha|}r_j$), we define $\tilde{\tilde{\DD}}^{M_p}_{L^{\infty}_{\eta},g,r}$ (resp. $\tilde{\tilde{\DD}}^{M_p}_{L^{\infty}_{\eta},g,(r_p)}$) to be the space of all $\varphi\in C^{\infty}(\RR^d)$ such that
\beqs
p_{g,r}(\varphi)=\sup_{\alpha\in\NN^d}\frac{r^{|\alpha|}\left\|gD^{\alpha}\varphi\right\|_{L^{\infty}_{\eta}}} {M_{\alpha}}<\infty\,\, \left(\mbox{resp.}\,\, p_{g,(r_p)}(\varphi)=\sup_{\alpha\in\NN^d}\frac{\left\|gD^{\alpha}\varphi\right\|_{L^{\infty}_{\eta}}} {M_{\alpha}R_{\alpha}}<\infty\right).
\eeqs
One easily obtains that $\tilde{\tilde{\DD}}^{M_p}_{L^{\infty}_{\eta},g,r}$ (resp. $\tilde{\tilde{\DD}}^{M_p}_{L^{\infty}_{\eta},g,(r_p)}$) becomes a $(B)$-space when equipped with the norm $p_{q,r}$ (resp. $p_{g,(r_p)}$). For $g\in\mathfrak{C}_0$ define
\beqs
\tilde{\tilde{\DD}}^{(M_p)}_{L^{\infty}_{\eta},g}=\lim_{\substack{\longleftarrow\\ r\rightarrow \infty}} \tilde{\tilde{\DD}}^{M_p}_{L^{\infty}_{\eta},g,r},\,\,\, \mbox{resp.}\,\,\, \tilde{\tilde{\DD}}^{\{M_p\}}_{L^{\infty}_{\eta},g}= \lim_{\substack{\longleftarrow\\ (r_p)\in\mathfrak{R}}} \tilde{\tilde{\DD}}^{M_p}_{L^{\infty}_{\eta},g,(r_p)}.
\eeqs
For each $g\in\mathfrak{C}_0$, $\tilde{\tilde{\DD}}^*_{L^{\infty}_{\eta},g}$ is a complete l.c.s. We define an order $\preceq$ on $\mathfrak{C}_0$ as follows: $g\preceq g_1$ when $g(x)\leq g_1(x)$ for all $x\in\RR^d$. Since for $g,g_1\in\mathfrak{C}_0$, $g_2(x)=\max\{g(x),g_1(x)\}$, $x\in\RR^d$, is again in $\mathfrak{C}_0$, $(\mathfrak{C}_0,\preceq)$ becomes a directed set. If $g,g_1\in\mathfrak{C}_0$ with $g\preceq g_1$, one has that $\tilde{\tilde{\DD}}^*_{L^{\infty}_{\eta},g_1}$ is continuously injected into $\tilde{\tilde{\DD}}^*_{L^{\infty}_{\eta},g}$ under the canonical inclusion mapping. Hence, we can define as l.c.s.
\beqs
\tilde{\tilde{\DD}}^{(M_p)}_{L^{\infty}_{\eta}}=\lim_{\substack{\longleftarrow\\ g\in \mathfrak{C}_0}} \tilde{\tilde{\DD}}^{(M_p)}_{L^{\infty}_{\eta},g},\,\,\, \mbox{resp.}\,\,\, \tilde{\tilde{\DD}}^{\{M_p\}}_{L^{\infty}_{\eta}}= \lim_{\substack{\longleftarrow\\ g\in\mathfrak{C}_0}} \tilde{\tilde{\DD}}^{\{M_p\}}_{L^{\infty}_{\eta},g}.
\eeqs
Clearly, $\tilde{\tilde{\DD}}^*_{L^{\infty}_{\eta}}$ is complete.

\begin{lemma}\label{bsubddd} As sets $\DD^*_{L^{\infty}_{\eta}}= \tilde{\tilde{\DD}}^*_{L^{\infty}_{\eta}}$
 and the identity mapping $\DD^*_{L^{\infty}_{\eta}}\rightarrow \tilde{\tilde{\DD}}^*_{L^{\infty}_{\eta}}$ is continuous. Moreover, $\DD^*_{L^{\infty}_{\eta}}$ and $\tilde{\tilde{\DD}}^*_{L^{\infty}_{\eta}}$ have the same bounded sets.
\end{lemma}

\begin{proof} It is easy to verify that the inclusion mapping $\DD^*_{L^{\infty}_{\eta}}\rightarrow \tilde{\tilde{\DD}}^*_{L^{\infty}_{\eta}}$ is a well-defined and continuous injection. We prove the surjectivity in the Roumieu case as the Beurling case is analogous. Let $\varphi\in \tilde{\tilde{\DD}}^{\{M_p\}}_{L^{\infty}_{\eta}}$ but $\varphi\not\in \DD^{\{M_p\}}_{L^{\infty}_{\eta}}$. So there exists $(r_p)\in\mathfrak{R}$ such that $\|\varphi\|_{L^{\infty}_{\eta},(r_p)}=\infty$. Let $g_1(x)=e^{-|x|}$. Clearly $g_1\in \mathfrak{C}_0$. Hence $C=p_{g_1,(r_p)}(\varphi)<\infty$. For $j\in\ZZ_+$, let $K_j=\{x\in\RR^d|\, |x|\leq j\}$. As $g_1$ is positive, $\ds C_j=\sup_{\alpha}\sup_{x\in K_j}\left|D^{\alpha}\varphi(x)\right|/\left(\eta(x)M_{\alpha}R_{\alpha}\right)<\infty$ for $j\in\ZZ_+$. Since $\|\varphi\|_{L^{\infty}_{\eta},(r_p)}=\infty$, $C_j$ monotonically increases to $\infty$. Hence we can find $\alpha^{(j)}\in\NN^d$ and $x^{(j)}\in\RR^d$, $j\in\ZZ_+$, such that $|x^{(j)}|+1\leq |x^{(j+1)}|$ and $\left|D^{\alpha^{(j)}}\varphi(x^{(j)})\right|\geq j \eta(x^{(j)})M_{\alpha^{(j)}} R_{\alpha^{(j)}}$ for all $j\in\ZZ_+$. Let $\rho_j=|x^{(j)}|$ and set $g_0(\rho_j)=j^{-1/2}$, for $j\in \ZZ_+$. Define $g_0:[0,\infty)\rightarrow (0,\infty)$ linearly on the intervals $(\rho_j,\rho_{j+1})$ and by the constant $1$ on $[0,\rho_1)$. Then $g_0$ is continuous monotonically decreasing and tends to $0$ when $\rho\rightarrow \infty$. Let $g(x)=g_0(|x|)$ for $x\in\RR^d$. It is easy to verify that $g\in\mathfrak{C}_0$. Observe that
\beqs
\frac{g(x^{(j)})\left|D^{\alpha^{(j)}}\varphi(x^{(j)})\right|}{\eta(x^{(j)})M_{\alpha^{(j)}}R_{\alpha^{(j)}}}\geq \sqrt{j}\rightarrow \infty,\,\, \mbox{as}\,\, j\rightarrow\infty,
\eeqs
i.e., $p_{g,(r_p)}(\varphi)=\infty$ which is a contradiction. It remains to prove that $\DD^*_{L^{\infty}_{\eta}}$ and $\tilde{\tilde{\DD}}^*_{L^{\infty}_{\eta}}$ have the same bounded sets. Clearly, each bounded set in the former space is bounded in the latter. We prove the converse in the Roumieu case, the Beurling case being similar. Let $B$ be a bounded subset of $\tilde{\tilde{\DD}}^{\{M_p\}}_{L^{\infty}_{\eta}}$ which is not bounded in $\DD^{\{M_p\}}_{L^{\infty}_{\eta}}$. Thus, there exists $(r_p)\in\mathfrak{R}$ such that $\ds \sup_{\varphi\in B}\|\varphi\|_{L^{\infty}_{\eta},(r_p)}=\infty$. Similarly as in the proof of the surjectivity, we can find $\varphi_j\in B$, $\alpha^{(j)}\in\NN^d$ and $x^{(j)}\in\RR^d$, $j\in\ZZ_+$, such that $|x^{(j)}|+1\leq |x^{(j+1)}|$ and $\left|D^{\alpha^{(j)}}\varphi_j(x^{(j)})\right|\geq j \eta(x^{(j)})M_{\alpha^{(j)}} R_{\alpha^{(j)}}$ for all $j\in\ZZ_+$. By defining $g\in\mathfrak{C}_0$ as above, one obtains
\beqs
\frac{g(x^{(j)})\left|D^{\alpha^{(j)}}\varphi_j(x^{(j)})\right|}{\eta(x^{(j)})M_{\alpha^{(j)}}R_{\alpha^{(j)}}}\geq \sqrt{j}\rightarrow \infty,\,\, \mbox{as}\,\, j\rightarrow\infty,
\eeqs
i.e., $\sup_{\varphi\in B}p_{g,(r_p)}(\varphi)=\infty$ which is a contradiction.
\end{proof}

\begin{lemma}\label{preinDL}
For each $g\in\mathfrak{C}_0$ and $r>0$ (resp $(r_p)\in\mathfrak{R}$), the set
\beqs
\tilde{B}=\left\{\frac{g(a)r^{|\alpha|}D^{\alpha}\delta_a}{\eta(a)M_{\alpha}}\Big|\,a\in\RR^d,\,\alpha\in\NN^d\right\}\,\, \left(\mbox{resp.}\,\, \tilde{B}=\left\{\frac{g(a)D^{\alpha}\delta_a}{\eta(a)M_{\alpha}R_{\alpha}}\Big|\,a\in\RR^d,\,\alpha\in\NN^d\right\}\right)
\eeqs
is precompact in $\DD'^*_{L^1_{\eta}}$.
\end{lemma}

\begin{proof} Let $B$ be a bounded subset of $\dot{\mathcal{B}}^*_{\eta}$ and consider the neighborhood of zero $B^{\circ}$ in $\DD'^*_{L^1_{\eta}}$. In the Beurling case, for $r'=2r$,
\beqs
\sup_{\varphi\in B}\sup_{a,\alpha}\left|\left\langle \frac{r'^{|\alpha|}D^{\alpha}\delta_a}{\eta(a)M_{\alpha}},\varphi\right\rangle\right|=\sup_{\varphi\in B}\sup_{a,\alpha} \frac{r'^{|\alpha|}\left|D^{\alpha}\varphi(a)\right|}{\eta(a)M_{\alpha}}=\sup_{\varphi\in B} \|\varphi\|_{L^{\infty}_{\eta},r'}=C'<\infty
\eeqs
and in the Roumieu case, for $(r'_p)=(r_p/2)$ denote $R'_{\alpha}=\prod_{j=1}^{|\alpha|}r'_j$, to obtain
\beqs
\sup_{\varphi\in B}\sup_{a,\alpha}\left|\left\langle \frac{D^{\alpha}\delta_a}{\eta(a)M_{\alpha}R'_{\alpha}},\varphi\right\rangle\right|=\sup_{\varphi\in B}\sup_{a,\alpha} \frac{\left|D^{\alpha}\varphi(a)\right|}{\eta(a)M_{\alpha}R'_{\alpha}}=\sup_{\varphi\in B} \|\varphi\|_{L^{\infty}_{\eta},(r'_p)}=C'<\infty.
\eeqs
Set $C=C'+1$. Thus, there exists $p_0\in\ZZ_+$ such that for all $|\alpha|\geq p_0$, $a\in\RR^d$, $\varphi\in B$,
\beqs
\left|\left\langle \frac{r^{|\alpha|}D^{\alpha}\delta_a}{\eta(a)M_{\alpha}},\varphi\right\rangle\right|\leq (2C\|g\|_{L^{\infty}})^{-1},\,\, \mbox{resp.}\,\, \left|\left\langle \frac{D^{\alpha}\delta_a}{\eta(a)M_{\alpha}R_{\alpha}},\varphi\right\rangle\right|\leq (2C\|g\|_{L^{\infty}})^{-1}.
\eeqs
Since $g\in\mathfrak{C}_0$ there exists $c\geq 1$ such that $g(x)\leq 1/(2C)$ for all $|x|\geq c$. Now, observe that the set
\beqs
\tilde{\tilde{B}}=\left\{\frac{g(a)r^{|\alpha|}D^{\alpha}\delta_a}{\eta(a)M_{\alpha}}\Big|\,|a|\leq c,\,|\alpha|\leq p_0\right\},\,\, \mbox{resp.}\,\, \tilde{\tilde{B}}=\left\{\frac{g(a)D^{\alpha}\delta_a}{\eta(a)M_{\alpha}R_{\alpha}}\Big|\,|a|\leq c,\,|\alpha|\leq p_0\right\},
\eeqs
is bounded in $\EE'(\RR^d)$ and since the latter space is Montel it must be precompact in $\EE'(\RR^d)$ and thus also in $\DD'^*_{L^1_{\eta}}$. Hence there exists a finite subset $H=\{f_1,...,f_n\}\subseteq \tilde{\tilde{B}}$ such that $\tilde{\tilde{B}}\subseteq H+B^{\circ}$. Fix $b\in\RR^d$ with $|b|\geq c$ and $\beta\in\NN^d$ with $|\beta|\geq p_0$ and set $f_0=g(b)r^{|\beta|}D^{\beta}\delta_b/(\eta(b)M_{\beta})\in \tilde{B}$ (resp. $f_0=g(b)D^{\beta}\delta_b/(\eta(b)M_{\beta}R_{\beta})\in\tilde{B}$). Define $\tilde{H}=H\cup \{f_0\}$. If $|\alpha|\leq p_0$ and $|a|\leq c$, then, by construction, there exists $f_j\in H$ such that $g(a)r^{|\alpha|}D^{\alpha}\delta_a/(\eta(a)M_{\alpha})-f_j\in B^{\circ}$ (resp. $g(a)D^{\alpha}\delta_a/(\eta(a)M_{\alpha}R_{\alpha})-f_j\in B^{\circ}$). If $|\alpha|\geq p_0$ or $|a|\geq c$ then $g(a)r^{|\alpha|}D^{\alpha}\delta_a/(\eta(a)M_{\alpha})-f_0\in B^{\circ}$ (resp. $g(a)D^{\alpha}\delta_a/(\eta(a)M_{\alpha}R_{\alpha})-f_0\in B^{\circ}$). Hence $\tilde{B}$ is precompact in $\DD'^*_{L^1_{\eta}}$.
\end{proof}

We denote by $\DD^*_{L^{\infty}_{\eta},c}$ the space
$\DD^*_{L^{\infty}_{\eta}}$ equipped with the topology of compact
convex circled convergence from the duality $\langle
\DD'^*_{L^1_{\eta}}, \DD^*_{L^{\infty}_{\eta}}\rangle$ (cf.
Theorem \ref{theorBB}).

\begin{proposition}
The spaces $\DD^*_{L^{\infty}_{\eta},c}$ and $\tilde{\tilde{\DD}}^*_{L^{\infty}_{\eta}}$ are isomorphic as l.c.s.
\end{proposition}

\begin{proof} Lemma \ref{bsubddd} states that these spaces are algebraically isomorphic. Let $\varepsilon>0$, $g\in\mathfrak{C}_0$, $r>0$, resp. $(r_p)\in\mathfrak{R}$ and consider the neighborhood of zero $V=\{\varphi\in\tilde{\tilde{\DD}}^{(M_p)}_{L^{\infty}_{\eta}}|\, p_{g,r}(\varphi)\leq \varepsilon\}$ (resp. $V=\{\varphi\in\tilde{\tilde{\DD}}^{\{M_p\}}_{L^{\infty}_{\eta}}|\, p_{g,(r_p)}(\varphi)\leq \varepsilon\}$) in $\tilde{\tilde{\DD}}^*_{L^{\infty}_{\eta}}$. For these $g$ and $r$ (resp. $(r_p)$), define $\tilde{B}$ as in the statement of Lemma \ref{preinDL}. The quoted lemma states that $\tilde{B}$ is precompact in $\DD'^*_{L^1_{\eta}}$ and hence so is its convex circled hull $\tilde{\tilde{B}}$. Denoting by $\tilde{\tilde{B}}^{\circ}$ the polar of $\tilde{\tilde{B}}$ with respect to the duality $\langle \DD'^*_{L^1_{\eta}}, \DD^*_{L^{\infty}_{\eta}}\rangle$, one easily verifies that $\varepsilon \tilde{\tilde{B}}^{\circ}\subseteq V$. Thus the topology of $\DD^*_{L^{\infty}_{\eta},c}$ is stronger than the topology of $\tilde{\tilde{\DD}}^*_{L^{\infty}_{\eta}}$. Conversely, let $B^{\circ}$ be the polar of a convex circled precompact subset $B$ of $\DD'^*_{L^1_{\eta}}$.
Proposition \ref{corprecom} implies that there exist a precompact
subset $B_1$ of $L^1_{\eta}$ and an ultradifferential operator
$P(D)$ of class $*$ such that each $f\in B$ can be represented by
$f=P(D)F$ for some $F\in B_1$. Let $P(D)=\sum_{\alpha}c_{\alpha}
D^{\alpha}$. There exist $C,r>0$ (resp. $(r_p)\in\mathfrak{R}$ and
$C>0$), such that $|c_{\alpha}|\leq C r^{|\alpha|}/M_{\alpha}$
(resp. $|c_{\alpha}|\leq C/(M_{\alpha}R_{\alpha})$, where
$R_{\alpha}=\prod_{j=1}^{|\alpha|}r_j$). Since $B_1$ is precompact
in $L^1_{\eta}$, we conclude that for each $n\in\ZZ_+$ there
exists $k_n\geq 1$ such that $\sup_{F\in B_1}\int_{|x|\geq k_n}
|F(x)|\eta(x)dx\leq 2^{-n}$. Without losing generality we can
assume that $k_{n+1}\geq k_n+1$, $\forall n\in\ZZ_+$. Define
$g_0:[0,\infty)\rightarrow (0,\infty)$ by $g_0(k_n)=1/n$ and
linearly on $(k_n,k_{n+1})$ to be continuous. Furthermore, define
$g_0(\rho)=1$ on $[0,k_1)$. Set $g(x)=g_0(|x|)$. Clearly
$g\in\mathfrak{C}_0$. For $n\in\ZZ_+$ denote $K_n=\{x\in\RR^d|\,
|x|\geq k_n\}$ and set $K_0=\RR^d$. Set $t=2r$ (resp.
$(t_p)=(r_p/2)$ and $T_p=\prod_{j=1}^p t_j$). Then, for $\psi\in
\tilde{\tilde{\DD}}^*_{L^{\infty}_{\eta}}$ and $f\in B$, in the
Roumieu case, we have
\beqs
|\langle \psi,f\rangle|&\leq& \int_{\RR^d}|P(-D)\psi(x)||F(x)|dx\\
&\leq& C \sum_{\alpha}2^{-|\alpha|} \int_{\RR^d}\frac{g(x)\left|D^{\alpha}\psi(x)\right|}{\eta(x)M_{\alpha}T_{\alpha}}\cdot \frac{|F(x)|\eta(x)}{g(x)}dx\\
&\leq& C_1 p_{g,(t_p)}(\psi)\sum_{n=0}^{\infty}\int_{K_n\backslash K_{n+1}}\frac{|F(x)|\eta(x)}{g_0(|x|)}dx\\
&\leq& C_1p_{g,(t_p)}(\psi)\left(\|F\|_{L^1_{\eta}}+\sum_{n=1}^{\infty}\frac{n+1}{2^n}\right)\leq C_2p_{g,(t_p)}(\psi)
\eeqs
and similarly, in the Beurling case, $\sup_{f\in B}|\langle \psi,f\rangle|\leq C_2p_{g,t}(\psi)$. Thus, by defining
\beqs
V=\left\{\varphi\in\tilde{\tilde{\DD}}^{(M_p)}_{L^{\infty}_{\eta}}\big|\, p_{g,t}(\varphi)\leq C_2^{-1}\right\}\,\, \left(\mbox{resp.}\,\, V=\left\{\varphi\in\tilde{\tilde{\DD}}^{\{M_p\}}_{L^{\infty}_{\eta}}\big|\, p_{g,(t_p)}(\varphi)\leq C_2^{-1}\right\}\right),
\eeqs
we have $V\subseteq B^{\circ}$.
\end{proof}

Since $\tilde{\tilde{\DD}}^*_{L^{\infty}_{\eta}}$ is complete l.c.s., the above proposition implies that $\DD^*_{L^{\infty}_{\eta},c}$ is complete l.c.s. and its topology is generated by the system of seminorms $p_{g,r}$, $g\in\mathfrak{C}_0$, $r>0$, resp. $p_{g,(r_p)}$, $g\in\mathfrak{C}_0$, $(r_p)\in\mathfrak{R}$. Now, it is easy to verify that $\SSS^*_{\dagger}(\RR^d)\hookrightarrow \DD^*_{L^{\infty}_{\eta},c}$. In fact, one can prove by a similar technique as in the proof of $ii)$ of Lemma \ref{appincl} that given $\psi\in\DD^*_{L^{\infty}_{\eta},c}$, $\varphi_n\psi\in\SSS^*_{\dagger}(\RR^d)$ and $\varphi_n\psi\rightarrow \psi$ in $\DD^*_{L^{\infty}_{\eta},c}$, where $\varphi_n\in\SSS^*_{\dagger}(\RR^d)$, $n\in\ZZ_+$, is the sequence of $ii)$ of Lemma \ref{appincl}. Thus, $\dot{\mathcal{B}}^*_{\eta}\hookrightarrow \DD^*_{L^{\infty}_{\eta},c}$ and denoting by $(\DD^*_{L^{\infty}_{\eta},c})'_b$ the strong dual of $\DD^*_{L^{\infty}_{\eta},c}$, we have the continuous inclusion $(\DD^*_{L^{\infty}_{\eta},c})'_b\rightarrow \DD'^*_{L^1_{\eta}}$. Moreover, we have the following

\begin{proposition}\label{dualofddd}
The spaces $(\DD^*_{L^{\infty}_{\eta},c})'_b$ and $\DD'^*_{L^1_{\eta}}$ are isomorphic as l.c.s.
\end{proposition}

\begin{proof} If $X$ is a l.c.s. with $X'$ being its dual, the topology $c$ of compact convex circled convergence on $X'$ is clearly stronger than the weak topology. Moreover, every compact convex circled subset of $X$ is clearly weakly compact, thus $X'_c$ has weaker topology than the Mackey topology on $X'$. Hence $(X'_c)'$ is algebraically isomorphic to $X$. Considering the duality $\langle \DD'^*_{L^1_{\eta}},\DD^*_{L^{\infty}_{\eta},c}\rangle$, we obtain that the dual of $\DD^*_{L^{\infty}_{\eta},c}$ is algebraically isomorphic to $\DD'^*_{L^1_{\eta}}$. The topology of $\DD'^*_{L^1_{\eta}}$ is the topology of uniform convergence on all equicontinuous subsets of $\DD^*_{L^{\infty}_{\eta},c}$. But, $\DD'^*_{L^1_{\eta}}$ is a complete barreled l.c.s. (in fact $\DD'^{\{M_p\}}_{L^1_{\eta}}$ is an $(F)$-space and $\DD'^{(M_p)}_{L^1_{\eta}}$ is barreled by Theorem \ref{theorBB}) hence its topology is in fact the topology of uniform convergence on all strongly bounded subsets of $(\DD'^*_{L^1_{\eta}})'_b=\DD^*_{L^{\infty}_{\eta}}$. But Lemma \ref{bsubddd} implies that the bounded subsets of $\DD^*_{L^{\infty}_{\eta}}$ and $\DD^*_{L^{\infty}_{\eta},c}$ are the same, hence $(\DD^*_{L^{\infty}_{\eta},c})'_b$ and $\DD'^*_{L^1_{\eta}}$ are isomorphic as l.c.s.
\end{proof}

\begin{remark} By using similar technique as in the proof of Lemma \ref{precompactscl} one can prove that every bounded subset of $\DD^*_{L^{\infty}_{\eta},c}$ is precompact and since the latter space is complete, also relatively compact. Hence $\DD^*_{L^{\infty},c}$ furnishes an elegant example of complete semi-Montel space which is not (infra)barreled since it is not reflexive (cf. Theorem \ref{theorBB} and Proposition \ref{dualofddd}).
\end{remark}

We need the following technical lemma whose proof is simple and we therefore omit it.

\begin{lemma}\label{conmulsdd}
The multiplication $(\varphi,\psi)\rightarrow \varphi\psi$ is continuous bilinear mapping in the following cases $\cdot: \DD^*_{L^{\infty}_{\eta}}(\RR^d)\times \SSS^*_{\dagger}(\RR^d)\rightarrow \SSS^*_{\dagger}(\RR^d)$, $\cdot: \DD^*_{L^{\infty}_{\eta},c}(\RR^d)\times \SSS^*_{\dagger}(\RR^d)\rightarrow \SSS^*_{\dagger}(\RR^d)$, $\cdot:\DD^*_{L^{\infty}}(\RR^d)\times\DD^*_{L^{\infty}_{\eta}}(\RR^d)\rightarrow \DD^*_{L^{\infty}_{\eta}}(\RR^d)$ and $\cdot:\DD^*_{L^{\infty},c}(\RR^d)\times\DD^*_{L^{\infty}_{\eta},c}(\RR^d)\rightarrow \DD^*_{L^{\infty}_{\eta},c}(\RR^d)$.
\end{lemma}

For $\varphi\in C(\RR^d)$ we define $\varphi^{\Delta}\in C(\RR^{2d})$ by $\varphi^{\Delta}(x,y)=\varphi(x+y)$. To consider the problem on the existence of convolution of two ultradistributions with restrict our attention to the case $\eta(x)=1$.

\begin{definition}\label{defco}
Let $f_1,f_2\in\SSS'^*_{\dagger}(\RR^d)$. We say that the convolution of $f_1$ and $f_2$ exists if for each $\varphi\in\SSS^*_{\dagger}(\RR^d)$, $(f_1\otimes f_2)\varphi^{\Delta}\in\ \DD'^*_{L^1}(\RR^{2d})$ and we define their convolution by
\beqs
\langle f_1*f_2,\varphi\rangle={}_{\DD'^*_{L^1}(\RR^{2d})}\langle (f_1\otimes f_2) \varphi^{\Delta},1_{x,y}\rangle_{\DD^*_{L^{\infty},c}(\RR^{2d})},\,\, \forall \varphi\in\SSS^*_{\dagger}(\RR^d),
\eeqs
where $1_{x,y}$ is the function that is identically equal to $1$; from now on denoted only by $1$.
\end{definition}

A priori it is not clear that if $f_1$ and $f_2$ are as in this definition, the convolution $f_1*f_2$ is a well defined element of $\SSS'^*_{\dagger}(\RR^d)$. To prove this, consider the linear mapping $F:\SSS^*_{\dagger}(\RR^d)\rightarrow \DD'^*_{L^1}(\RR^{2d})$, $F(\varphi)=(f_1\otimes f_2)\varphi^{\Delta}$. If we consider $F$ as a linear mapping from $\SSS^*_{\dagger}(\RR^d)$ to $\SSS'^*_{\dagger}(\RR^{2d})$ it is clearly continuous, hence it has closed graph. But since for each $\varphi\in\SSS^*_{\dagger}(\RR^d)$, $(f_1\otimes f_2)\varphi^{\Delta}\in \DD'^*_{L^1}(\RR^{2d})$ its graph is closed in $\SSS^*_{\dagger}(\RR^d)\times \DD'^*_{L^1}(\RR^{2d})$. Now, $\SSS^*_{\dagger}(\RR^d)$ is ultrabornological (since it is bornological and complete) and $\DD'^*_{L^1}(\RR^{2d})$ is a webbed space of De Wilde since $\DD'^{(M_p)}_{L^1}(\RR^{2d})$ is the strong dual of an $(F)$-space (cf. \cite[Theorem 11, p. 64]{kothe2}) and since $\DD'^{\{M_p\}}_{L^1}(\RR^{2d})$ is an $(F)$-space (cf. \cite[Theorem 4, p. 55]{kothe2}). The closed graph theorem of De Wilde \cite[Theorem 2, p. 57]{kothe2} implies that $F$ is continuous. Now, observe that $f_1*f_2$ is the composition of the two continuous mappings $F:\SSS^*_{\dagger}(\RR^d)\rightarrow \DD'^*_{L^1}(\RR^{2d})$ and $\langle \cdot,1\rangle:\DD'^*_{L^1}(\RR^{2d})\rightarrow \CC$ (cf. Theorem \ref{theorBB}).

\begin{theorem}\label{convolution}
Let $f_1,f_2\in\SSS'^*_{\dagger}(\RR^d)$. The following statements are equivalent
\begin{itemize}
\item[$i)$] the convolution of $f_1$ and $f_2$ exists;
\item[$ii)$] for all $\varphi\in\SSS^*_{\dagger}(\RR^d)$, $(\varphi*\check{f}_1)f_2\in\DD'^*_{L^1}(\RR^d)$;
\item[$iii)$] for all $\varphi\in\SSS^*_{\dagger}(\RR^d)$, $(\varphi*\check{f}_2)f_1\in\DD'^*_{L^1}(\RR^d)$;
\item[$iv)$] for all $\varphi,\psi\in\SSS^*_{\dagger}(\RR^d)$, $(\varphi*\check{f}_1)(\psi*f_2)\in L^1(\RR^d)$.
\end{itemize}
\end{theorem}

\begin{proof} $i) \Rightarrow ii)$. Let $\theta,\psi,\chi\in\SSS^*_{\dagger}(\RR^d)$. Clearly $\psi(x+y)(\check{\chi}*\varphi)(y)\in\SSS^*_{\dagger}(\RR^{2d})$. We have
\beqs
\left\langle \left((\psi*\check{f}_1)f_2\right)*\chi,\theta\right\rangle= \langle f_1\otimes f_2, \psi(x+y)(\check{\chi}*\theta)(y)\rangle.
\eeqs
Let $\varphi_n\in\SSS^*_{\dagger}(\RR^d)$, $n\in\ZZ_+$, be as in $ii)$ of Lemma \ref{appincl}. One easily verifies that $\psi(x+y)\varphi_n(x)(\check{\chi}*\theta)(y)\rightarrow \psi(x+y)(\check{\chi}*\theta)(y)$ in $\SSS^*_{\dagger}(\RR^{2d})$ (cf. the proof of $ii)$ of Lemma \ref{appincl}) and $\varphi_n(x)(\check{\chi}*\theta)(y)\rightarrow 1_x\otimes (\check{\chi}*\theta)(y)$ in $\DD^*_{L^{\infty},c}(\RR^{2d})$. Hence\\
\\
$\langle f_1\otimes f_2, \psi(x+y)(\check{\chi}*\theta)(y)\rangle$
\beqs
&=&\lim_{n\rightarrow \infty} \langle f_1\otimes f_2, \psi(x+y)\varphi_n(x)(\check{\chi}*\theta)(y)\rangle=\lim_{n\rightarrow \infty} \langle (f_1\otimes f_2)\psi^{\Delta}, \varphi_n(x)(\check{\chi}*\theta)(y)\rangle\\
&=&\langle (f_1\otimes f_2)\psi^{\Delta}, 1_x\otimes(\check{\chi}*\theta)(y)\rangle,
\eeqs
where the last equality follows by $i)$ and Proposition \ref{dualofddd}. Thus $\left\langle \left((\psi*\check{f}_1)f_2\right)*\chi,\theta\right\rangle=\langle (f_1\otimes f_2)\psi^{\Delta}, 1_x\otimes(\check{\chi}*\theta)(y)\rangle$. Again, $i)$ and Proposition \ref{dualofddd} imply that there exist $C>0$, $g\in\mathfrak{C}_0$ and $r>0$, resp. $(r_p)\in\mathfrak{R}$, such that
\beqs
\left|\langle (f_1\otimes f_2)\psi^{\Delta}, 1_x\otimes(\check{\chi}*\theta)(y)\rangle\right|&\leq& C p_{g,r}(1_x\otimes(\check{\chi}*\theta)(y))\leq C_1 \|\theta\|_{L^{\infty}},\,\,
\mbox{resp.}\\
\left|\langle (f_1\otimes f_2)\psi^{\Delta}, 1_x\otimes(\check{\chi}*\theta)(y)\rangle\right|&\leq& C p_{g,(r_p)}(1_x\otimes(\check{\chi}*\theta)(y))\leq C_1 \|\theta\|_{L^{\infty}}.
\eeqs
Since $\SSS^*_{\dagger}(\RR^d)$ is dense in $C_0(\RR^d)$, Theorem \ref{karak} implies that $(\psi*\check{f}_1)f_2\in \DD'^*_{L^1}(\RR^d)$. The proof of $i) \Rightarrow iii)$ is analogous.\\
\indent $ii)\Rightarrow iv)$. By similar arguments as in the discussion after Definition \ref{defco}, using the De Wilde closed graph theorem, one can prove that the mapping $\varphi\mapsto (\varphi*\check{f}_1)f_2$, $\SSS^*_{\dagger}(\RR^d)\rightarrow \DD'^*_{L^1}(\RR^d)$, is continuous. Hence, for fixed $\chi\in\dot{\mathcal{B}}^*(\RR^d)$, the mapping $\varphi\mapsto \left\langle (\varphi*\check{f}_1)f_2,\chi\right\rangle$, $\SSS^*_{\dagger}(\RR^d)\rightarrow \CC$, is continuous. On the other hand, for fixed $\varphi\in\SSS^*_{\dagger}(\RR^d)$, the mapping $\chi\mapsto \left\langle(\varphi*\check{f}_1)f_2, \chi\right\rangle$, $\dot{\mathcal{B}}^*(\RR^d)\rightarrow \CC$, is continuous by $ii)$. Thus the bilinear mapping $G:\SSS^*_{\dagger}(\RR^d)\times \dot{\mathcal{B}}^*(\RR^d)\rightarrow \CC$, $G(\varphi,\chi)=\left\langle(\varphi*\check{f}_1)f_2, \chi\right\rangle$, is separately continuous. Since $\SSS^{(M_p)}_{(A_p)}(\RR^d)$ and $\dot{\mathcal{B}}^{(M_p)}(\RR^d)$ are $(F)$-spaces, resp. $\SSS^{\{M_p\}}_{\{A_p\}}(\RR^d)$ and $\dot{\mathcal{B}}^{\{M_p\}}(\RR^d)$ are barreled $(DF)$-spaces, $G$ is continuous. Thus $G$ defines a continuous map $\tilde{G}:\SSS^*_{\dagger}(\RR^d)\hat{\otimes} \dot{\mathcal{B}}^*(\RR^d)\rightarrow \CC$ where the topology on the tensor product is $\pi=\epsilon$ ($\SSS^*_{\dagger}(\RR^d)$ is nuclear).\\
\indent Consider the linear transformation on $\RR^{2d}$, $\Theta(x,y)=(x+y,y)$. It induces topological isomorphism $\tilde{\Theta}:\theta\mapsto \theta\circ\Theta$ on $\SSS^*_{\dagger}(\RR^{2d})$ and on $\dot{\mathcal{B}}^*(\RR^{2d})$. Hence ${}^t\tilde{\Theta}$ is topological isomorphism on $\SSS'^*_{\dagger}(\RR^{2d})$ and on $\DD'^*_{L^1}(\RR^{2d})$.\\
\indent For $\varphi,\chi\in\SSS^*_{\dagger}(\RR^d)$, we have
\beqs
\tilde{G}(\varphi\otimes\chi)=\langle (\varphi*\check{f_1})f_2, \chi\rangle=\langle f_1\otimes f_2, \varphi(x+y)\chi(y)\rangle=\langle {}^t\tilde{\Theta}(f_1\otimes f_2), \varphi\otimes\chi\rangle.
\eeqs
Since $\SSS^*_{\dagger}(\RR^d)\otimes\SSS^*_{\dagger}(\RR^d)$ is dense in $\SSS^*_{\dagger}(\RR^{2d})$, ${}^t\tilde{\Theta}(f_1\otimes f_2)=\tilde{G}\in \left(\SSS^*_{\dagger}(\RR^d)\hat{\otimes} \dot{\mathcal{B}}^*(\RR^d)\right)'$. There exist $C>0$ and equicontinuous subset $H'$ of $\SSS'^*_{\dagger}(\RR^d)$ and equicontinuous subset $K'$ of $\DD'^*_{L^1}(\RR^d)$ such that
\beq\label{kkkrr15}
|\tilde{G}(\theta)|\leq C\sup_{u\in H'}\sup_{v\in K'}|\langle u\otimes v,\theta\rangle|
\eeq
for all $\theta\in\SSS^*_{\dagger}(\RR^d)\otimes \dot{\mathcal{B}}^*(\RR^d)$. It is easy to verify that the set $W=\{\Phi\in\SSS'^*_{\dagger}(\RR^{2d})|\, \Phi=u\otimes v,\, u\in H',\, v\in K'\}$ is equicontinuous subset of $\SSS'^*_{\dagger}(\RR^{2d})$. Since $\SSS^*_{\dagger}(\RR^{2d})$ is continuously injected into $\SSS^*_{\dagger}(\RR^d)\hat{\otimes} \dot{\mathcal{B}}^*(\RR^d)$, (\ref{kkkrr15}) holds for all $\theta\in\SSS^*_{\dagger}(\RR^{2d})$. As $H'$ and $K'$ are equicontinuous subsets of $\SSS'^*_{\dagger}(\RR^d)$ and $\DD'^*_{L^1}(\RR^d)$ respectively, there exist $C_1>0$ and $r>0$ (resp. $(r_p)\in\mathfrak{R}$) such that for all $\varphi\in\SSS^*_{\dagger}(\RR^d)$ and $\chi\in\dot{\mathcal{B}}^*(\RR^d)$
\beqs
&{}&\sup_{u\in H'}|\langle u,\varphi\rangle|\leq C_1 \sigma_r(\varphi)\,\, \mbox{and}\,\, \sup_{v\in K'}|\langle v,\chi\rangle|\leq C_1 \|\chi\|_{L^{\infty},r}\\
&{}&\left(\mbox{resp.}\,\,\sup_{u\in H'}|\langle u,\varphi\rangle|\leq C_1 \sigma_{(r_p)}(\varphi)\,\, \mbox{and}\,\, \sup_{v\in K'}|\langle v,\chi\rangle|\leq C_1 \|\chi\|_{L^{\infty},(r_p)}\right).
\eeqs
We continue the proof in the Roumieu case, since the Beurling case is similar. By Lemma \ref{nwseq}, we can assume that $(r_p)$ is such that $R_{p+q}\leq 2^{p+q} R_pR_q$ for all $p,q\in\ZZ_+$. Let $(r'_p)=(r_p/(2H))$. For $\theta\in \SSS^{\{M_p\}}_{\{A_p\}}(\RR^{2d})$ and $u\in H'$ and $v\in K'$ we have\\
\\
$|\langle u(x)\otimes v(y),\theta(x,y)\rangle|$
\beqs
&=&|\langle u(x),\langle v(y),\theta(x,y)\rangle\rangle|\leq C_1\sup_{\alpha\in\NN^d}\sup_{x\in\RR^d} \frac{\left|\langle v(y),D^{\alpha}_x \theta(x,y)\rangle\right|e^{B_{r_p}(|x|)}}{M_{\alpha}R_{\alpha}}\\
&\leq& C_1^2\sup_{\alpha,\beta\in\NN^d}\sup_{x,y\in\RR^d} \frac{\left|D^{\alpha}_x D^{\beta}_y \theta(x,y)\right|e^{B_{r_p}(|x|)}}{M_{\alpha}M_{\beta}R_{\alpha}R_{\beta}}\\
&\leq& C_2\sup_{\alpha,\beta\in\NN^d}\sup_{x,y\in\RR^d} \frac{\left|D^{\alpha}_x D^{\beta}_y \theta(x,y)\right|e^{B_{r_p}(|x|)}}{M_{\alpha+\beta}R'_{\alpha+\beta}}.
\eeqs
Let $\varphi,\psi\in\SSS^{\{M_p\}}_{\{A_p\}}(\RR^d)$ be fixed. For $\chi\in\SSS^{\{M_p\}}_{\{A_p\}}(\RR^d)$, clearly
\beqs
\chi_{\varphi,\psi}(x,y)=\int_{\RR^d}\varphi(x-y+t)\psi(t-y)\chi(t)dt\in\SSS^{\{M_p\}}_{\{A_p\}}(\RR^{2d}).
\eeqs
Moreover, observe that
\beqs
e^{B_{r_p}(|x|)}\leq 2e^{B_{r_p}(2|x-y+t|)}e^{B_{r_p}(2|y-t|)}\leq 2e^{B_{r'_p}(|x-y+t|)}e^{B_{r'_p}(|t-y|)}
\eeqs
and thus
\beqs
\sup_{\alpha,\beta}\sup_{x,y}\frac{e^{B_{r_p}(|x|)}\left|D^{\alpha}_x D^{\beta}_y\chi_{\varphi,\psi}(x,y)\right|}{M_{\alpha+\beta}R'_{\alpha+\beta}}\leq C_3\|\chi\|_{L^{\infty}}.
\eeqs
Observe the mapping $F_{\varphi,\psi}:\SSS^{\{M_p\}}_{\{A_p\}}(\RR^d)\rightarrow \CC$, $F_{\varphi,\psi}(\chi)=\tilde{G}(\chi_{\varphi,\psi})$. By the above estimates, we have $|F_{\varphi,\psi}(\chi)|\leq C_4 \|\chi\|_{L^{\infty}}$. Since $\SSS^{\{M_p\}}_{\{A_p\}}(\RR^d)$ is dense in $C_0(\RR^d)$ this mapping can be continuously extended to $\tilde{F}_{\varphi,\psi}:C_0(\RR^d)\rightarrow\CC$. Observe that, for $\chi\in\SSS^{\{M_p\}}_{\{A_p\}}(\RR^d)$,
\beqs
\tilde{F}_{\varphi,\psi}(\chi)&=&\langle {}^t\tilde{\Theta}(f_1\otimes f_2), \chi_{\varphi,\psi}\rangle=\langle f_1\otimes f_2\otimes 1_t, \varphi(x+t)\psi(t-y)\chi(t)\rangle\\
&=&\langle (\varphi*\check{f}_1)(\psi*f_2),\chi\rangle.
\eeqs
Thus $(\varphi*\check{f}_1)(\psi*f_2)\in\mathcal{M}^1(\RR^d)$. But $(\varphi*\check{f}_1)(\psi*f_2)$ is a continuous function, hence $(\varphi*\check{f}_1)(\psi*f_2)\in L^1(\RR^d)$. The proof of $iii)\Rightarrow iv)$ is similar.\\
\indent $iv) \Rightarrow i)$. By similar arguments as in the discussion after Definition \ref{defco}, using De Wilde closed graph theorem, one verifies that the bilinear mapping $G:\SSS^*_{\dagger}(\RR^d)\times \SSS^*_{\dagger}(\RR^d)\rightarrow L^1(\RR^d)$, $G(\varphi,\psi)=(\varphi*\check{f}_1)(\psi*f_2)$, is separately continuous, hence continuous since $\SSS^{(M_p)}_{(A_p)}(\RR^d)$ is an $(F)$-space, resp. $\SSS^{\{M_p\}}_{\{A_p\}}(\RR^d)$ is a barreled $(DF)$-space. Keeping in mind $\SSS^*_{\dagger}(\RR^{2d})=\SSS^*_{\dagger}(\RR^d)\hat{\otimes}\SSS^*_{\dagger}(\RR^d)$ where the topology on the tensor product is $\pi=\epsilon$ ($\SSS^*_{\dagger}(\RR^d)$ is nuclear), $G$ extends to a continuous mapping $\tilde{G}:\SSS^*_{\dagger}(\RR^{2d})\rightarrow L^1(\RR^d)$. Denote by $V$ the continuous mapping
\beqs
\SSS^*_{\dagger}(\RR^{2d})\times\dot{\mathcal{B}}^*(\RR^d)\xrightarrow{\mathrm{Id}\times \mathrm{Id}} \SSS^*_{\dagger}(\RR^{2d})\times C_0(\RR^d)\xrightarrow{\tilde{G}\times\mathrm{Id}} L^1(\RR^d)\times C_0(\RR^d)\xrightarrow{\langle \cdot,\cdot\rangle} \CC,
\eeqs
where the last bilinear mapping is the duality between $\mathcal{M}^1(\RR^d)$ and $C_0(\RR^d)$ ($L^1(\RR^d)$ is closed subspace of $\mathcal{M}^1(\RR^d)$). $V$ extends to a continuous mapping $\tilde{V}:\SSS^*_{\dagger}(\RR^{2d})\hat{\otimes}\dot{\mathcal{B}}^*(\RR^d)\rightarrow \CC$, where the topology on the tensor product is $\pi=\epsilon$ ($\SSS^*_{\dagger}(\RR^{2d})$ is nuclear). Fix $\theta_1,\theta_2\in\SSS^*_{\dagger}(\RR^d)$. The mapping $F_{\theta_1,\theta_2}:\dot{\mathcal{B}}^*(\RR^{2d})\rightarrow \SSS^*_{\dagger}(\RR^{2d})$, $F_{\theta_1,\theta_2}(\chi)=(\theta_1\otimes\theta_2)\chi$, is continuous, hence so is the mapping $F_{\theta_1,\theta_2}\otimes \mathrm{Id}:\dot{\mathcal{B}}^*(\RR^{2d})\otimes_{\epsilon}\dot{\mathcal{B}}^*(\RR^d)\rightarrow \SSS^*_{\dagger}(\RR^{2d})\otimes_{\epsilon}\dot{\mathcal{B}}^*(\RR^d)$. By Proposition \ref{epsdd}, $F_{\theta_1,\theta_2}\otimes \mathrm{Id}$ extends to a continuous mapping $F_{\theta_1,\theta_2}\hat{\otimes} \mathrm{Id}:\dot{\mathcal{B}}^*(\RR^{3d})\rightarrow \SSS^*_{\dagger}(\RR^{2d})\hat{\otimes}\dot{\mathcal{B}}^*(\RR^d)$. Denote by $\tilde{U}_{\theta_1,\theta_2}$ the continuous mapping $\tilde{V}\circ (F_{\theta_1,\theta_2}\hat{\otimes} \mathrm{Id}):\dot{\mathcal{B}}^*(\RR^{3d})\rightarrow \CC$. For $\varphi,\psi,\chi\in\SSS^*_{\dagger}(\RR^d)$, we have
\beqs
\tilde{U}_{\theta_1,\theta_2}(\varphi\otimes\psi\otimes\chi)=\langle f_1(x)\otimes f_2(y)\otimes 1_t, \theta_1(x+t)\theta_2(t-y) \varphi(x+t)\psi(t-y) \chi(t)\rangle.
\eeqs
Since $\SSS^*_{\dagger}(\RR^d)\otimes\SSS^*_{\dagger}(\RR^d)\otimes\SSS^*_{\dagger}(\RR^d)$ is dense in $\SSS^*_{\dagger}(\RR^{3d})$, for $w\in\SSS^*_{\dagger}(\RR^{3d})$ there exists a net $\{\tilde{w}_{\lambda}\}_{\lambda\in\Lambda}\subseteq \SSS^*_{\dagger}(\RR^d)\otimes\SSS^*_{\dagger}(\RR^d)\otimes\SSS^*_{\dagger}(\RR^d)$ which converges to $w$ in $\SSS^*_{\dagger}(\RR^{3d})$. Thus, for $w\in\SSS^*_{\dagger}(\RR^{3d})$, we have
\beqs
\tilde{U}_{\theta_1,\theta_2}(w)=\langle f_1(x)\otimes f_2(y)\otimes 1_t, \theta_1(x+t)\theta_2(t-y) w(x+t,t-y,t)\rangle
\eeqs
Let $\varphi_n\in\SSS^*_{\dagger}(\RR^d)$, $n\in\ZZ_+$, be as in $ii)$ of Lemma \ref{appincl} and for $m,n\in\ZZ_+$ and $\chi\in\SSS^*_{\dagger}(\RR^{2d})$ define $w_{n,m}(x,y,t)=\varphi_n(x-y)\chi(x-t,t-y)\varphi_m(x)\in\SSS^*_{\dagger}(\RR^{3d})$. Then
\beqs
\tilde{U}_{\theta_1,\theta_2}(w_{n,m})&=&\langle f_1(x)\otimes f_2(y)\otimes 1_t, \theta_1(x+t)\theta_2(t-y) \varphi_n(x+y)\varphi_m(x+t)\chi(x,y)\rangle\\
&=&\langle f_1(x)\otimes f_2(y), (\theta_1\varphi_m)*\check{\theta_2}(x+y)\varphi_n(x+y)\chi(x,y)\rangle.
\eeqs
One easily verifies that $\theta_1\varphi_m\rightarrow \theta_1$ in $\SSS^*_{\dagger}(\RR^d)$, hence $\left((\theta_1\varphi_m)*\check{\theta_2}\right)^{\Delta}\rightarrow (\theta_1*\check{\theta_2})^{\Delta}$ in $\DD^*_{L^{\infty}}(\RR^{2d})$. Lemma \ref{conmulsdd} implies that for each fixed $n\in\ZZ_+$,
\beqs
&{}&(\theta_1\varphi_m)*\check{\theta_2}(x+y)\varphi_n(x+y)\chi(x,y)\rightarrow \theta_1*\check{\theta_2}(x+y)\varphi_n(x+y)\chi(x,y),\,\, \mbox{in}\,\, \SSS^*_{\dagger}(\RR^{2d}),\\
&{}&w_{n,m}(x,y,t)\rightarrow \varphi_n(x-y)\chi(x-t,t-y)=w_n(x,y,t),\,\, \mbox{in}\,\, \DD^*_{L^{\infty},c}(\RR^{2d}),
\eeqs
as $m\rightarrow \infty$. If we let $m\rightarrow \infty$ in the above equality, by using Proposition \ref{dualofddd}, we have
\beqs
\tilde{U}_{\theta_1,\theta_2}(w_n)=\langle f_1(x)\otimes f_2(y), \theta_1*\check{\theta_2}(x+y)\varphi_n(x+y)\chi(x,y)\rangle,\,\, \forall n\in\ZZ_+.
\eeqs
One easily verifies that $\varphi_n(x+y)\rightarrow 1_{x,y}$ and $\varphi_n(x-y)\rightarrow 1_{x,y}$ in $\DD^*_{L^{\infty},c}(\RR^{2d})$. Set $w(x,y,t)=\chi(x-t,t-y)\in \DD^*_{L^{\infty},c}(\RR^{3d})$. If we let $n\rightarrow\infty$ in this equality, Lemma \ref{conmulsdd} and Proposition \ref{dualofddd} imply
\beqs
\tilde{U}_{\theta_1,\theta_2}(w)=\langle f_1(x)\otimes f_2(y), \theta_1*\check{\theta_2}(x+y)\chi(x,y)\rangle=\langle (f_1\otimes f_2) (\theta_1*\check{\theta_2})^{\Delta}, \chi\rangle.
\eeqs
Since $\tilde{U}_{\theta_1,\theta_2}\in\DD'^*_{L^1}(\RR^{3d})$, by Proposition \ref{dualofddd} there exist $C>0$, $g\in\mathfrak{C}_0$ and $r>0$ (resp. $(r_p)\in\mathfrak{R}$) such that
\beqs
&{}&\left|\tilde{U}_{\theta_1,\theta_2}(w)\right|\leq Cp_{g,r}(w)\leq C_1 \|\chi\|_{L^{\infty}(\RR^{2d}),2r}\\
&{}&\left(\mbox{resp.}\,\,\left|\tilde{U}_{\theta_1,\theta_2}(w)\right|\leq Cp_{g,(r_p)}(w)\leq C_1 \|\chi\|_{L^{\infty}(\RR^{2d}),(r_p/2)}\right).
\eeqs
We obtain that for each $\theta_1,\theta_2\in\SSS^*_{\dagger}(\RR^d)$, $(f_1\otimes f_2) (\theta_1*\theta_2)^{\Delta}\in\DD'^*_{L^1}(\RR^{2d})$. By using De Wilde closed graph theorem once again, one proves that the bilinear mapping $(\theta_1,\theta_2)\mapsto (f_1\otimes f_2) (\theta_1*\theta_2)^{\Delta}$, $\SSS^*_{\dagger}(\RR^d)\times\SSS^*_{\dagger}(\RR^d)\rightarrow \DD'^*_{L^1}(\RR^{2d})$ is continuous. Fix $\chi\in\SSS^*_{\dagger}(\RR^{2d})$. Since $\DD'^*_{L^1}(\RR^{2d})$ is bornological (cf. Theorem \ref{theorBB} for the Beurling case), Theorem \ref{karak} implies that the mapping $f\mapsto f*\chi$, $\DD'^*_{L^1}(\RR^{2d})\rightarrow L^1(\RR^{2d})$ is continuous. Hence the bilinear mapping $Q_{\chi}:\ \SSS^*_{\dagger}(\RR^d)\times\SSS^*_{\dagger}(\RR^d)\rightarrow L^1(\RR^{2d})$, $Q_{\chi}(\theta_1,\theta_2)=\left((f_1\otimes f_2) (\theta_1*\theta_2)^{\Delta}\right)*\chi$, is continuous. We continue the proof in the Roumieu case, as the Beurling case is similar. For $(r_p)\in\mathfrak{R}$ denote by $X_{(r_p)}$ the closure of $\SSS^*_{\dagger}(\RR^d)$ in $\SSS^{M_p,(r_p)}_{A_p,(r_p)}$. There exists $(r_p)\in\mathfrak{R}$ such that $Q_{\chi}$ extends to a continuous bilinear mapping $\tilde{Q}_{\chi}: X_{(r_p)}\times X_{(r_p)}\rightarrow L^1(\RR^{2d})$. Take $(l_p)\in\mathfrak{R}$ with $(l_p)\leq (r_p)$ such that $\SSS^{M_p,(l_p)}_{A_p,(l_p)}\subseteq X_{(r_p)}$ (for the construction of such $(l_p)$ see the proof of Proposition \ref{parametrix1}). For this $(l_p)$, Proposition \ref{parametrix} implies that there exist $u\in\SSS^{M_p,(l_p)}_{A_p,(l_p)}$ and $P(D)$ of class $\{M_p\}$ such that $P(D)u=\delta$. Moreover, by the way we choose $(l_p)$, there exist $\tilde{\theta}_n\in \SSS^*_{\dagger}(\RR^d)$, $n\in\ZZ_+$, such that $\tilde{\theta}_n\rightarrow u$ in $X_{(r_p)}$. Observe that for arbitrary but fixed $\theta\in\SSS^*_{\dagger}(\RR^d)$, we have
\beqs
\left((f_1\otimes f_2) (\theta*P(D)\tilde{\theta}_n)^{\Delta}\right)*\chi=Q_{\chi}(P(D)\theta,\tilde{\theta}_n),\,\, \forall n\in\ZZ_+.
\eeqs
By construction, the right hand side tends to $\tilde{Q}_{\chi}(P(D)\theta,u)$ in $L^1(\RR^{2d})$. For $\varphi\in\SSS^*_{\dagger}(\RR^{2d})$, for the left hand side we have
\beqs
\left\langle\left((f_1\otimes f_2) (\theta*P(D)\tilde{\theta}_n)^{\Delta}\right)*\chi,\varphi\right\rangle=\left\langle f_1\otimes f_2, (P(D)\theta*\tilde{\theta}_n)^{\Delta}\check{\chi}*\varphi\right\rangle.
\eeqs
One easily verifies that for each $\psi\in\SSS^*_{\dagger}(\RR^d)$, $(\psi*\tilde{\theta}_n)^{\Delta}\rightarrow (\psi*u)^{\Delta}$ in $\DD^*_{L^{\infty}}(\RR^{2d})$. Hence, Lemma \ref{conmulsdd} implies $(P(D)\theta*\tilde{\theta}_n)^{\Delta}\check{\chi}*\varphi\rightarrow (P(D)\theta*u)^{\Delta}\check{\chi}*\varphi$ in $\SSS^*_{\dagger}(\RR^{2d})$. Observe that $P(D)\theta*u=\theta*P(D)u=\theta$ in $\SSS'^*_{\dagger}(\RR^d)$. Clearly $P(D)\theta*u$ is $C^{\infty}$ function, hence the equality also holds pointwise. We obtain that
\beqs
\left((f_1\otimes f_2) (\theta*P(D)\tilde{\theta}_n)^{\Delta}\right)*\chi\rightarrow \left((f_1\otimes f_2) \theta^{\Delta}\right)*\chi\,\, \mbox{weakly in}\,\, \SSS'^*_{\dagger}(\RR^{2d}).
\eeqs
But the latter space is Montel, hence the convergence also holds in the strong topology. We obtain $\left((f_1\otimes f_2) \theta^{\Delta}\right)*\chi=\tilde{Q}_{\chi}(P(D)\theta,u)\in L^1(\RR^{2d})$. Since $\chi\in\SSS^*_{\dagger}(\RR^{2d})$ is arbitrary, Theorem \ref{karak} implies that $(f_1\otimes f_2) \theta^{\Delta}\in\DD'^*_{L^1}(\RR^{2d})$.
\end{proof}

\begin{remark} If the convolution of $f_1$ and $f_2$ exists, $ii)$ and $iii)$ of Theorem \ref{convolution} imply that $\langle f_1*f_2,\psi\rangle =\langle (\psi*\check{f}_1)f_2,1\rangle= \langle (\psi*\check{f}_2)f_1,1\rangle$. This can be proved by using the sequence $\{\varphi_n\}_{n\in\ZZ_+}\subseteq \SSS^*_{\dagger}(\RR^d)$ from $ii)$ of Lemma \ref{appincl} via a similar argument as in the proof of $i)\Rightarrow ii)$ of Theorem \ref{convolution}.
\end{remark}

As direct consequence of this remark and Theorem \ref{convolution}, we obtain the following useful corollary.

\begin{corollary}
Let the convolution of $f_1$ and $f_2$ exists. Then the convolution of $f_2$ and $f_1$ also exists and $f_1*f_2=f_2*f_1$. If $P(D)$ is ultradifferential operator of class $*$, then the convolution of $f_1$ and $P(D)f_2$ exists and the convolution of $P(D)f_1$ and $f_2$ also exists and $P(D)(f_1*f_2)=f_1*P(D)f_2=P(D)f_1*f_2$.
\end{corollary}

\end{document}